\documentclass[11pt]{amsart}
\usepackage[latin1]{inputenc}
\usepackage{amsmath,amsthm}
\usepackage{amssymb,amsfonts}
\usepackage{hyperref}
\usepackage{tikz}

\usepackage[shortlabels]{enumitem}
\usepackage{bm}
\usepackage{cite}
\usepackage{graphicx}
\usepackage{verbatim}
\usepackage{setspace}
\usepackage{color}

\usepackage[a4paper]{geometry}

\usepackage{stmaryrd}
\DeclareSymbolFont{bbold}{U}{bbold}{m}{n}
\DeclareSymbolFontAlphabet{\mathbbm}{bbold}

\title[]{The model-companionship spectrum of set theory, generic absoluteness, and the Continuum problem}
\author{Matteo Viale}
\thanks{The author acknowledge support from INDAM through GNSAGA and from the project:
\emph{PRIN 2017-2017NWTM8R
Mathematical Logic: models, sets, computability.
}}

\theoremstyle{plain}
	\newtheorem{Theorem}{Theorem}
	\newtheorem{Fact}{Fact}
	\newtheorem{Corollary}{Corollary}
	\newtheorem{Lemma}{Lemma}
	\newtheorem{Notation}{Notation}
	\newtheorem{Remark}{Remark}
	
	\newtheorem{theorem}{Theorem}[section]
	\newtheorem{proposition}[theorem]{Proposition}
	\newtheorem{lemma}[theorem]{Lemma}
	\newtheorem{corollary}[theorem]{Corollary}
	\newtheorem{fact}[theorem]{Fact}

	\newtheorem{claim}{Claim}
	\newtheorem{subclaim}{Subclaim}
\theoremstyle{definition}
	\newtheorem{Definition}{Definition}

	\newtheorem{definition}[theorem]{Definition}
	\newtheorem{notation}[theorem]{Notation}
	\newtheorem{example}[theorem]{Example}

\theoremstyle{remark}
	\newtheorem{remark}[theorem]{Remark}

\newcommand{\ZFC}{\ensuremath{\mathsf{ZFC}}}
\newcommand{\ZF}{\ensuremath{\mathsf{ZF}}}

\newcommand{\WFE}{\ensuremath{\mathsf{WFE}}}

\DeclareMathOperator{\dom}{dom}
\DeclareMathOperator{\ran}{ran}
\DeclareMathOperator{\crit}{crit}
\DeclareMathOperator{\otp}{otp}

\DeclareMathOperator{\Ult}{Ult}

\DeclareMathOperator{\Coll}{Coll}

\newcommand{\maxUB}{\ensuremath{
{\mathbf{MAX}(\mathsf{UB})}}}
\newcommand{\Pmax}{\ensuremath{\mathbb{P}_{\mathrm{max}}}}
\newcommand{\NS}{\ensuremath{\mathbf{NS}}} 
\newcommand{\stUB}{\ensuremath{(*)\text{-}\mathsf{UB}}}
\newcommand{\stA}{\ensuremath{(*)\text{-}\mathcal{A}}}

\newcommand{\bool}[1]{\mathsf{#1}}
\newcommand{\tow}[1]{\mathcal{#1}}

\newcommand{\pow}[1]{\mathcal{P}\left(#1\right)}

\newcommand{\qp}[1]{\left[ #1 \right]}

\newcommand{\ap}[1]{\langle #1 \rangle}
\newcommand{\bp}[1]{\left\lbrace #1 \right\rbrace}

\newcommand{\Cod}{\ensuremath{\text{{\rm Cod}}}}
\newcommand{\ST}{\ensuremath{\text{{\sf ST}}}}
\newcommand{\UB}{\ensuremath{\text{{\sf UB}}}}
\newcommand{\lUB}{\ensuremath{\text{{\rm l-UB}}}}

\newcommand{\MM}{\ensuremath{\text{{\sf MM}}}} 
\newcommand{\AX}{\ensuremath{\text{{\sf AX}}}} 
\newcommand{\CH}{\ensuremath{\text{{\sf CH}}}} 

\newcommand{\SSP}{\ensuremath{\text{{\sf SSP}}}}

\begin{document}

\begin{abstract}
We show that for $\Pi_2$-properties of second or 
third order arithmetic as formalized in appropriate natural signatures the apparently weaker notion of \emph{forcibility} 
overlaps with the standard notion of \emph{consistency} (assuming large cardinal axioms). 

Among such $\Pi_2$-properties we mention: the negation of the continuum hypothesis, 
Souslin Hypothesis, the negation of Whitehead's conjecture on free groups, the non-existence of outer automorphisms for the Calkin algebra, etc...
In particular this gives an a posteriori explanation of the success forcing (and forcing axioms) met in producing models of such properties.

Our main results relate generic absoluteness theorems for second order arithmetic, Woodin's axiom $(*)$ and forcing axioms to Robinson's notion of model companionship (as applied to set theory).
We also briefly outline in which ways these results provide an argument to refute $\CH$.

 \end{abstract}

\maketitle




\section*{Introduction}

\subsection*{Model completeness, model companionship, and the model companionship spectrum of a theory}
Model companionship and model completeness are model theoretic notions introduced
by Robinson which give a simple first order characterization of the way algebraically closed fields sits inside the class of rings with no zero-divisors. 
We start this paper rushing through the main properties of model completess and model companionship
(we will later on analyze carefully all these concepts in Section \ref{sec:modth}). 
Our aim is to show in a few paragraphs how we can use these notions to reformulate in a simple model-theoretic terminology 
deep generic absoluteness results for second order arithmetic by Woodin and others, 
as well as other major results on forcing axioms and 
Woodin's Axiom $(*)$. 

The key model-theoretic concept we are interested in is that of existentially closed model of a first order theory\footnote{We adopt the following notational conventions: $\sqsubseteq$ denotes the substructure relation between structures;
$\mathcal{M}\prec_n\mathcal{N}$ indicates that $\mathcal{M}$ is a $\Sigma_n$-elementary substructure of $\mathcal{N}$, we omit the $n$ to denote full-elementarity;
given a first order theory $T$, $T_\forall$ denotes the universal sentences which are consequences of $T$,
likewise we interpret $T_\exists, T_{\forall\exists},\dots$.} 
$T$:
\begin{Definition}
Let $\tau$ be a signature and $T$ be a first order theory.
$\mathcal{M}$ is $T$-existentialy closed ($T$-ec) if for 
any $\tau$-structure $\mathcal{N}\sqsupseteq\mathcal{M}$ which is a model of $T$ we have that
\[
\mathcal{M}\prec_1 \mathcal{N}.
\]
\end{Definition}

A key non-trivial fact is that $\mathcal{M}$ is $T$-ec if and only if it is $T_\forall$-ec.

It doesn't take long to realize that in signature 
$\tau=\bp{+,\cdot,0,1}$ the $\tau$-theory $T$ 
of fields has as its class of 
existentially closed models exactly the algebraically closed fields.
Note also that if we let $S$ be the class of rings with no zero-divisors which are not fields, we still have that
the $S$-existentially closed structures are the algebraically closed fields (even if no field is a model of $S$).

Model completeness and model companionship allow to generalize these features of the class of 
rings with no zero divisors to arbitrary first order theories.

\begin{Definition}
Let $\tau$ be a first order signature. 
\begin{itemize}
\item
A $\tau$-theory $T$ is \emph{model complete} if any model of $T$ is $T$-ec.
\item
$T$ is the \emph{model companion} of a $\tau$-theory $S$ if: 
\begin{itemize}
\item
any model of $S$ embeds into a model of $T$ and conversely, 
\item
$T$ is model complete.
\end{itemize}
\end{itemize}
\end{Definition}

In particular in signature $\tau=\bp{+,\cdot,0,1}$, the theory of algebraically closed fields is model complete
and is the model companion both of the theory of fields and of the theory of rings with no zero-divisors which are not fields.

We will also need here the following equivalent characterization of model completeness: 
$T$ is model complete whenever 
\begin{quote}
\emph{
For $\mathcal{M},\mathcal{N}$ models of $T$, $\mathcal{M}\prec\mathcal{N}$ if and only if $\mathcal{M}\prec_1\mathcal{N}$
if and only if $\mathcal{M}\sqsubseteq\mathcal{N}$}.
\end{quote}

Note also that: 
\begin{itemize}
\item
any theory $T$ admitting quantifier elimination is model complete;
\item
any model complete theory $T$ is the model companion of itself;
\item
two $\tau$-theories $T$ and $S$ which have no model in common can have the same model companion, but the model companion
of a theory $T$ if it exists is unique;
\item
if $T^*$ is the model companion of $T$ it can be the case that no model of $T$ is a model of $T^*$ and conversely;
\item
there are $\tau$-theories $T$ which do not admit a model companion (for example this is the case
for the theory of groups in signature $\tau=\bp{\cdot,1}$).
\end{itemize}
Much in the same way as
the algebraic closure of a ring $R$ with no zero-divisors closes off $R$ with respect to 
solutions to polynomial equations with coefficients in $R$ and which exist in some superring of $R$ which has no zero-divisors (and which does not have to be algebraically closed), for a theory $T$ with model companion $T^*$ any model 
$\mathcal{M}$ of $T$ brings to a supermodel $\mathcal{N}$ of $T^*$ which is obtained by adding (at least) the solutions to the existential formulae with 
parameters in $\mathcal{M}$ which are consistent with the universal fragment of $T$ (in the case of ring with no zero-divisors the key universal property one has to maintain is the non-existence of zero-divisors along with the ring axioms).

A key property of model companionship which brought our attention to this notion is the following (see Section \ref{sec:modth} for details):
\begin{Fact}\label{fac:maxpropmodcompcompl}
Let $\tau$ be a first order signature and $T$ be a \emph{complete} 
$\tau$-theory with model companion $T^*$.
Then $T^*$ is axiomatized by $T^*_{\forall\exists}$ and
TFAE for a $\Pi_2$-sentence $\psi$ for $\tau$:
\begin{itemize}
\item $T_\forall+\psi$ is consistent.
\item $\psi\in T^*$.
\end{itemize}
\end{Fact}

In case $T$ is a companionable non-complete theory, further weak hypothesis on $T$ (which are satisfied by set theory) allow to characterize its model companion $T^*$ as the unique theory axiomatized by the $\Pi_2$-sentences which
are consistent with the universal fragment of any completion of $T$ (see Lemma \ref{fac:proofthm1-2}).

Unlike other notions of complexity (such as stability, NIP, simplicity) model companionship and model completeness are very sensitive to the signature in which one formalizes a first order theory $T$.

\begin{Notation}\label{not:keynotation0}
For a given signature $\tau$, $\tau^*$ is the signature extending $\tau$
with new function symbols\footnote{As usual we confuse $0$-ary function symbols with constants.} $f_\phi$ and new
relation symbols $R_\phi$ for any $\tau$-formula $\phi(x_0,x_1,\dots,x_n)$.  
$T_\tau$ is the $\tau^*$-theory with axioms
\[
\AX^0_\phi:= \forall\vec{x}[\phi(\vec{x})\leftrightarrow R_\phi(\vec{x})]
\]
\[
\AX^1_\phi:= \forall x_1,\dots,x_n[\exists  y\phi(y,x_1,\dots,x_n)\rightarrow \phi(f_\phi(x_1,\dots,x_n),x_1,\dots,x_n)],
\]
as $\phi$ ranges over the $\tau$-formulae.
\end{Notation}

It is clear that any $\tau$-structure admits a unique extension to a $\tau^*$-model of $T_\tau$ and 
any $\tau$-theory
$T$ is such $T\cup T_\tau$ admits quantifier elimination, hence is model complete and is its own 
model companion relative to signature $\tau^*$. 
This holds regardless of whether the $\tau$-theory $T$ is model complete or admits
a model companion in signature $\tau$ (cfr. $T$ being the theory of groups in signature $\bp{\cdot,1}$).
On the other hand $T$ is stable (simple, NIP) if and only if so is $T_\tau$.
This is a serious drawback if one wishes to use model companionship to gauge the complexity of a mathematical theory $T$, since
model companionship of $T$ is very much dependent on the signature in which we formalize it: 
$T$ can trivially be model complete if we formalize it in a rich enough signature.

We now introduce a simple trick to render model companionship a useful classification tool for mathematical theories regardless of the signature 
in which we give their first order axiomatization. Roughly the idea is to consider all possible signatures in which a theory can be formalized and pay attention only to those for which the theory admits a model companion.

\begin{Definition}\label{def:compspectrum}
Let $\tau$ be a  signature and $F_\tau$ denote the set of $\tau$-formulae. 

Given $A\subseteq F_\tau\times 2$, let $\tau_A$ be the signature 
$\tau\cup\bp{R_\phi:(\phi,0)\in A}\cup\bp{f_\phi:(\phi,1)\in A}$. 
A $\tau$-theory 
$T$ is \emph{$(A,\tau)$-companionable} if 
\[
T_A=T\cup\bp{\AX^i_\phi,:(\phi,i)\in A}
\]
admits a model companion for the signature $\tau_A$.

Given a $\tau$-theory $T$ its \emph{$\tau$-companionship spectrum} is given by those
$A\subseteq F_\tau\times\bp{0,1}$ such that $T$ is $(A,\tau)$-companionable.
\end{Definition}

Note that $F_\tau\times\bp{i}$ is always in the companionship spectrum of $T$, but proving that some 
$\bar{A}\subsetneq F_\tau$ is such that some $A\subseteq \bar{A}\times 2$
is in the companionship spectrum of $T$ is a (possibly highly) non-trivial and 
informative result on $T$; model-companionability for $T$ amounts to say that $T$ is 
$(\emptyset,\tau)$-companionable.
The $\tau$-companionship spectrum of $T$ is non-informative if $T$ is model complete in signature $\tau$: in this case the $\tau$-companionship spectrum of $T$ is $\pow{F_\tau\times 2}$. 

Note also that even if $T$ is $(\emptyset,\tau)$-companionable there could be  many 
$A\subseteq F_\tau\times 2$ such that $T$ is $A$-companionable and many 
$B\subseteq F_\tau\times 2$ such that $T$ is not $B$-companionable;
in principle nothing prevents the families of such $A$s and $B$s to be both of size $2^{|F_\tau|}$ 
and to produce a complex ordering of the $\tau$-companionship spectrum of $T$ with respect to 
$\subseteq$.

To better grasp the above considerations, let for a $\tau$-theory $T$ $C_T$ be the category whose objects 
are the $\tau$-models of $T$ and whose arrows are the $\tau$-morphisms.
NIP, stability, simplicity are properties which consider only the objects in this category, model completeness 
and model companionship pay also attention to the arrows of this category. 
We get a much deeper insight on the properties of $C_T$ if we are able to detect for which 
$A\subseteq F_\tau\times 2$
$T_A$ is model companionable: for any $A\subseteq F_\tau\times 2$ in the passage from $C_T$ to 
$C_{T_A}$ we maintain the same class of objects, but the $\tau_A$-morphisms 
(i.e the arrows of $C_{T_A}$) are just the 
$\tau$-morphisms between models of $T$ which preseve the formulae in $A$, 
hence we are possibly destroying many arrows.

Our definition of $\tau$-companionship spectrum of a mathematical theory 
is apparently dependent on the signature $\tau$ in which 
we formalize it. We may argue that this is not the case, but to uncover why would bring 
us far afield and we defer this task to another paper. We will in this paper confine our attention to use 
this notion to analyze first order axiomatizations of set theory enriched with large cardinal axioms.
In this case we can certainly say that proving that some $A\subsetneq F_{\bp{\in}}\times 2$
is in the $\in$-companionship spectrum of set theory is an informative result: $\bp{\in}$ is a minimal 
signature in which set theory can be formalized (in the empty signature we certainly cannot formalize it), 
hence any 
$A\subseteq F_{\bp{\in}}\times 2$ for which set theory is $A$-companionable gives non-trivial information 
on set theory.
Moreover we can easily verify that any reasonable $\in$-axiomatization of set theory is not model complete for 
the $\in$-signature, hence the  $\in$-companionship spectrum of set theory is certainly non-trivial.

\subsection*{Some of our main results}
We can now state in an informative way key parts of our main results.

The first non-trivial result states that for any definable cardinal $\kappa$ there is
at least one signature admitting a constant for the cardinal $\kappa$ such that set theory is companionable for this signature. 

It is convenient from now on to adopt the following short-hand notation for structures:
\begin{Notation}
Given a signature $\tau$, $\mathcal{M}=(M,\tau^{\mathcal{M}})$ is a shorthand for the $\tau$-structure 
$(M,R^{\mathcal{M}}:R\in\tau)$.
\end{Notation}

\begin{Theorem}\label{thm:mainthm0}
Let $T\supseteq\ZFC$ be a $\in$-theory, and
$\kappa$ be a $T$-definable cardinal (i.e. such that for some
$\in$-formula $\phi_\kappa(x)$ 
$T$ proves $\exists!x[\phi_\kappa(x)\wedge x$\emph{ is a cardinal}$]$).

Then there is at least one  $A_\kappa\subsetneq F_\in$
such that letting $\bar{A}_\kappa=A_\kappa\times 2$:
\begin{enumerate}
\item\label{thm:mainthm0.0}
For all models $(V,\in)$ of $T$ 
$(H_{\kappa^+}^V,\bp{\in}_{\bar{A}_\kappa}^V)\prec_1 (V,\bp{\in}_{\bar{A}_\kappa}^V)$.
\item\label{thm:mainthm0.1}
$T$ is $\bar{A}_\kappa$-companionable.
\item\label{thm:mainthm0.2}
The model companion $T^*_\kappa$ of $T_{\bar{A}_\kappa}$ for signature $\bp{\in}_{\bar{A}_\kappa}$ is the
$\bp{\in}_{\bar{A}_\kappa}$-theory common to $H_{\kappa^+}^V$ as $(V,\in)$ 
ranges over $\in$-models of $T$
and $\kappa$ is the constant of $\bp{\in}^*$ given by the formula $\exists!x[\phi_\kappa(x)\wedge x$\emph{ is a cardinal}$]$.
\item\label{thm:mainthm0.3} 
$T^*_\kappa$ is also axiomatized by the $\Pi_2$-sentences for
$\bp{\in}_{\bar{A}_\kappa}$ which are consistent with 
$S_\forall$ for any 
$\bp{\in}_{\bar{A}_\kappa}$-theory
$S$ which is a complete extension of 
$T_{\bar{A}_\kappa}$.
\end{enumerate}
\end{Theorem}

Note that the above theorem allows to put in the companionship spectrum of any extension of $\ZFC$ at least one 
$\bar{A}_\kappa$ for each definable cardinal $\kappa$ such as $\omega,\omega_1,\dots, \aleph_\omega,\dots, \aleph_{\omega_1},\dots,\kappa,\dots$ for $\kappa$
the least inaccessible, measurable, Woodin, supercompact, extendible\dots

In case $\kappa=\omega,\omega_1$ we can say much more and prove that for $\Pi_2$-sentences in the appropriate signature
forcibility and consistency overlap (assuming large cardinal axioms). 

This gives an a  posteriori explanation of the success forcing has met in proving the consistency of $\Pi_2$-properties
(according to the right signature) for second or third order artithmetic: 
our results show that there are
no other means to prove the consistency of such statements.

\begin{Theorem}\label{thm:mainthm1}
Let $S$ be \emph{any} extension of 
\[
\ZFC+\text{\textbf{suitable} large cardinal axioms}
\]
in signature $\tau=\bp{\in}$.
There are $A_1\neq A_2\subseteq F_{\bp{\in}}$ recursive sets of $\in$-formulae such that
(letting $\bar{A}_i=A_i\times 2$ for $i=1,2$):
\begin{enumerate}
\item\label{thm:mainthm1.0}
For all models $(V,\in)$ of $S$ $(H_{\omega_i}^V,\bp{\in}_{\bar{A}_i}^V)\prec_1 (V,\bp{\in}_{\bar{A}_i}^V)$ for both $i=1,2$.

\item\label{thm:mainthm1.1}
$S$ is $\bar{A}_i$-companionable for both 
$i=1,2$\footnote{With very strong large cardinal axioms for the case for 
$T_{\bar{A}_2}$, and no large cardinal axioms in the case for $T_{\bar{A}_1}$.}.
\item\label{thm:mainthm1.2}
The model companion of $S_{\bar{A}_1}$ is the
$\tau_{\bar{A}_1}$-theory common to the models $H_{\aleph_1}^{V[G]}$ as $(V\in)$ ranges over 
$\in$-models of $S$ and $G$ is $V$-generic for 
some\footnote{If one is not at ease with the (inconsistent) assumption that $V[G]$ exists, 
this can be reformulated as: $(V\,in)\models\exists P\, (P\Vdash\psi^{H_{\omega_1}})$ and 
$(V,\in)\models S$.} $P\in V$.
\item\label{thm:mainthm1.3}
The model companion of $S_{\bar{A}_2}$ is the
$\tau_{\bar{A}_2}$-theory common to all
$H_{\aleph_2}^{V[G]}$ for $V[G]$ a forcing extension of $V$ which models
$\MM^{++}$ and $(V,\in)$ a $\in$-model of $S$ \footnote{With very 
strong large cardinal axioms holding in $V$. 
$\MM^{++}$ is one of the strongest forcing axioms.}.
\item\label{thm:mainthm1.4}
$(S_{\bar{A}_1})_\forall$ and $(S_{\bar{A}_2})_\forall$ are both invariant across 
forcing extensions of $V$ for any $\in$-model $(V,\in)$ 
of $S$ (assuming the existence of class many Woodin cardinals in $V$).
\end{enumerate}
\end{Theorem}

\begin{Corollary}\label{cor:maincor}
Assume $S$ extends $\ZFC$ with the correct large cardinal axioms.

Let $X$ be any among $\bar{A}_1,\bar{A}_2\subseteq F_{\bp{\in}}\times 2$ as in the previous theorem, and:
\begin{itemize}
\item
$S_X$ be the $\tau_X$-theory 
$S\cup \bp{\AX^i_\phi: (\phi,i)\in X}$, 
\item
$S^*_X$ be the model companion theory of $S_X$ given by the previous theorem. 
\end{itemize}

TFAE for any $\Pi_2$-sentence $\psi$ for 
$\tau_X$:
\begin{enumerate}[(A)]
\item\label{cor:maincorA} 
$\psi\in S^*_X$;
\item\label{cor:maincorB}
$R_\forall+\psi$ is consistent for all $\tau_X$-theories $R$ which are complete extensions of $S_X$; 
\item \label{cor:maincorC}
(if $X=\bar{A}_2$)
\[S_X\models\exists P\,\qp{\Vdash_P\psi^{H_{\omega_{2}}}};\]

\noindent (if $X=\bar{A}_1$)
\[S_X\models\exists P\,\qp{\Vdash_P\psi^{H_{\omega_{1}}}}.\]

\end{enumerate}

\end{Corollary}

In particular the equivalence of \ref{cor:maincorB} with \ref{cor:maincorC} shows that \emph{forcibility} and \emph{consistency} 
overlap for $\Pi_2$-sentences in signature $\bp{\in}_X$.

We complete this introduction outlining a bit more the significance of the above results and trying to get a better insight on what are the
signatures $\bp{\in}_{\bar{A}_1},\bp{\in}_{\bar{A}_2},\bp{\in}_{\bar{A}_\kappa}$ mentioned in the theorems.



\subsection*{What is the right signature for set theory?}

The $\in$-signature is certainly sufficient to give by means of $\ZFC$ a first order axiomatization of set theory (with eventually 
other extra hypothesis such as large cardinal axioms), but we can see rightaway that it is not efficient to formalize many 
basic set theoretic concepts. Consider for example the notion of ordered pair: on the board we write $x=\ap{y,z}$ to mean that
\emph{$x$ is the ordered pair with first component $y$ and second component $z$}. In set theory this concept is formalized by means
of Kuratowski's trick stating that $x=\bp{\bp{y},\bp{y,z}}$. However the $\in$-formula formalizing the above is:
\[
\exists t\exists u\;[\forall w\,(w\in x\leftrightarrow w=t\vee w=u)
\wedge\forall v\,(v\in t\leftrightarrow v=y)
\wedge\forall v\,(v\in u\leftrightarrow v=y\vee v=z)].
\]
It is clear that the meaning of this $\in$-formula is hardly decodable with a rapid glance (unlike $x=\ap{y,z}$), moreover just 
from the point of view of its syntactic complexity it is already $\Sigma_2$.
On the other hand we do not regard the notion of ordered pair as a complex or doubtful concept (as is the case for the notion of uncountability, or many of the properties of the continuum such as its correct place in the hierarchy of uncountable cardinals, etc...). Other vary basic notions such as: being a function, a binary relation, the domain or the range of a function, etc.. are formalized already by rather complicated $\in$-formulae, both from the point of view of readability for human beings and from the mere computation of their syntactic complexity according to the Levy hierarchy.

The standard solution adopted by set theorists (e.g. \cite[Chapter IV]{KUNEN}) is to regard as elementary all 
those properties which can be formalized using $\in$-formulae all of whose quantifiers are bounded to range 
over the elements of some set, i.e. the so called $\Delta_0$-formulae (see \cite[Chapter IV, Def. 3.5]{KUNEN}). We henceforth adopt this point of view and let
$B_0\subseteq F_{\bp{\in}}$ be the set of such formulae and denote by $\tau_{\ST}$ what according to our previous terminology 
should rather be $\bp{\in}_{B_0\times 2}$. 
For the sake of convenience and also to further outline some very nice syntactic features of $\ZFC$
as formalized in $\tau_\ST$, let us bring to front an explicit axiomatization of $T_{B_0}$ 
(which from now on will be denoted by $T_\ST$).

\begin{Notation}\label{not:keynotation}
\emph{}

\begin{itemize}
\item
$\tau_{\mathsf{ST}}$ is the extension of the first order signature $\bp{\in}$ for set theory 
which is obtained 
by adjoining predicate symbols
$R_\phi$ of arity $n$ for any $\Delta_0$-formula $\phi(x_1,\dots,x_n)$, function symbols of arity $k$
for any $\Delta_0$-formula $\theta(y,x_1,\dots,x_k)$
and constant symbols for 
$\omega$ and $\emptyset$.
\end{itemize}

\begin{itemize}
\item $\ZFC^{-}$ is the
$\in$-theory given by the axioms of
$\ZFC$ minus the power-set axiom.

\item
$T_\ST$ is the $\tau_\ST$-theory given by the axioms
\[
\forall \vec{x} \,(R_{\forall z\in y\phi}(y,z,\vec{x})\leftrightarrow \forall z(z\in y\rightarrow R_\phi(y,z,\vec{x}))
\]
\[
\forall \vec{x} \,[R_{\phi\wedge\psi}(\vec{x})\leftrightarrow (R_{\phi}(\vec{x})\wedge R_{\psi}(\vec{x}))]
\]
\[
\forall \vec{x} \,[R_{\neg\phi}(\vec{x})\leftrightarrow \neg R_{\phi}(\vec{x})]
\]
\[
\forall \vec{x}\qp{\exists!y \,R_{\phi}(y,\vec{x})\leftrightarrow R_{\phi}(f_\phi(\vec{x}),\vec{x})}
\]
for all $\Delta_0$-formulae $\phi(\vec{x})$, together with the $\Delta_0$-sentences
\[
\forall x\in\emptyset\,\neg(x=x),
\]
\[
\omega\text{ is the first infinite ordinal}
\]
(the former is an atomic $\tau_\ST$-sentence, the latter is expressible as the atomic sentence for 
$\tau_\ST$ stating that
$\omega$ is a non-empty limit ordinal all whose elements are successor ordinals or $0$).
\item
$\ZFC^-_\ST$ is the $\tau_\ST$-theory 
\[
\ZFC^{-}\cup T_\ST.
\] 
\item
Accordingly we define $\ZFC_\ST$.
\end{itemize}
\end{Notation}

Note that $T_\ST$ is axiomatized by $\Pi_2$-sentences of $\tau_\ST$.

\subsection*{Levy absoluteness and model companionship results for set theory}
Kunen's \cite[Chapter IV]{KUNEN} gives a rather convincing summary of the reasons why it is convenient to formalize set theory using
$\tau_\ST$ rather than $\in$. We focus here on the role Levy's absoluteness plays in the search of 
$A\subseteq F_{\bp{\in}}\times 2$
for which set theory is $A$-companionable. 
\begin{Lemma}\label{lem:levabsgen}
Let $(V,\in)$ be a model of $\ZFC$ and $\kappa$ be an infinite cardinal for $V$.
Then
\[
(H_{\kappa^+}^V,\tau_\ST^V,A: A\subseteq\pow{\kappa}^k,\, k\in\mathbb{N})\prec_1(V,\tau_\ST^V,A: A\subseteq\pow{\kappa}^k,\,k\in\mathbb{N})
\]
\end{Lemma}
Its proof is a trivial variant of the classical result of Levy (which is the above theorem stated just for the signature $\tau_\ST$); it is given in
\cite[Lemma 5.3]{VIAVENMODCOMP}.

The upshot is that for any model $V$ of $\ZFC$ and any signature $\sigma$
such that 
\[
\tau_\ST\cup\bp{\kappa}\subseteq\sigma\subseteq \tau_\ST\cup\bp{\kappa}\cup\bigcup_{k\in\mathbb{N}}\pow{\kappa}^k
\]
$H_{\kappa^+}$ is $\Sigma_1$-elementary in $V$ according to $\sigma$.
This is a first indication that for a $\ZFC$-definable cardinal $\kappa$ (e.g. $\kappa=\omega,\omega_1,\aleph_\omega,\dots$, more precisely $\kappa$ being provably in some $T\supseteq\ZFC$ the unique solution of an $\in$-formula $\phi_\kappa(x)$)
if $\sigma_\kappa=\tau_\ST\cup\bp{\kappa}$ and $T_\kappa$ is the $\sigma_\kappa$-theory given by
$\ZFC+\phi_\kappa(\kappa)$, we get that
the $\sigma_\kappa$-theory common to all of the $H_{\kappa^+}^V$ as $V$ ranges over the model of $\ZFC$ is not that far from being $\ZFC$-ec, since a model of this theory is always a $\Sigma_1$-substructure of some
$\sigma_\kappa$-model of $\ZFC$.

A second indication that the $\sigma_\kappa$-theory of $H_{\kappa^+}$ is close to be the model companion of the $\sigma_\kappa$-theory of $V$ is the fact that the $\Pi_2$-sentence for $\sigma_\kappa$
\[
\forall x\exists f:\kappa\to x\text{ surjective function}
\]
is realized in $H_{\kappa^+}^V$ for any model $V$ of $\ZFC$ (note that by Levy's absoluteness this sentence is consistent
with the universal fragment of the  $\sigma_\kappa$-theory of $V$, hence by Fact \ref{fac:maxpropmodcompcompl}
 it belongs to the model companion of set theory for $\sigma_\kappa$ --- if such a model companion exists).

In particular if some $T\supseteq\ZFC$ is $\sigma$-companionable for some $\sigma$ as above, the model companion of
$T$ for $\sigma$ should be the theory of
$H_{\kappa^+}^{\mathcal{M}}$ for suitably chosen $\mathcal{M}$ which are models of set theory.

A natural question is: 
\begin{quote}
\emph{Can we cook up $\sigma\supseteq \tau_\ST\cup\bp{\kappa}$ so that  the $\sigma$-theory of $H_{\kappa^+}$ is the model companion of the $\sigma$-theory of $V$?} 
\end{quote}
Theorem \ref{thm:mainthm0} answer affirmatively to this question for many natural choices of $\sigma$ and for all definable cardinals $\kappa$.

\subsection*{Why the continuum is the second uncountable cardinal}\label{subsec:negCH}

Theorem \ref{thm:mainthm1} refines Thm. \ref{thm:mainthm0} for the cases $\kappa=\omega,\omega_1$. 
In these cases our knowledge of the theory of $H_{\kappa^+}$
is much more extensive; moreover most of mathematics can be formalized in $H_{\omega_1}$ (all of second order arithmetic)
or in $H_{\omega_2}$ (most of third order arithmetic).

We now want to outline briefly why Thm. \ref{thm:mainthm1} provides an interesting metamathematical 
argument in favour of strong forcing axioms and against $\CH$.
The considerations of this brief paragraph will be expanded in more details in a forthcoming paper and have been elaborated jointly with Giorgio Venturi.
Those who are familiar with forcing axioms know that Martin's maximum and its bounded forms
have been instrumental to prove the consistency of a solution of many problems of third order arithmetic which are provably undecidable in $\ZFC$
(or even in $\ZFC$ supplemented by large cardinal axioms), 
a sample of these solutions include:
the negation of the continuum hypothesis \cite{FORMAGSHE,CAIVEL06,woodinBOOK,TOD02,MOO06}, 
the negation of Whitehead conjecture on free abelian groups \cite{SHE74}, 
the non-existence of outer automorphism of the Calkin algebra \cite{FAR11}, 
the Suslin hypothesis \cite{JENSEN}, 
the existence of a five element basis for uncountable linear order \cite{moore.basis}\dots
All statements of the above list (with the exception of the non-existence of outer automorphism of the Calkin algebra) and many others can be
formalized as $\Pi_2$-sentences in signature $\tau_{\omega_1}=\tau_\ST\cup\bp{\omega_1}$ (where $\omega_1$ is a 
new constant symbol which is the unique solution of some formula in one free variable defining the first uncountable cardinal).
For example $\neg\CH$ is formalized by 
\[
\forall f\,\qp{(f\text{ is a function}\wedge \dom(f)=\omega_1)\rightarrow \exists r\,(r\subseteq \omega\wedge r\not\in\ran(f))}.
\]
In particular there has been empiric evidence that forcing axioms produce models of set theory which maximize the family of $\Pi_2$-sentences
which hold true in $H_{\omega_2}$ for the signature $\tau_{\omega_1}$.
Thm. \ref{thm:mainthm1} makes this empiric evidence a true mathematical fact: first of all it is important to note here that (sticking to the notation of Thm. \ref{thm:mainthm1})
$\bp{\in}_{\bar{A}_2}\supseteq\tau_{\omega_1}$.
Now let $T$ be a theory as in the assumption of Thm. \ref{thm:mainthm1}; take (in signature $\bp{\in}_{\bar{A}_2}$)
any $\Pi_2$-sentence $\psi$ which is consistent with $S_\forall$ whenever $S$ is a complete extension of 
$T_{\bar{A}_2}$;
then by \ref{cor:maincorA}$\Longleftrightarrow$\ref{cor:maincorB} of Corollary \ref{cor:maincor} 
$\psi$ is in the model companion of $T_{\bar{A}_2}$,
and 
Thm. \ref{thm:mainthm1.3}(\ref{thm:mainthm1.3})
(almost) asserts that $\psi^{H_{\omega_2}}$ is 
derivable from $\MM^{++}$.
Note that $\MM^{++}$ is one of the strongest forcing axioms.

Another key observation is that (assuming large cardinals)
the signature $\bp{\in}_{\bar{A}_2}$ is such that the universal fragment of set theory as formalized in 
$\bp{\in}_{\bar{A}_2}$ is 
\emph{invariant through forcing extensions of $V$}. 
What this means is that one \emph{can and must use forcing} to establish whether some 
$\Pi_2$-sentence $\psi$
is in the model companion of set theory according to $\bp{\in}_{\bar{A}_2}$.

This is the major improvement of Thm. \ref{thm:mainthm1} with respect to Thm. \ref{thm:mainthm0}: for most of 
the signatures $\tau_{\bar{A}_\kappa}$ mentioned in Thm. \ref{thm:mainthm0} we cannot just 
use forcing to establish whether a $\Pi_2$-sentence $\psi$ for this signature is in the model companion
of set theory for
 $\tau_{\bar{A}_\kappa}$. 

Let us develop more on this point because it is in our eyes one of the major advances given by the results of the present paper.
Take $S\supseteq \ZFC$; for a given $X\subseteq F_{\bp{\in}}\times 2$ for which we can prove that $S_X$ has a model companion in signature $\bp{\in}_X$
we would like to show that a certain $\Pi_2$-sentence $\psi$ for $\bp{\in}_X$ is in the model companion of $S_{X}$. 

Let us first suppose that $X$ is some $\bar{A}_\kappa$ as in Thm. \ref{thm:mainthm0}.
A first observation is that (with the exception of the $\Delta_0$-formulae) all the formulae in $A_\kappa$
define subsets of $\pow{\kappa}^n$ for some $n$, hence
$(H_{\kappa^+}^V,\bp{\in}_{\bar{A}_\kappa}^V)\prec_1 (V,\bp{\in}_{\bar{A}_\kappa}^V)$ for any $(V,\in)$ which models $\ZFC$.
This gives that if $(W,\bp{\in}_{\bar{A}_\kappa}^W)$ models $(S_X)_\forall$, then so does $(H_{\kappa^+}^W,\bp{\in}_{\bar{A}_\kappa}^W)$. 

A natural strategy to put $\psi$ in the model companion of $S_X$ 
would then be to start from some complete $T\supseteq S$ and some $(V,\in)$ model of $T$; then force over $V$
that in some $V[G]$ $\psi^{H_{\kappa^+}^{V[G]}}$ holds true;
if $(T_{\bar{A}_\kappa})_\forall$ holds in $V[G]$, then $\psi$ would be in the model companion of $S_{\bar{A}_\kappa}$ by 
Thm. \ref{thm:mainthm0}(\ref{thm:mainthm0.3}): Levy's absoluteness applied to 
$(V[G],\tau_{\bar{A}_\kappa}^{V[G]})$ would yield that 
$H_{\kappa^+}^{V[G]}\models\psi+(T_{\bar{A}_\kappa})_\forall$.

Now starting from any model $V$ of $S$ 
we may be able to design a forcing in $V$ 
such that $\psi^{H_{\kappa^+}^{V[G]}}$ holds if $G$ is $V$-generic for this forcing, 
but it may be the case that 
$(S_{\bar{A}_\kappa})_\forall$ fails in $V[G]$; in which case we cannot
use $H_{\kappa^+}^{V[G]}$ as a witness that $\psi$ is in the model companion of $S_{\bar{A}_\kappa}$. 
Remark \ref{rem:limbobangeninv} shows that $(S_{\bar{A}_\kappa})_\forall$ is not preserved through forcing extensions 
whenever $\kappa>\omega_1$.

On the other hand for the signatures $\tau_X$ for $X$ being the $\bar{A}_1$ or $\bar{A}_2$ mentioned in Thm. \ref{thm:mainthm1} the above strategy works:
the universal fragment of $(S_X)_\forall$ is preserved through the forcing extensions of models of $S$;
hence $\psi$ will be in the model companion of $S_X$ if for any model $V$ of $S$
we can design a forcing making true $\psi$ in $H_{\kappa^+}^{V[G]}$ 
(for $\kappa=\omega,\omega_1$ according to whether 
$\psi$ is a formula for $\tau_{\bar{A}_1}$ or for $\tau_{\bar{A}_2}$).

Summing up one \emph{may and should only} use forcing to establish the consistency with large cardinals of $\psi^{H_{\omega_2}}$ for any 
$\Pi_2$-sentence formalizable in signatures $\tau_{\omega_1}\subseteq \bp{\in}_{\bar{A}_2}$:  the strategy we outlined above is
efficient (as the many applications of forcing axioms have already shown) and sufficient to compute all $\Pi_2$-sentences which axiomatize the model companion of $S_{\bar{A}_2}$, provided $S$ is any set theory satisfying sufficiently strong large 
cardinal axioms (by Corollary \ref{cor:maincor} all other means to produce the consistenty of $\psi$ with
the universal fragment of $S$ are reducible to forcing).

Our take on the above considerations is that if one embraces the standpoint that the universe of sets should be as large as possible, model companionship (in particular
Fact \ref{fac:maxpropmodcompcompl} -- actually its more refined version provided by Lemma \ref{fac:proofthm1-2} and used in Thm. \ref{thm:mainthm1}) gives a simple model theoretic property to instantiate this slogan: all $\Pi_2$-sentences talking about $\omega_1$ (i.e. expressible in signature 
$\tau_{\omega_1}$)
which are not outward contradictory with the basic properties of $\omega_1$ 
(i.e. with the universal theory of some model of 
$\ZFC+$\emph{large cardinals} in signature $\tau_{\omega_1}$) should hold true in $H_{\omega_2}$.
This is what Thm. \ref{thm:mainthm1} says to be the case in models of strong forcing axioms such as 
$\MM^{++}$.

Note that this is exactly parallel to the way one singles out algebraically closed fields from rings with no zero-divisors: in this set-up
one is interested to solve polynomial equations while preserving the ring axioms and not adding zero-divisors; the $\Pi_2$-sentences
for the signature $\bp{+,\cdot,0,1}$ which are consistent with the ring axioms and the non existence of zero divisors are exactly 
the axioms of algebraically closed fields.

Now coming back to $\CH$ we already observed that its negation is a $\Pi_2$-sentence for $\tau_{\omega_1}$ (hence also for $\bp{\in}_{\bar{A}_2}$), 
but we can actually get more.
Caicedo and Veli\v{c}kovi\'{c} \cite{CAIVEL06} proved that there is a quantifier free $\tau_{\omega_1}$-formula $\phi(x,y,z)$ such that
$(\forall x,y\exists z\phi(x,y,z))^{H_{\omega_2}}$ is forcible (by a proper forcing) over any model of $\ZFC$;
moreover if $V\models\ZFC+(\forall x,y\exists z\phi(x,y,z))^{H_{\omega_2}}$, then $V\models 2^\omega=\omega_2$.
In particular if we accept as true large cardinal axioms and we require that the correct axiomatization of set theory
maximizes the set of $\Pi_2$-sentences for $\tau_{\omega_1}$ which may hold for $H_{\aleph_2}$,
we are bound to accept that $2^\omega=\omega_2$ holds true.

\subsection*{Structure of the paper}
It is now a good place to streamline the remainder of this paper and specify what the reader need to know in order to grasp 
each of its parts.

\begin{itemize}
\item
Section \ref{sec:modth} gives a detailed and self-contained account of model companionship; the unique result
which we are not able to trace elsewhere in the literature is Lemma \ref{fac:proofthm1-2},
which isolates a key property of (possibly incomplete) first order theories granting model companionship results; we apply it in later parts of this paper to various (possibly recursive or incomplete) axiomatizations of set theory. Since we expect 
that many of our readers are not familiar with model companionship, we decided it was worth including 
here the key results (with proofs) on this notion. The reader familiar with these notions can skim through this section 
or jump it and refer to its relevant bits when needed elsewhere.
\item
Section \ref{sec:Hkappa+} proves Theorem \ref{thm:mainthm0}.
 \item
Section \ref{sec:geninv} proves the results needed to establish item \ref{thm:mainthm1.4} of Thm. \ref{thm:mainthm1}.
 We first give a self-contained proof of the form of Woodin's generic absoluteness results
 for second order arithmetic we employ in this paper. This identifies which subsets of $F_{\bp{\in}}$ 
 can play the role of $A_1$ for item \ref{thm:mainthm1.4} of Thm. \ref{thm:mainthm1}.
 Then we show that the universal theory of $V$ as formalized in a signature extending $\tau_\ST$
 with predicates for the non-stationary ideal and for the universally Baire sets cannot be changed using set sized forcing if there are class many Woodin cardinals. 
 This identifies which subsets of $F_{\bp{\in}}$ can play the role of $A_2$ for item \ref{thm:mainthm1.4} of Thm. \ref{thm:mainthm1}. 
\item
Section \ref{sec:Homega1} deals with Theorem \ref{thm:mainthm1} for the signature $\tau_{\bar{A}_1}$.
We expand slightly the results of \cite{VIAVENMODCOMP}: by taking advantage of Lemma \ref{fac:proofthm1-2}, we are able here to generalize also to non complete axiomatizations of set theory the model companionship results given in \cite{VIAVENMODCOMP} for complete set theories.
\item
Section \ref{sec:Homega2} deals with Theorem \ref{thm:mainthm1} for the signature $\tau_{\bar{A}_2}$.
\item
We conclude the paper with a final section with some comments and open questions.
\end{itemize}
Any reader familiar enough with set theory and model theory to follow this introduction can easily grasp the content of Sections
\ref{sec:modth}, \ref{sec:Hkappa+}. The same applies for the results of Section \ref{sec:Homega1} 
provided one accepts as a black-box Woodin's generic absoluteness results for second order arithmetic
given in Section \ref{sec:geninv}.
The proofs in Section \ref{sec:geninv} require familiarity with Woodin's stationary tower forcing and 
(in its second part, cfr. Section \ref{subsec:geninvtoa}) also with Woodin's $\Pmax$-technology. Section \ref{sec:Homega2} can be fully appreciated only by readers familiar with forcing axioms, Woodin's stationary tower forcing, Woodin's $\Pmax$-technology. 
 
\subsection*{Acknowledgements}
{\small 
This paper wouldn't exist without the brilliant idea by Venturi to relate the study of the generic multiverse of set theory to Robinson's model companionship, or without the major breakthrough of Asper\`o and Schindler
establishing that Woodin's axioms $(*)$ follows from $\MM^{++}$. I'm grateful to Boban Velickovic for many useful discussions on the scopes and limits of the results presented here (the necessity of large cardinal assumptions in the hypothesis of Thm. \ref{thm:PI1invomega2}, and Remark \ref{rem:limbobangeninv} are due to him). 
I'm also grateful to Philipp Schlicht who realized Thm. \ref{thm:mainthm0} could be easily established by slightly generalizing the proofs of results occurring in previous drafts of this paper.
I've had fruitful discussions on the content of this paper with many people, let me mention and thank David Asper\'o, Ilijas Farah, Juliette Kennedy, Menachem Magidor, Ralf Schindler, Jouko Vaananen,  Andres Villaveces, Hugh Woodin.
Clearly none of the persons mentioned here has any responsibility for any error or 
orror existing in this paper....
It has been important to have the opportunity to present these results in several set theory (or logic) seminars among which those in Toronto, Muenster, Jerusalem, Paris, Bogot\`a, Chicago, Helsinki, Torino.
This research has been performed mostly while on sabbatical in the \'Equipe de Logique Math\'ematique of the University of Paris 7 in the academic year 2019-2020; as long as possible it has been a productive and pleasant experience (until the Covid19 pandemic took place).
}

\section{Existentially closed structures, model completeness, model companionship}\label{sec:modth}

We present this topic expanding on 
\cite[Sections 3.1-3.2]{TENZIE}.
We decided to include detailed proofs since the presentation of \cite{TENZIE} is (in some occasions) rather 
sketchy, and the focus is not exactly ours.

The first objective is to isolate necessary and sufficient conditions granting that some $\tau$-structure 
$\mathcal{M}$ embeds into some model of some $\tau$-theory $T$. 

We expand Notation \ref{not:keynotation} as follows:
\begin{notation} 
We feel free to confuse a $\tau$-structure
$\mathcal{M}=(M,\tau^M)$ with its domain $M$
and an ordered tuple $\vec{a}\in\mathcal{M}^{<\omega}$
with its set of elements. Moreover we often
write $\mathcal{M}\models\phi(\vec{a})$ rather than
$\mathcal{M}\models\phi(\vec{x})[\vec{x}/\vec{a}]$ when
$\mathcal{M}$ is $\tau$-structure $\vec{a}\in\mathcal{M}^{<\omega}$, $\phi$ is a $\tau$-formula.
We let the atomic diagram of a $\tau$-model $\mathcal{M}=(M,\tau^M)$ be the family of quantifier free sentences $\phi(\vec{a})$
in signature $\tau\cup M$ such that .$\mathcal{M}\models\phi(\vec{a})$.
\end{notation}

\begin{definition}
Given $\tau$-theories $T,S$,
a $\tau$-sentence $\psi$ separates $T$ from $S$ if $T\vdash\psi$ and $S\vdash\neg\psi$.

$T$ is $\Pi_n$-separated from $S$ if some $\Pi_n$-sentence for $\tau$ 
separates $T$ from $S$.
\end{definition}

\begin{lemma}\label{lem:biembequivequnivth}
Assume $S,T$ are $\tau$-theories. 
TFAE:
\begin{enumerate}
\item \label{lem:biembequivequnivth-1}  $T$ is not $\Pi_1$-separated from $S$
(i.e. no universal sentence $\psi$ is such that $T\vdash \psi$ and $S\vdash\neg\psi$).
\item \label{lem:biembequivequnivth-3} 
There is \emph{some} $\tau$-model $\mathcal{M}$ of $S$ 
which can be embedded in 
some
$\tau$-model $\mathcal{N}$ of $T$.
\end{enumerate}
\end{lemma}

See also \cite[Lemma 3.1.1, Lemma 3.1.2, Thm. 3.1.3]{TENZIE}
\begin{proof}
We assume $T,S$ are closed under logical consequences.

\begin{description}
\item[(\ref{lem:biembequivequnivth-3}) implies (\ref{lem:biembequivequnivth-1})]
By contraposition we prove 
$\neg$(\ref{lem:biembequivequnivth-1})$\to\neg$(\ref{lem:biembequivequnivth-3}).

Assume some universal sentence 
$\psi$ separates $T$ from $S$. Then for any model of $T$, all its substructures model
$\psi$, therefore they cannot be models of $S$.
\item[(\ref{lem:biembequivequnivth-1}) implies (\ref{lem:biembequivequnivth-3})]
By contraposition we prove 
$\neg$(\ref{lem:biembequivequnivth-3})$\to\neg$(\ref{lem:biembequivequnivth-1}).

Assume that for any model $\mathcal{M}$ of $S$ and $\mathcal{N}$ of $T$ 
$\mathcal{M}\not\sqsubseteq\mathcal{N}$. We must show that $T$ is $\Pi_1$-separated from $S$.

Given a $\tau$-structure $\mathcal{M}=(M,\tau^M)$ which models $S$,
let $\Delta_0(\mathcal{M})$ be the atomic diagram\footnote{We let the atomic diagram of a $\tau$-model $\mathcal{M}=(M,\tau^M)$ be the family of quantifier free formulae
in signature $\tau\cup M$ which holds in the natural expansion of 
$\mathcal{M}$ to $\tau\cup M$.} of $\mathcal{M}$ in the signature
$\tau\cup M$.

The theory $T\cup\Delta_0(\mathcal{M})$ is inconsistent,
otherwise $\mathcal{M}$ embeds into some model of $T$:
let $\bar{\mathcal{Q}}$ be a $\tau\cup\mathcal{M}$-model 
of $\Delta_0(\mathcal{M})\cup T$ and
$\mathcal{Q}$ be the $\tau$-structure obtained from $\bar{\mathcal{Q}}$ omitting the interpretation of the constants not in $\tau$. Clearly $\mathcal{Q}$ models $T$.
The interpretation of the constants in $\tau\cup\mathcal{M}$ inside 
$\bar{\mathcal{Q}}$ defines a $\tau$-substructure of $\mathcal{Q}$ isomorphic to $\mathcal{M}$.

By compactness (since $\Delta_0(\mathcal{M})$ is closed under finite conjunctions)
 there is a quantifier free $\tau$-formula 
$\psi_{\mathcal{M}}(\vec{x})$ and $\vec{a}\in\mathcal{M}^{<\omega}$ such that
$T+\psi_{\mathcal{M}}(\vec{a})$ is inconsistent.
This gives that $T\vdash\neg\psi_{\mathcal{M}}(\vec{a})$.
Since $\vec{a}$ is a family of constants never occurring in $T$, we get that
$T\vdash\forall\vec{x}\neg\psi_{\mathcal{M}}(\vec{x})$ and 
$\mathcal{M}\models \exists \vec{x}\psi_{\mathcal{M}}(\vec{x})$.

The theory 
\[
S\cup\bp{\neg\exists \vec{x}\psi_{\mathcal{M}}(\vec{x}):\mathcal{M}\models S}
\]
is inconsistent,
since $\neg\exists \vec{x}\psi_{\mathcal{M}}(\vec{x})$ fails
 in any model $\mathcal{M}$ of $S$.

By compactness 
there is a finite set of formulae $\psi_{\mathcal{M}_1}\dots\psi_{\mathcal{M}_k}$ such that
\[
S+ \bigwedge\bp{\neg\exists \vec{x}_i\psi_{\mathcal{M}_i}(\vec{x}_i):i=1,\dots,k}
\]
is inconsistent. This gives that
\[
S\vdash \bigvee_{i=1}^k\exists \vec{x}_i\psi_{\mathcal{M}_i}(\vec{x}_i).
\]
The $\tau$-sentence $\psi:=\bigvee_{i=1}^k\exists \vec{x}_i\psi_{\mathcal{M}_i}(\vec{x}_i)$ holds 
in all models of $S$ and its negation
\[
\bigwedge\bp{\neg\exists \vec{x}_i\psi_{\mathcal{M}_i}(\vec{x}_i):i=1,\dots,k}
\]
is a conjunction of universal sentences (hence ---modulo logical equivalence--- universal) derivable from $T$. Hence $\neg\psi$ separates $T$ from $S$.
\end{description}
\end{proof}

The following Lemma shows that models of $T_\forall$ can always be extended
to superstructures which model $T$.

\begin{lemma}\label{lem:keylemembed}
Let $T$ be a $\tau$-theory and $\mathcal{M}$ be a $\tau$-structure.
TFAE:
\begin{enumerate}
\item \label{lem:keylemembed-1}
$\mathcal{M}$ is a $\tau$-model of $T_\forall$.
\item \label{lem:keylemembed-2}
There exists $\mathcal{N}\sqsupseteq \mathcal{M}$
which models $T$.
\end{enumerate}
\end{lemma}
\begin{proof}
(\ref{lem:keylemembed-2}) implies (\ref{lem:keylemembed-1}) is trivial.

Conversely:
\begin{claim}
$T$ is not $\Pi_1$-separated from $\Delta_0(\mathcal{M})$ (in the signature $\tau\cup\mathcal{M}$).
\end{claim}
\begin{proof}
If not there are $\vec{a}\in \mathcal{M}^{<\omega}$,
and a quantifier free 
$\tau$-formula
$\phi(\vec{x},\vec{z})$ such that
\[
T\vdash \forall\vec{z}\phi(\vec{a},\vec{z}),
\]
while
\[
\Delta_0(\mathcal{M})\vdash\neg\forall\vec{z}\phi(\vec{a},\vec{z}). 
\]
The latter yields that
\[
\Delta_0(\mathcal{M})\vdash\exists\vec{x}\exists\vec{z}\neg\phi(\vec{x},\vec{z}),
\]
and therefore also that
\[
\mathcal{M}\models\exists\vec{x}\exists\vec{z}\neg\phi(\vec{x},\vec{z}).
\]

On the other hand, since the constants $\vec{a}$ do not appear in any of the sentences in $T$, we also get that
\[
T\vdash \forall\vec{x}\forall\vec{z}\phi(\vec{x},\vec{z}).
\]
This is a contradiction since $\mathcal{M}$ models $T_\forall$.
\end{proof}
By the Claim and Lemma \ref{lem:biembequivequnivth}
some $\tau\cup\mathcal{M}$-model $\bar{\mathcal{P}}$
of $\Delta_0(\mathcal{M})$ embeds into some $\tau\cup\mathcal{M}$-model 
$\bar{\mathcal{Q}}$ of 
$T$. Let $\mathcal{Q}$ be the $\tau$-structure obtained from $\bar{\mathcal{Q}}$ omitting the interpretation of the constants not in $\tau$. Then $\mathcal{Q}$ models $T$ and 
contains a substructure isomorphic to $\mathcal{M}$.
\end{proof}

\begin{corollary}[Resurrection Lemma]\label{cor.resurrlemma}
Assume $\mathcal{M}\prec_1\mathcal{N}$ are $\tau$-structures. 
Then there is $\mathcal{Q}\sqsupseteq\mathcal{N}$ which is an elementary extension of $\mathcal{M}$.
\end{corollary}
\begin{proof}
Let $T$ be the elementary diagram $\Delta_\omega(\mathcal{M})$  of 
$\mathcal{M}$ in the signature $\tau\cup\mathcal{M}$.
It is easy to check that any model of $T$ when restricted to the signature $\tau$ is an elementay extension
of $\mathcal{M}$.
Since $\mathcal{M}\prec_1\mathcal{N}$, the natural extension of $\mathcal{N}$ to a 
$\tau\cup\mathcal{M}$-structure
realizes the $\Pi_1$-fragment of $T$
in the signature $\tau\cup\mathcal{M}$.
Now apply the previous Lemma.
\end{proof}
The Resurrection Lemma motivates the resurrection axioms introduced by Hamkins and Johnstone in 
\cite{HAMJOH13}, and their iterated versions introduced by the author and Audrito in \cite{VIAAUD14}.

\subsection{Existentially closed structures}

The objective is now to isolate the ``generic'' models of some universal theory
$T$ (i.e. all axioms of $T$ are universal sentences). These are described by the 
$T$-existentially closed models.

\begin{definition}
Given a first order signature $\tau$,
let $T$ be any consistent $\tau$-theory.
A $\tau$-structure $\mathcal{M}$ is $T$-existentially closed ($T$-ec) if
\begin{enumerate}
\item
$\mathcal{M}$ can be embedded in a model of $T$.
\item
$\mathcal{M}\prec_{\Sigma_1}\mathcal{N}$ for all 
$\mathcal{N}\sqsupseteq\mathcal{M}$ which are models of $T$.
\end{enumerate}
\end{definition}

In general $T$-ec models need not be 
models\footnote{For example let 
$T$ be the theory of commutative rings with no zero divisors which are not fields
in the signature $(+,\cdot,0,1)$. Then 
the $T$-ec structures are exactly all the algebraically closed fields, and no $T$-ec model is a model of $T$. By Thm. \ref{Thm:mainthm-1}
$(H_{\omega_1},\sigma_{\omega}^V)$ is $S$-ec for $S$ the $\sigma_{\omega}$-theory of $V$, 
but it is not a model of $S$: the $\Pi_2$-sentence asserting that every set has countable transitive closure is true in $(H_{\omega_1},\sigma_{\omega}^V)$ but denied by $S$.} of $T$,
but only of their universal fragment. 
A standard diagonalization argument shows that for any theory $T$ there are $T$-ec
models, see Lemma \ref{lem:exTecmod} below or \cite[Lemma 3.2.11]{TENZIE}.

A trivial observation which will come handy in the sequel is the following:
\begin{fact}\label{fac:presECmod}
Assume $\mathcal{M}$ is a $T$-ec model and $S\supseteq T$ is such that
some $\mathcal{N}\sqsupseteq\mathcal{M}$ models $S$.
Then $\mathcal{M}$ is $S$-ec.
\end{fact}

\begin{proposition}\label{prop:pi2satKaiserhull}
Assume a $\tau$-structure $\mathcal{M}$ is $T$-ec. 
Then:
\begin{enumerate}
\item \label{rmk:itm1} $\mathcal{M}\models T_\forall$.
\item \label{rmk:itm0} $\mathcal{M}$ is also $T_\forall$-ec.
\item \label{rmk:itm2} If $\mathcal{N}\prec_{\Sigma_1}\mathcal{M}$, then $\mathcal{N}$ is also $T$-ec.
\item \label{rmk:itm3} Let $\forall\vec{x}\exists\vec{y}\psi(\vec{x},\vec{y},\vec{a})$ be a
$\Pi_2$-sentence with 
$\psi(\vec{x},\vec{y},\vec{z})$ quantifier free $\tau$-formula and
parameters $\vec{a}$ in $\mathcal{M}^{<\omega}$. Assume it holds in some  
$\mathcal{N}\sqsupseteq\mathcal{M}$ which models $T_\forall$, then it holds in
$\mathcal{M}$.
\item \label{rmk:itm4} 
Let $S$ be the $\tau$-theory of $\mathcal{M}$.
For any $\Pi_2$-sentence $\psi$ in the signature $\tau$ TFAE:
\begin{itemize}
\item $\psi$ holds in some model of $S_\forall$.
\item  $\psi$ holds in $\mathcal{M}$.
\end{itemize}
\end{enumerate}
\end{proposition}
\begin{proof}
\emph{}

\begin{description}

\item[(\ref{rmk:itm1})]
There is at least one super-structure of $\mathcal{M}$ which models
$T$, and any
$\psi\in T_\forall$ holds in this superstructure, hence in $\mathcal{M}$.

\item[(\ref{rmk:itm0})]
Assume $\mathcal{M}\sqsubseteq\mathcal{P}$ for some model $\mathcal{P}$ of $T_\forall$.
We must argue that $\mathcal{M}\prec_1\mathcal{P}$.

By Lemma~\ref{lem:keylemembed}, there is $\mathcal{Q}\sqsupseteq\mathcal{P}$ which models $T$.

Since $\mathcal{M}$ and $\mathcal{Q}$ are both models of $T$ and 
$\mathcal{M}$ is $T$-ec, we get the following diagram:
\[
		\begin{tikzpicture}[xscale=0.8,yscale=-0.4]
				\node (A0_0) at (0, 0) {$\mathcal{M}$};
				\node (A0_2) at (6, 0) {$\mathcal{Q}$};
				\node (A1_1) at (3, 3) {$\mathcal{P}$};
				\path (A0_0) edge [->]node [auto] {$\scriptstyle{\Sigma_1}$} (A0_2);
				\path (A1_1) edge [->]node [auto,swap] {$\scriptstyle{\sqsubseteq}$} (A0_2);
				\path (A0_0) edge [->]node [auto,swap] {$\scriptstyle{\sqsubseteq}$} (A1_1);
			\end{tikzpicture}
		\]
Then any 
$\Sigma_1$-formula $\psi(\vec{a})$ with $\vec{a}\in\mathcal{M}^{<\omega}$ realized in
$\mathcal{P}$ holds in $\mathcal{Q}$, and is therefore reflected to $\mathcal{M}$.
We are done by Tarski-Vaught's criterion.

\item[(\ref{rmk:itm2})]
Assume $\mathcal{N}\sqsubseteq\mathcal{P}$ for some model of $T_\forall$ $\mathcal{P}$.
Let $\Delta_0(\mathcal{P})$ be the atomic diagram
of $\mathcal{P}$ in the signature $\tau\cup\mathcal{P}\cup\mathcal{M}$ and 
$\Delta_0(\mathcal{M})$ be the atomic diagram
of $\mathcal{M}$ in the same signature\footnote{We are considering $\mathcal{P}\cup\mathcal{M}$ as the union of the domains of the structure 
$\mathcal{P},\mathcal{M}$ amalgamated over $\mathcal{N}$;
in particular we add a new constant for each element of  
$\mathcal{P}\setminus\mathcal{N}$, a new constant for each element of 
$\mathcal{M}\setminus\mathcal{N}$, a new constant for each element of $\mathcal{N}$.}.

\begin{claim}
$T_\forall\cup\Delta_0(\mathcal{P})\cup\Delta_0(\mathcal{M})$ is a consistent 
$\tau\cup\mathcal{M}\cup\mathcal{P}$-theory.
\end{claim}
\begin{proof}
Assume not. Find $\vec{a}\in (\mathcal{P}\setminus\mathcal{N})^{<\omega}$,  
$\vec{b}\in (\mathcal{M}\setminus\mathcal{N})^{<\omega}$,
$\vec{c}\in \mathcal{N}^{<\omega}$
 and $\tau$-formulae $\psi_0(\vec{x},\vec{z})$, 
$\psi_1(\vec{y},\vec{z})$ such that: 
\begin{itemize}
\item
$\psi_0(\vec{a},\vec{c})\in \Delta_0(\mathcal{P})$,
\item
$\psi_1(\vec{b},\vec{c})\in \Delta_0(\mathcal{M})$,
\item
$T\cup\bp{\psi_0(\vec{a},\vec{c}),\psi_1(\vec{b},\vec{c})}$
is inconsistent.
\end{itemize} 
Then
\[
T\vdash\neg\psi_0(\vec{a},\vec{c})\vee\neg\psi_1(\vec{b},\vec{c}).
\]
Since the constants appearing in $\vec{a},\vec{b},\vec{c}$ 
are never appearing in sentences of $T$,
we get that
\[
T\vdash\forall\vec{z}\,(\forall\vec{x}\neg\psi_0(\vec{x},\vec{z}))\vee
(\forall\vec{y}\neg\psi_1(\vec{y},\vec{z})).
\]
Since $\mathcal{P}$ models $T_\forall$, and 
\[
\mathcal{P}\models\psi_0(\vec{x},\vec{z})[\vec{x}/\vec{a},\vec{z}/\vec{c}],
\]
we get that
\[
\mathcal{P}\models\forall\vec{y}\neg\psi_1(\vec{y},\vec{c}).
\]
Therefore 
\[
\mathcal{N}\models\forall\vec{y}\neg\psi_1(\vec{y},\vec{c})
\]
being a substructure of $\mathcal{P}$, and so does $\mathcal{M}$
since $\mathcal{N}\prec_1\mathcal{M}$.
This contradicts $\psi_1(\vec{b},\vec{c})\in \Delta_0(\mathcal{M})$.
\end{proof}

If $\bar{\mathcal{Q}}$ is a model realizing 
$T_\forall\cup\Delta_0(\mathcal{P})\cup\Delta_0(\mathcal{M})$, and $\mathcal{Q}$ is the 
$\tau$-structure obtained forgetting the constant symbols not in $\tau$,
we get that:
\begin{itemize} 
\item
$\mathcal{P}$ and $\mathcal{M}$ are both substructures of $\mathcal{Q}$
containing $\mathcal{N}$ as a common substructure;
\item 
$\mathcal{N}\prec_1\mathcal{M}\prec_1\mathcal{Q}$, 
since $\mathcal{Q}$ realizes $T_\forall$ and $\mathcal{M}$ is $T_\forall$-ec.
\end{itemize}
We can now conclude that if a $\Sigma_1$-formula $\psi(\vec{c})$ for $\tau\cup\mathcal{N}$
with parameters in 
$\mathcal{N}$ holds in $\mathcal{P}$, it holds in $\mathcal{Q}$ as well (since
$\mathcal{Q}\sqsupseteq\mathcal{P}$), and therefore also in 
$\mathcal{N}$ (since $\mathcal{N}\prec_1\mathcal{Q}$).

\item[(\ref{rmk:itm3})] Observe that
for all $\vec{b}\in\mathcal{M}^{<\omega}$, $\exists \vec{y}\,\psi(\vec{b},\vec{y},\vec{a})$ holds in $\mathcal{N}$, and therefore in
$\mathcal{M}$, since $\mathcal{M}$ is $T$-ec; hence 
$\mathcal{M}\models \forall \vec{x}\exists \vec{y}\psi(\vec{x},\vec{y},\vec{a})$.

\item[(\ref{rmk:itm4})] 
First of all note that $\mathcal{M}$ is $S$-ec since $S\supseteq T$ 
(by  Fact \ref{fac:presECmod}).
By Lemma~\ref{lem:keylemembed} (applied to $S_\forall+\psi$ and $\mathcal{M}$)
any
$\Pi_2$-sentence $\psi$ for $\tau$ which holds in some model of $S_\forall$
holds in some model of $S_\forall$ which is a superstructure of
$\mathcal{M}$. Now apply \ref{rmk:itm3}.
\end{description}
\end{proof}

In particular a structure  is $T$-ec if and only if it is $T_\forall$-ec, and 
a $T$-ec structure realizes all $\Pi_2$-sentences which are consistent with its 
$\Pi_1$-theory.

We now show that any structure $\mathcal{M}$ 
can always be extended to a $T$-ec structure for 
any $T$ which is not separated from the $\Pi_1$-theory of $\mathcal{M}$.

\begin{lemma}\label{lem:exTecmod}\cite[Lemma 3.2.11]{TENZIE}
Given a first order $\tau$-theory $T$, 
any model of $T_\forall$
can be extended to a $\tau$-superstructure 
which is $T$-ec.
\end{lemma}
\begin{proof}
Given a model $\mathcal{M}$ of $T$, 
we construct an ascending chain of $T_\forall$-models as follows.
Enumerate all quantifier free $\tau$-formulae as
$\bp{\phi_\alpha(y,\vec{x}_\alpha):\alpha<|\tau|}$. 
Let $\mathcal{M}_0=\mathcal{M}$ have size $\kappa\geq|\tau|+\aleph_0$. 
Fix also some enumeration 
\begin{align*}
\pi:&\kappa\to|\tau|\times\kappa^2\\
&\alpha\mapsto(\pi_0(\alpha),\pi_1(\alpha),\pi_2(\alpha))
\end{align*} 
such that $\pi_2(\alpha)\leq\alpha$ for all $\alpha<\kappa$
and for each $\xi<|\tau|$, and $\eta,\beta<\kappa$
there are unboundedly many $\alpha<\kappa$ such that $\pi(\alpha)=(\xi,\eta,\beta)$.

Let now 
$\mathcal{M}_\eta$ with enumeration   
$\bp{\vec{m}^\xi_\eta:\xi<\kappa}$ of $\mathcal{M}_\eta^{<\omega}$ be given for all $\eta\leq\beta$.
If $\mathcal{M}_\beta$ is $T$-ec, stop the construction.
Else check whether $T_\forall\cup\Delta_0(\mathcal{M}_\beta)\cup\bp{\exists y\phi_{\pi_0(\alpha)}(y,\vec{m}^{\pi_1(\alpha)}_{\pi_2(\alpha)})}$
is a consistent $\tau\cup\mathcal{M}_\beta$-theory; if so let $\mathcal{M}_{\beta+1}$ have size $\kappa$ and
realize this theory.
At limit stages $\gamma$, let $\mathcal{M}_\gamma$ be the direct limit of the chain of $\tau$-structures 
$\bp{\mathcal{M}_\beta:\beta<\gamma}$. Then all $\mathcal{M}_\xi$ are models of $T_\forall$, and at some stage $\beta\leq\kappa$
$\mathcal{M}_\beta$ is $T_\forall$-ec (hence also $T$-ec), since all existential $\tau$-formulae with parameters in some 
$\mathcal{M}_\eta$
will be considered along the construction, and realized along the way if this is possible, 
and all $\mathcal{M}_\eta$ are always models of 
$T_\forall$ (at limit stages the ascending chain of $T_\forall$-models remains a $T_\forall$-model).
\end{proof}

Compare the above construction with the standard consistency proofs of bounded forcing axioms as given for example in
\cite[Section 2]{ASPBAGAPAL2001}.
In the latter case to preserve $T_\forall$ at limit stages we use iteration theorems\footnote{Assume $G$ is $V$-generic for a forcing
which is a 
limit of an iteration of length $\omega$ of forcings 
$\bp{P_n:n<\omega}$. In general
$H_{\omega_2}^{V[G]}$ is not given by the union of
 $H_{\omega_2}^{V[G\cap P_n]}$,
hence a subtler argument is needed to maintain that 
$H_{\omega_2}^{V[G]}$ preserves $T_\forall$.}.

\subsection{The Kaiser hull of a first order theory}

The Kaiser Hull of a theory $T$ describes the smallest elementary class containing all the
 ``generic'' structures for $T$. For most theories $T$ the models of the respective
 Kaiser hulls  realize 
exactly all  $\Pi_2$-sentences which are consistent with the
universal fragment of any extension of
$T$.

\begin{definition}\cite[Lemma 3.2.12, Lemma 3.2.13]{TENZIE}
Given a theory $T$ in a signature $\tau$, its Kaiser hull 
$\mathrm{KH}(T)$ is given by the
$\Pi_2$-sentences of $\tau$ which holds in all $T$-ec structures.
\end{definition}

\begin{definition}
A $\tau$-theory $T$ is $\Pi_n$-complete, if it is consistent and
for any $\Pi_n$-sentence 
either $\phi\in T$ or 
$\neg\phi\in T$.
\end{definition}

By Proposition \ref{prop:pi2satKaiserhull}.\ref{rmk:itm4} we get:

\begin{fact}\label{fac:charKaihull}
Given a $\Pi_1$-complete first order $\tau$-theory $T$, its Kaiser Hull is
a $\Pi_2$-complete $\tau$-theory defined by the request that
for any $\Pi_2$-sentence $\psi$
\[
\psi\in \mathrm{KH}(T)\quad \text{ if and only if }\quad \bp{\psi}\cup T_\forall\text{ is consistent}.
\] 
\end{fact}

In particular any model of the Kaiser hull of a $\Pi_1$-complete
$T$ realizes simultaneously 
all $\Pi_2$-sentences which are individually consistent with $T_\forall$.

For theories $T$ of interests to us their Kaiser hull can be described in the same terms, 
but the proof is much more delicate.
We start with the following weaker property which holds for arbitrary theories:
\begin{fact}\label{fac:charkaihullnonpi1comp}
Given a $\tau$-theory $T$, its Kaiser hull $\mathrm{KH}(T)$ contains
the set of $\Pi_2$-sentences $\psi$ for $\tau$ such that for all complete $S\supseteq T$,
$S_\forall\cup\bp{\psi}$ is consistent.
\end{fact}

\begin{proof}
Assume $\psi$ is a $\Pi_2$-sentence such that 
for all complete $S\supseteq T$,
$S_\forall\cup\bp{\psi}$ is consistent. We must show that
$\psi$ holds in all $T$-ec models.

Fix $\mathcal{M}$ an existentially closed model for $T$ (it exists by Lemma~\ref{lem:exTecmod}); 
we must show that
$\mathcal{M}\models \psi$.
Let $\mathcal{N}\sqsupseteq\mathcal{M}$ be a model of $T$ and $S$ be the 
$\tau$-theory of $\mathcal{N}$. Then $S$ is a complete theory and
$\mathcal{M}\models S_{\forall}$ since 
$\mathcal{M}\prec_1\mathcal{N}$ (being $T$-ec).
Since $S\supseteq T$, $\mathcal{M}$ is also $S$-ec (by Fact~\ref{fac:presECmod}).
Since  $S_\forall\cup\bp{\psi}$ is consistent, and $S_\forall$ is 
$\Pi_1$-complete, we obtain that
$\mathcal{M}$ models $\psi$, being an $S_\forall$-ec model, and using Fact \ref{fac:charKaihull}. 
\end{proof}

We will show in Lemma~\ref{fac:proofthm1-2}
that the set of $\Pi_2$-sentences described in the Fact provides an equivalent characterization of the Kaiser
hull for many theories admitting a model companion, among which the axiomatizations of set theory considered in this paper. 


\subsection{Model completeness}

It is possible (depending on the choice of the 
theory $T$)
that there are models of the Kaiser hull of $T$ which are not 
$T$-ec.
Robinson has come up with two model theoretic properties 
(model completeness and model companionship)
which describe the case in which the models of the Kaiser hull of 
$T$ are exactly the class of 
$T$-ec models (even in case $T$ is not a complete theory).

\begin{definition}\label{def:modcompl}
A $\tau$-theory $T$ is \emph{model complete} if for all $\tau$-models $\mathcal{M}$ and $\mathcal{N}$ of $T$ we have that 
$\mathcal{M} \sqsubseteq \mathcal{N}$ implies $\mathcal{M} \prec \mathcal{N}$. 
\end{definition}

Remark that theories admitting quantifier elimination are automatically model complete.
On the other hand model complete theories need not be complete\footnote{For example the theory of 
algebraically closed fields is model complete, but algebraically closed fields of different characteristics are 
elementarily inequivalent.}.
However for theories $T$ which are $\Pi_1$-complete,
model completeness entails completeness:
any two models of a $\Pi_1$-complete, model complete $T$ share the same 
$\Pi_1$-theory, therefore 
if $T_1\supseteq T$ and $T_2\supseteq T$ with $\mathcal{M}_i$ a model of $T_i$, we can suppose (by Lemma 
\ref{lem:biembequivequnivth}) that $\mathcal{M}_1\sqsubseteq\mathcal{M}_2$. Since they are both models of $T$,
model completeness entails that $\mathcal{M}_1\prec \mathcal{M}_2$.

\begin{lemma}\cite[Lemma 3.2.7]{TENZIE}\label{lem:robtest}
(Robinson's test) Let $T$ be a $\tau$-theory. The following are equivalent:
\begin{enumerate}[(a)]
\item \label{lem:robtest-1} $T$ is model complete. 
\item \label{lem:robtest-2} Any model of $T$ is $T$-ec.
\item \label{lem:robtest-4} Each \emph{existential} $\tau$-formula $\phi(\vec{x})$ in free variables $\vec{x}$ 
is $T$-equivalent to a universal $\tau$-formula $\psi(\vec{x})$ in the same free variables. 
\item \label{lem:robtest-3} Each $\tau$-formula $\phi(\vec{x})$ in free variables $\vec{x}$ 
is $T$-equivalent to a universal $\tau$-formula $\psi(\vec{x})$ in the same free variables. 
\end{enumerate}
\end{lemma}
Remark that \ref{lem:robtest-3} (or \ref{lem:robtest-4}) shows that being a model complete $\tau$-theory $T$ is expressible by
a $\Delta_0(\tau,T)$-property in any model of $\ZFC$, hence it is absolute with respect 
to forcing.

\begin{proof}
\emph{}

\begin{description}
\item[\ref{lem:robtest-1} implies \ref{lem:robtest-2}]
Immediate.
\item[\ref{lem:robtest-2} implies \ref{lem:robtest-4}]
Fix an existential formula $\phi(\vec{x})$ in free variables $x_1,\dots,x_n$.
If $\phi(\vec{x})$ is not consistent with $T$ it is $T$-equivalent to the 
trivial formula $\forall y(y\neq y)$ in free variables $\vec{x}$.
Hence we may assume that $T\cup\phi(\vec{x})$ is a consistent theory.
Let $\vec{c}=(c_1,\dots,c_n)$ be a finite set of new constant symbols.
Then $T\cup\phi(\vec{c})$ is a consistent $\tau\cup\bp{c_1,\dots,c_n}$-theory.

Let $\Gamma$ be the set of universal $\tau$-formulae $\theta(\vec{x})$ such that
\[
T\vdash\forall\vec{x}\,(\phi(\vec{x})\rightarrow \theta(\vec{x})).
\]
Note that $\Gamma$ is closed under finite conjunctions and disjunctions.
Let $\Gamma(\vec{c})=\bp{\theta(\vec{c}):\, \theta(\vec{x})\in \Gamma}$.
Note that $T\cup\Gamma(\vec{c})$ is a consistent $\tau\cup\bp{c_1,\dots,c_n}$-theory,
since it holds in any $\tau\cup\bp{c_1,\dots,c_n}$-model of $T\cup\phi(\vec{c})$.

It suffices to prove 
\begin{equation}\label{eqn:keyeq2-->3}
T\cup\Gamma(\vec{c})\models\phi(\vec{c});
\end{equation}
if this is the case, by compactness, a finite subset $\Gamma_0(\vec{c})$ of $\Gamma(\vec{c})$ is such that
\[
T\cup\Gamma_0(\vec{c})\models\phi(\vec{c});
\]
letting $\bar{\theta}(\vec{x}):=\bigwedge\bp{\psi(\vec{x}): \psi(\vec{c})\in\Gamma_0(\vec{c})}$,
the latter 
gives that 
\[
T\models \forall\vec{x}\,(\bar{\theta}(\vec{x})\rightarrow\phi(\vec{x}))
\]
(since the constants $\vec{c}$ do not appear in $T$). 

$\bar{\theta}(\vec{x})\in \Gamma$ is a universal formula witnessing \ref{lem:robtest-4} for $\phi(\vec{x})$.

So we prove (\ref{eqn:keyeq2-->3}):
\begin{proof}
Let $\mathcal{M}$ be a $\tau\cup\bp{c_1,\dots,c_n}$-model of 
$T\cup\Gamma(\vec{c})$. We must show that
$\mathcal{M}$ models $\phi(\vec{c})$.

The key step is to prove the following:

\begin{claim}
$T\cup\Delta_0(\mathcal{M})\cup\bp{\phi(\vec{c})}$
is consistent (where $\Delta_0(\mathcal{M})$ is the $\tau\cup\bp{c_1,\dots,c_n}$-atomic 
diagram of $\mathcal{M}$ in signature $\tau\cup\bp{c_1,\dots,c_n}\cup\mathcal{M}$).
\end{claim}

Assume the Claim holds and let $\mathcal{N}$ realize the above theory.
Then 
\[
\mathcal{M}\sqsubseteq\mathcal{N}\restriction(\tau\cup\bp{c_1,\dots,c_n}).
\]
Hence 
\[
\mathcal{M}\restriction\tau\sqsubseteq\mathcal{N}\restriction\tau.
\]
By \ref{lem:robtest-2} 
\[
\mathcal{M}\restriction\tau\prec_1\mathcal{N}\restriction\tau.
\]

Now let $b_1,\dots,b_n\in\mathcal{M}$ be the interpretations of $c_1,\dots,c_n$
in the $\tau\cup\bp{c_1,\dots,c_n}$-structure $\mathcal{M}$.
Then 
\[
\mathcal{N}\restriction\tau\models\phi(x_1,\dots,x_n)[b_1,\dots,b_n].
\]
Since $\phi(\vec{x})$ is $\Sigma_1$ for $\tau$ and $b_1,\dots,b_n\in \mathcal{M}$, we get that
\[
\mathcal{M}\restriction\tau\models\phi(x_1,\dots,x_n)[b_1,\dots,b_n],
\]
hence
\[
\mathcal{M}\models\phi(c_1,\dots,c_n),
\]
and we are done.

So we are left with the proof of the Claim.
\begin{proof}
Let $\psi(\vec{x},\vec{y})$ be a quantifier free $\tau$-formula such that
$\psi(\vec{c},\vec{a})\in\Delta_0(\mathcal{M})$ for some $\vec{a}\in\mathcal{M}$.

Clearly $\mathcal{M}$ models
$\exists \vec{y}\psi(\vec{c},\vec{y})$.

Then the universal formula $\neg\exists \vec{y}\psi(\vec{c},\vec{y})\not\in\Gamma(\vec{c})$, since
$\mathcal{M}$ models its negation and $\Gamma(\vec{c})$ at the same time.

This gives that 
\[
T\not\vdash\forall\vec{x}\,(\phi(\vec{x})\rightarrow \neg\exists \vec{y}\psi(\vec{x},\vec{y})),
\]
i.e.
\[
T\cup \bp{\exists \vec{x}\,[\phi(\vec{x})\wedge\exists \vec{y}\psi(\vec{x},\vec{y})]}
\]
is consistent.

We conclude that
\[
T\cup\bp{\phi(\vec{c})\wedge\psi(\vec{c},\vec{a})}
\]
is consistent for any tuple $a_1,\dots,a_k\in\mathcal{M}$ and formula $\psi$ such that
$\mathcal{M}$ models $\psi(\vec{c},\vec{a})$
(since $\vec{c},\vec{a}$ are constants never appearing in the formulae of $T$).

This shows that $T\cup\Delta_0(\mathcal{M})\cup\bp{\phi(\vec{c})}$ is consistent.
\end{proof}

(\ref{eqn:keyeq2-->3}) is proved.
\end{proof}

\item[\ref{lem:robtest-4} implies \ref{lem:robtest-3}]
We prove by induction on $n$ that $\Pi_n$-formulae and $\Sigma_n$-formulae are $T$-equivalent to a $\Pi_1$-formula.

\ref{lem:robtest-4} gives the base case $n=1$ of the induction for $\Sigma_1$-formulae
and (trivially) for $\Pi_1$-formulae. 

Assuming we have proved the implication for all $\Sigma_{n}$ formulae for some fixed 
$n>0$, we obtain it  for $\Pi_{n+1}$-formulae $\forall\vec{x}\psi(\vec{x},\vec{y})$ (with $\psi(\vec{x},\vec{y})$
$\Sigma_n$)
applying the inductive assumptions to $\psi(\vec{x},\vec{y})$; 
next we 
observe that a $\Sigma_{n+1}$-formula is equivalent to the negation of a
$\Pi_{n+1}$-formula, which is in turn equivalent to the negation of a universal formula (by what we already argued),
which is equivalent to an existential formula, and thus equivalent to a universal formula (by \ref{lem:robtest-4}).

\item[\ref{lem:robtest-3} implies \ref{lem:robtest-1}]
By \ref{lem:robtest-3} every formula is $T$-equivalent both to a universal formula and to an existential formula (since its negation is $T$-equivalent to a universal formula).

This gives that $\mathcal{M}\prec \mathcal{N}$ whenever $\mathcal{M}\sqsubseteq\mathcal{N}$ are models of $T$, since truth of universal formulae is inherited by substructures, while truth of existential formulae pass to superstructures. 
\end{description}
\end{proof}

We will also need the following:

\begin{fact}\label{fac:proofthm1}
Let $\tau$ be a signature and 
$T$ a model complete $\tau$-theory. 
Let $\sigma\supseteq \tau$ be a signature and 
$T^*\supseteq T$ a $\sigma$-theory such that every
$\sigma$-formula is $T^*$-equivalent to a $\tau$-formula. Then $T^*$ is model complete.
\end{fact}
\begin{proof}
By the model completeness of $T$ and the assumptions on $T^*$ we get that every $\sigma$-formula is equivalent
to a $\Pi_1$-formula for $\tau\subseteq\sigma$. We conclude by Robinson's test.
\end{proof}

Later on we will show that in most cases model complete theories maximize the family of $\Pi_2$-sentences compatible with any $\Pi_1$-completion of their universal fragment.
This will be part of a broad family of 
properties for first order theories which require a new concept in order to be properly 
formulated, that of model companionship.

\subsection{Model companionship}
Model completeness comes in pairs with another fundamental concept which generalizes to arbitrary first order 
theories the relation existing between algebraically closed fields and commutative rings without zero-divisors. As a matter of fact, the case described below
occurs when $T^*$ is the theory of algebraically closed fields
and 
$T$ is the theory of commutative rings with no zero divisors.

\begin{definition}\label{def:dmodcompship}
Given two theories $T$ and $T^*$ in the same language $\tau$, 
$T^*$ is the \emph{model companion} of $T$ if the following conditions holds:
\begin{enumerate}
\item Each model of $T$ can be extended to a model of $T^*$.
\item Each model of $T^*$ can be extended to a model of $T$.
\item $T^*$ is model complete. 
\end{enumerate}
\end{definition}

Different theories can have the same model companion, for example the theory of fields 
and the theory of commutative rings with 
no zero-divisors which are not fields both have the theory of algebraically closed fields  
as their model companion.

\begin{theorem}\cite[Thm 3.2.14]{TENZIE}
\label{thm:modcompletionchar}
Let $T$ be a first order theory. If its model companion $T^*$ exists, then
\begin{enumerate}
\item \label{thm:modcompletionchar-1} $T_{\forall} = T^*_{\forall}$.
\item \label{thm:modcompletionchar-2} $T^*$ is the theory of the existentially closed models of $T_{\forall}$. 
\end{enumerate}
\end{theorem}
\begin{proof}
\emph{}

\begin{enumerate}
\item By Lemma~\ref{lem:keylemembed}.
\item By  Robinson's test \ref{lem:robtest} $T^*$ is the theory realized exactly by the $T^*$-ec models; 
by Proposition \ref{prop:pi2satKaiserhull}(\ref{rmk:itm0}) $\mathcal{M}$ is $T^*$-ec if and only if it is
$T^*_\forall$-ec;
by (\ref{thm:modcompletionchar-1}) $T^*_\forall=T_\forall$.
\end{enumerate}
\end{proof}

An immediate by-product of the above Theorem is that
the model companion of a theory does not necessarily exist, but, if it does, it is unique
and is its Kaiser hull.

\begin{theorem} \cite[Thm. 3.2.9]{TENZIE}\label{thm:uniqmodcompan}
Assume $T$ has a model companion $T^*$. Then
$T^*$ is axiomatized by its $\Pi_2$-consequences and is the Kaiser hull of $T_\forall$.

Moreover $T^*$ is the unique model companion of $T$ and is characterized by the property of 
being the unique model complete theory $S$ such that $S_\forall=T_\forall$.
\end{theorem}
\begin{proof}
For quantifier free formulae $\psi(\vec{x},\vec{y})$ and $\phi(\vec{x},\vec{z})$
the assertion
\[
\forall\vec{x}\,[\exists\vec{y}\psi(\vec{x},\vec{y})\leftrightarrow\forall\vec{z}\phi(\vec{x},\vec{z})]
\]
is a $\Pi_2$-sentence.

Let $T^{**}$ be the theory given by the $\Pi_2$-consequences of $T^*$.

Since $T^*$ is model complete, by Robinson's test \ref{lem:robtest}\ref{lem:robtest-4},  
for any $\Sigma_1$-formula $\exists\vec{y}\psi(\vec{x},\vec{y})$
there is a universal formula $\forall\vec{z}\phi(\vec{x},\vec{z})$ such that
\[
\forall\vec{x}\,[\exists\vec{y}\psi(\vec{x},\vec{y})\leftrightarrow\forall\vec{z}\phi(\vec{x},\vec{z})]
\]
is in $T^{**}$.

Again by Robinson's test \ref{lem:robtest}\ref{lem:robtest-4} $T^{**}$ is model complete.

Now assume $S$ is a model complete theory such that 
$S_\forall=T_\forall$.
Clearly $T^{*}_\forall=T_\forall=S_\forall$. 
By Robinson's test \ref{lem:robtest}\ref{lem:robtest-2} and Proposition \ref{prop:pi2satKaiserhull}(\ref{rmk:itm0}),
$S_\forall$ holds exactly in the $T_\forall$-ec models, but these are exactly the models of $T^*$.
Hence $T^*=S$.

This shows that any model complete theory is axiomatized by its $\Pi_2$-consequences, 
that the model companion $T^*$ of $T$ is unique, that $T^*$ is also the Kaiser hull of $T$
(being axiomatized by the $\Pi_2$-sentences which hold in all
$T$-ec-models), and is characterized by the property of being the unique model complete
theory $S$ such that $T_\forall=S_\forall$.
\end{proof}

Thm. \ref{thm:uniqmodcompan} provides an
equivalent characterization of model companion theories
(which is expressible by a $\Delta_0$-property in parameters $T$ and $T^*$, hence absolute for transitive models of $\ZFC$). 

Note also that Robinson's test \ref{lem:robtest}\ref{lem:robtest-3} 
gives an explicit axiomatization of a model complete theory $T$:
\begin{fact}\label{fact:charmodcompl*}
Assume $T$ is a model complete $\tau$-theory. 
Let $\psi\mapsto\theta_\psi^T$ be a function assigning to each
$\Sigma_1$-formula $\psi(\vec{x})$ for $\tau$ a $\Pi_1$-formula
$\theta_\psi^T(\vec{x})$ which is $T$-equivalent to $\psi(\vec{x})$.

Then $T$ is axiomatized by $T_\forall$ and the $\Pi_2$-sentences
\[
\AX_\psi^T\equiv \forall\vec{x}(\psi(\vec{x})\leftrightarrow \theta_\psi^T(\vec{x}))
\]
as $\psi(\vec{x})$ ranges over the $\Sigma_1$-formulae for $\tau$.
\end{fact}
\begin{proof}
First of all 
\[
T^*=\bp{\AX_\psi^T:\psi\text{ a $\tau$-formula}}
\]
is a model complete theory, since $T^*$ satisfies Robinson's test \ref{lem:robtest}\ref{lem:robtest-3}.
Let $S=T^*+T_\forall$. Note that $S$ is also model complete (by Robinson's test \ref{lem:robtest}\ref{lem:robtest-3}).
Moreover $S\subseteq T$ (since $\AX_\psi^T\in T$ for all $\Sigma_1$-formulae $\psi$), and
$S_\forall\supseteq T_\forall$ (since $T_\forall$ is certainly among the universal consequences of $S$).
We conclude that $S_\forall=T_\forall$. Therefore $S$ is the model companion of $T$.
$S=T$ by uniqueness of the model companion.
\end{proof}

We use the following criteria for model companionship in the proofs of Theorems \ref{Thm:mainthm-1}, \ref{thm:modcompanHomega1}, \ref{Thm:mainthm-1ter}.

\begin{lemma}\label{fac:proofthm1-2}
Let $T,T_0$ be $\tau$-theories 
with $T_0$ model complete.
Assume that for every $\Pi_1$-sentence $\theta$ for $\tau$
$T+\theta$ is consistent if and only if so is $T_0+\theta$.
%
Then:
\begin{enumerate}
\item \label{fac:proofthm1-2-a}
$T^*=T_0+T_\forall$ is the model companion of $T$.
\item \label{fac:proofthm1-2-b}
$T^*$ is axiomatized by the the set of $\Pi_2$-sentences $\psi$ for $\tau$ such that 
$S_\forall\cup\bp{\psi}$ is consistent for all $\Pi_1$-complete $S\supseteq T$.
\item \label{fac:proofthm1-2-c}
$T^*$ is axiomatized by the the set of $\Pi_2$-sentences $\psi$ for $\tau$ such that for all 
universal $\tau$-sentences $\theta$
$T_\forall+\theta+\psi$ is consistent if and only if so is $T+\theta$.
\end{enumerate}
\end{lemma}

\begin{proof}
By assumption $T_0$ is consistent with any finite subset of $T_\forall$; hence, by compactness,
$T^*=T_0+T_\forall$ is consistent. 
By Fact \ref{fac:proofthm1} $T^*$ is model complete. 
\begin{enumerate}
\item
We need to show that any model of $T^*$ embeds into a model of $T$ and conversely.

Assume $\mathcal{N}$ models $T^*$.
Then $\mathcal{N}$ models $T_\forall$. By Lemma \ref{lem:keylemembed}
there exists $\mathcal{M}\sqsupseteq \mathcal{N}$
which models $T$.

Conversely let $\mathcal{M}$ model $T$ and
$S$ be the $\tau$-theory of $\mathcal{M}$.
By 
assumption (and compactness) there is $\mathcal{N}$ which models $T_0+S_\forall$
(but this $\mathcal{N}$ may not be a superstructure of $\mathcal{M}$).
Let $S^*$ be the $\tau$-theory of $\mathcal{N}$.
Then $S^*_\forall=S_\forall$, since $S_\forall$ and $S^*_\forall$ 
are $\Pi_1$-complete theories with $S^*_\forall\supseteq S_\forall$.
Moreover $S^*\supseteq T^*$, since $S_\forall\supseteq T_\forall$.

\begin{claim}
The $\tau\cup\mathcal{M}$-theory 
$S^*\cup\Delta_0(\mathcal{M})$
is consistent. 
\end{claim}

Assume the Claim holds, then $\mathcal{M}$ is a $\tau$-substructure of a model of 
$S^*\supseteq T^*$ and we are done.

\begin{proof}
If not there is $\psi(\vec{a})\in \Delta_0(\mathcal{M})$ such that 
$S^*\cup\bp{\psi(\vec{a})}$ is inconsistent.
This gives that
\[
S^*\vdash \neg \psi(\vec{a}).
\]
Since none of the constant in $\vec{a}$ occurs in $\tau$, we get that
\[
S^*\vdash \forall\vec{x}\neg \psi(\vec{x}),
\]
i.e. $\forall\vec{x}\neg \psi(\vec{x})\in S^*_\forall=S_\forall$.
But $\mathcal{M}$  models $S_\forall$ and
$\forall\vec{x}\neg \psi(\vec{x})$ fails in $\mathcal{M}$; a contradiction.
\end{proof}
\item
Assume $\psi\in T^*$ and $S$ is a $\Pi_1$-complete extension of $T$, 
we must show that
$S_\forall+\psi$ is consistent:
by assumption there is $\mathcal{N}$ which models
$T_0+S_\forall=T_0+T_\forall+S_\forall=T^*+S_\forall$, and we are done.
Conversely assume $R_\forall+\psi$ is consistent whenever 
$R$ is a $\Pi_1$-complete extension of $T$. We must show that $\psi\in T^*$:
pick $\mathcal{M}$ model of $T^*$ and let $S$ be its theory.
The assumptions of the Lemma (and compactness) grant that $T+S_\forall$ is consistent.
Since $S$ is complete $S_\forall$ is the $\Pi_1$-fragment of $T+S_\forall$.
Hence $S_\forall+\psi$ is consistent, by our assumption on $\psi$.
Therefore $\mathcal{M}\models\psi$ by Proposition \ref{prop:pi2satKaiserhull}.
\item Left to the reader (as the previous item, modulo compactness arguments).
\end{enumerate}
\end{proof}

\begin{remark}\label{rmk:keyrmkcharkaihull}
We do not know whether the characterization of the model companion of $T$ given in 
Lemma~\ref{fac:proofthm1-2}(\ref{fac:proofthm1-2-c})
can be proved for \emph{all} theories $T$ admitting a model companion: following the notation of 
the Lemma, it is conceivable that some $\tau$-theory $T$ has a model companion $T^*$, but there is 
some universal $\tau$-sentence $\theta$ such that for any 
model $\mathcal{M}$ of $T+\theta$ any superstructure of $\mathcal{M}$ which models
$T^*$ kills the truth of $\theta$.
In this case some $\Pi_2$-sentence in the Kaiser
hull of $T$ is inconsistent with the universal fragment of $T+\theta$.

Note also that if $T^*$ is the model companion of $T$ and $\theta$ is a universal sentence
such that $T^*+\theta$ is consistent, so is $T+\theta$: if $\mathcal{M}\models T^*+\theta$
there is a superstructure $\mathcal{N}$ of $\mathcal{M}$ which models $T$ (since 
$T^*$ is the model companion of $T$).
Now $\mathcal{M}\prec _1\mathcal{N}$, since $\mathcal{M}$ is $T$-ec.
Hence $\mathcal{N}\models \theta$.

\end{remark}

\subsection{Is model companionship a tameness notion?}\label{subsec:tameness-modcompan}

As we already outlined in the introduction
model completeness and model companionship are ``tameness'' notion for first order theories
which must be handled with care. We spell out the details in this small section.

\begin{proposition}\label{prop:quantelimallthe}
Given a signature $\tau$ consider the signature $\tau^*$ which adds an $n$-ary predicate
symbol $R_\phi$ for any $\tau$-formula $\phi(x_1,\dots,x_n)$ with displayed free variables.

Let $T_{\tau}$ be the following $\tau^*$-theory:
\begin{itemize}
\item
$\forall\vec x\,(\phi(\vec{x})\leftrightarrow R_\phi(\vec{x}))$ for all quantifier free $\tau$-formulae $\phi(\vec{x})$,
\item
$\forall\vec x\,[R_{\phi\wedge\psi}(\vec{x})\leftrightarrow (R_\psi(\vec{x})\wedge R_\phi(\vec{x}))]$
for all $\tau$-formulae $\phi(\vec{x}),\psi(\vec{x})$,
\item
$\forall\vec x\,[R_{\neg\phi}(\vec{x})\leftrightarrow \neg R_\phi(\vec{x})]$
for all $\tau$-formulae $\phi(\vec{x})$,
\item
$\forall\vec x\,[\exists yR_{\phi}(y,\vec{x})\leftrightarrow R_{\exists y\phi}(\vec{x})]$
for all $\tau$-formulae $\phi(y,\vec{x})$.
\end{itemize}

Then any $\tau$-structure $\mathcal{N}$ admits a unique extension to 
a $\tau^*$-structure $\mathcal{N}^*$ which models $T_\tau$.
Moreover every $\tau^*$-formula is $T_\tau$-equivalent to 
an atomic $\tau^*$-formula. In particular for any $\tau$-model $\mathcal{N}$, the algebras of
its $\tau$-definable subsets and of the $\tau^*$-definable subsets of $\mathcal{N}^*$ are the same.

Therefore for any consistent $\tau$-theory $T$, $T\cup T_{\tau}$ is consistent and 
admits quantifier elemination, hence
is model complete.
\end{proposition}

\begin{proof}
By an easy induction one can prove that any $\tau$-formula $\phi(\vec{x})$
is $T_\tau$-equivalent to the atomic $\tau^*$-formula $R_{\phi}(\vec{x})$.

Another simple inductive argument brings that any
$\tau^*$-formula $\phi(\vec{x})$ is $T_\tau$-equivalent to the 
$\tau$-formula obtained by replacing all symbols $R_\psi(\vec{x})$ occurring in
$\phi$ by the $\tau$-formula $\psi(\vec{x})$. Combining these observations together
we get that any $\tau^*$-formula is equivalent to an atomic 
$\tau^*$-formula. 

$T_{\tau}$ forces the $\mathcal{M}^*$-interpretation of any relation symbol $R_\phi(\vec{x})$
in $\tau^*\setminus\tau$ to
be the $\mathcal{M}$-interpretation of the $\tau$-formula $\phi(\vec{x})$ to which it is $T_{\tau}$-equivalent.
\end{proof}

Observe that the expansion of the language from $\tau$ to $\tau^*$ 
behaves well with respect to several model theoretic notions of tameness distinct 
from model completeness: for example $T$ is a \emph{stable} $\tau$-theory if and only if 
so is the $\tau^*$-theory $T\cup T_{\tau}$, the same holds for 
NIP-theories, or for $o$-minimal theories, or for $\kappa$-categorical theories.

The passage from $\tau$-structures to $\tau^*$-structures which model 
$T_{\tau}$ can have effects
on the embeddability relation; for example assume $\mathcal{M}\sqsubseteq\mathcal{N}$ is a non-elementary embedding of 
$\tau$-structures; then $\mathcal{M}^*\not\sqsubseteq\mathcal{N}^*$: if the non-atomic $\tau$-formula
$\phi(\vec{a})$ in parameter $\vec{a}\in \mathcal{M}^{<\omega}$ 
holds in $\mathcal{M}$ and does not hold in
$\mathcal{N}$, the atomic $\tau^*$-formula $R_\phi(\vec{a})$ holds in $\mathcal{M}^*$ and does not hold in
$\mathcal{N}^*$.

However if $T$ is a model complete $\tau$-theory, then for 
$\mathcal{M}\sqsubseteq\mathcal{N}$ $\tau$-models of $T$, we get that
$\mathcal{M}\prec\mathcal{N}$; this entails that $\mathcal{M}^*\sqsubseteq\mathcal{N}^*$, which (by the quantifier
elimination of $T\cup T_{\tau}$) gives that $\mathcal{M}^*\prec\mathcal{N}^*$. 
In particular for a model complete $\tau$-theory $T$ and $\mathcal{M},\mathcal{N}$ 
$\tau$-models of $T$,
$\mathcal{M}\sqsubseteq\mathcal{N}$ if and only if 
$\mathcal{M}^*\sqsubseteq\mathcal{N}^*$.

Let us now investigate the case of model companionship.
If $T$ is the model companion of $S$ with $S\neq T$ in the signature $\tau$, 
$T\cup T_\tau$ and $S\cup T_\tau$ are both model complete theories in the signature $\tau^*$. 
But $T\cup T_\tau$ cannot be the model companion of $S\cup T_\tau$, by uniqueness of the model companion,
since each of these theories is the model companion of itself and they are distinct.
Moreover if $T$ and $S$ are also complete, no $\tau^*$-model of $S\cup T_\tau$ can embed into a 
$\tau^*$-model of $T\cup T_\tau$:
since $T$ is the model companion of 
$S$ and $S\neq T$, $T_\forall=S_\forall$ and there is some $\Pi_2$-sentence $\psi$
$\forall x\exists y\phi(x,y)$ with $\phi$-quantifer free in $T\setminus S$.
Therefore $\forall x\,R_{\exists y\phi}(x) \in (T\cup T_\tau)_\forall \setminus (S\cup T_\tau)_\forall$; 
we conclude by Lemma \ref{lem:biembequivequnivth}, since $T\cup T_\tau$ and $S\cup T_\tau$ are complete, 
hence the above sentence separates $(T\cup T_\tau)_\forall$ from
$(S\cup T_\tau)_\forall$.

\subsection{Summing up}
The results of this section gives that for any $\tau$-theory $T$:
\begin{itemize}
\item The universal fragment of $T$ describes the family of substructures
of models of $T$, and (in most cases, e.g. if $T$ is $\Pi_1$-complete) 
the $T$-ec models realize all $\Pi_2$-sentences which are ``absolutely'' consistent with 
$T_\forall$ 
(i.e. consistent with the universal fragment of any extension of $T$).
\item Model companionship and model completeness describe (almost all) the cases 
in which the family of $\Pi_2$-sentences which are ``absolutely'' consistent with $T$ (as defined in the previous item)
describes the elementary class given by the $T$-ec structures.
\item One can always extend $\tau$ to a signature $\tau^*$ so that $T$ has a 
conservative extension to a $\tau^*$-theory $T^*$ which is model complete, but this process may be completely 
uninformative since it may completely destroy the
substructure relation existing between $\tau$-models of $T$ (unless $T$ is already model complete).
\item On the other hand 
for certain theories $T$ (as the axiomatizations of set theory 
considered in the present paper),
 one can unfold their ``tameness'' 
by carefully extending $\tau$ to a signature $\tau^*$ in which only certain $\tau$-formulae 
are made equivalent to atomic $\tau^*$-formulae. In the new signature $T$ can be extended to a conservative extension $T^*$ which has a model companion $\bar{T}$, while
this process has mild consequences on the $\tau^*$-substructure relation for models of $T^*_\forall$ 
(i.e. for the pairs of interest of $\tau$-models $\mathcal{M}_0\sqsubseteq\mathcal{M}_1$ of a suitable fragment of $T$, 
their unique extensions to $\tau^*$-models $\mathcal{M}^*_i$ are still models of $T^*_\forall$
and maintain that $\mathcal{M}^*_0\sqsubseteq\mathcal{M}^*_1$ also for $\tau^*$).
This gives useful structural information on the web of relations existing between $\tau^*$-models of $T^*_\forall$
(as outlined by Theorems 
\ref{Thm:mainthm-1}, \ref{thm:modcompanHomega1}, \ref{Thm:mainthm-1ter}).
\item
Our conclusion is that model completeness and model companionship are tameness properties of elementary classes 
$\mathcal{E}$ defined by a theory $T$ rather than of the theory $T$ itself: 
these model-theoretic notions outline certain regularity patterns for the substructure relation on
models of $\mathcal{E}$, patterns which may be unfolded only when passing to a signature distinct 
from the one in which
$\mathcal{E}$ is first axiomatized (much the same way as it occurs for Birkhoff's 
characterization of algebraic varieties in terms of universal theories).

%
\item The results of the present paper shows that if we consider set theory together with large cardinal axioms as formalized in the signature $\sigma_\omega,\sigma_{\omega,\NS_{\omega_1}},\sigma_{\omega_1}$, we obtain (until now unexpected) tameness properties for this first order theory, properties 
which couple perfectly with well 
known (or at least published) generic absoluteness results.
The notion of companionship spectrum gives a model theoretic criterium for selecting these signatures out of the 
continuum many
signatures which produce definable extensions of $\ZFC$. 
Moreover the common practice of set theory (independently of our results)  
motivate the choice of signatures for set theory made in the present paper 
(signatures which belong to the companionship spectrum of set theory), 
and our results validate it.
\end{itemize}

\section{The theory of $H_{\kappa^+}$ is the model companion of set theory}\label{sec:Hkappa+}

In this section we prove Thm. \ref{thm:mainthm0}
The following piece of notation will be used all along this section and supplements Notations \ref{not:keynotation0}, \ref{not:keynotation}:
\begin{notation}\label{not:keynotation1}
\emph{}

\begin{itemize}
\item
$\sigma_\ST$ is the signature containing a 
predicate symbol 
$S_\phi$ of arity $n$ for any $\in$-formula $\phi$ with $n$-many free variables.
\item
$\sigma_\kappa=\sigma_\ST\cup\tau_\ST\cup\bp{\kappa}$ with $\kappa$ a constant symbol.
\end{itemize}

\begin{itemize}
%
\item
$T_\kappa$ is the $\sigma_\ST\cup\bp{\kappa}$-theory given by the axioms
\begin{equation}\label{eqn:keytau*kappa}
\forall x_1\dots x_n\,[S_\psi(x_1,\dots,x_n)\leftrightarrow 
(\bigwedge_{i=1}^n x_i\subseteq \kappa^{<\omega}\wedge \psi^{\pow{\kappa^{<\omega}}}(x_1,\dots,x_n))]
\end{equation}
as $\psi$ ranges over the $\in$-formulae.
\item
$\ZFC^-_\kappa$ is the $\tau_\ST\cup\bp{\kappa}$-theory 
\[
\ZFC^-_\ST\cup\bp{\kappa\text{ is an infinite cardinal}};
\]
\item
$\ZFC^{*-}_\kappa$ is the $\sigma_\kappa$-theory 
\[
\ZFC^-_\kappa\cup T_\kappa;
\]
\item
Accordingly we define 
$\ZFC_\kappa$, 
$\ZFC^*_\kappa$.
\end{itemize}
\end{notation}

\begin{notation}\label{not:modthnot2}
Given a $\in$-structure $(M,E)$ and $\tau$ a signature
extending $\tau_\ST$,
from now we let
$(M,\tau^M)$ be the unique extension of 
$(M,E)$ defined in accordance with 
Notation \ref{not:keynotation} which satisfies $T_\tau$.
In particular $(M,\tau^M)$ is a shorthand for 
$(M,S^M:S\in\tau)$.
If $(N,E)$ is a substructure of $(M,E)$ we also write
$(N,\tau^M)$ as a shorthand for 
$(N,S^M\restriction N:S\in\tau)$.
\end{notation}

\subsection{By-interpretability of the first order theory of $H_{\kappa^+}$
with the first order theory of $\pow{\kappa}$}
\label{subsec:secordequiv}

Let's compare the first order theory of the structure
\[
(\pow{\kappa},S_\phi^V:\phi \text{ an atomic $\tau_\ST$-formula})
\]
with 
that of the $\tau_\ST$-theory of $H_{\kappa^+}$ in models of $\ZFC_{\ST}$. 
We will show that they
are $\ZFC_{\tau_{\ST}}$-provably by-interpretable with a by-interpetation translating $H_{\kappa^+}$ in a $\Pi_1$-definable subset of  $\pow{\kappa^2}$ and
atomic predicates into 
$\Sigma_1$-relations over this set. 
This result is the key to the proof of 
Thm. \ref{thm:mainthm0} and is just outlining the model theoretic consequences of the well-known fact that sets can be coded by well-founded extensional graphs.

\begin{definition}\label{def:codkappa}
Given $a\in H_{\kappa^+}$, $R\in \pow{\kappa^2}$ codes $a$, if
$R$ codes a well-founded extensional relation on 
some $\alpha\leq\kappa$ with top element $0$
so that the transitive collapse mapping of $(\alpha,R)$ maps $0$ to $a$.

\begin{itemize}
\item
$\WFE_\kappa$ is the set of $R\in \pow{\kappa}$ which 
are a well founded extensional relation with domain 
$\alpha\leq\kappa$ and top element $0$.
\item
 $\Cod_\kappa:\WFE_\kappa\to H_{\kappa^+}$ is the map assigning $a$ to $R$ if and only if 
 $R$ codes $a$.
\end{itemize}
\end{definition}

The following theorem shows that the structure $(H_{\kappa^+},\in)$ is interpreted by means of ``imaginaries'' in the structure
$(\pow{\kappa},\tau_{\ST}^V)$ by means of:
\begin{itemize}
    \item a universal $\tau_\ST\cup\bp{\kappa}$-formula (with quantifiers
    ranging over subsets of $\kappa^{<\omega}$)
    defining a set $\WFE_\kappa\subseteq\pow{\kappa^2}$.
    \item an equivalence relation $\cong_\kappa$ on $\WFE_\kappa$
    defined by an existential $\tau_\ST\cup\bp{\kappa}$-formula (with quantifiers
    ranging over subsets of $\kappa^{<\omega}$)
    \item A binary relation $E_\kappa$ on $\WFE_\kappa$
    invariant under $\cong_\kappa$ representing the $\in$-relation as the extension of 
an existential $\tau_\ST\cup\bp{\kappa}$-formula (with quantifiers
    ranging over subsets of $\kappa^{<\omega}$)\footnote{See \cite[Section 25]{JECHST} for proofs of the case $\kappa=\omega$; in particular the statement and proof of Lemma 25.25 and the proof of \cite[Thm. 13.28]{JECHST} contain all ideas on which one can elaborate to draw the conclusions of Thm.~\ref{thm:keypropCod}.}.
\end{itemize}
\begin{theorem}\label{thm:keypropCod}
Assume $\ZFC^{-}_{\kappa}$. The following holds\footnote{Many transitive supersets of $H_{\kappa^+}$ are 
$\tau_\ST\cup\bp{\kappa}$-model of $\ZFC^{-}_{\kappa}$ for $\kappa$ an infinite cardinal (see \cite[Section IV.6]{KUNEN}). 
To simplify notation we assume to have fixed a transitive 
$\tau_\ST\cup\bp{\kappa}$-model $\mathcal{N}$ of $\ZFC^{-}_\kappa$
with domain $N\supseteq H_{\kappa^+}$. The reader can easily realize that all these statements holds for an arbitrary model $\mathcal{N}$ of $\ZFC^-_\kappa$ replacing $H_{\kappa^+}$ with its version according to $\mathcal{N}$.}:
 \begin{enumerate}
\item
The map $\mathrm{Cod}_\kappa$ and $\WFE_\kappa$ are defined by $\ZFC^-_\kappa$-provably 
$\Delta_1$-properties  in parameter $\kappa$. Moreover $\Cod_\kappa:\WFE_\kappa\to H_{\kappa^+}$
is surjective (provably in $\ZFC^{-}_{\kappa}$), and
$\WFE_\kappa$ is defined by a universal 
$\tau_\ST\cup\bp{\kappa}$-formula with quantifiers
ranging over subsets of $\kappa^{<\omega}$.
\item 
There are existential $\tau_\ST\cup\bp{\kappa}$-formulae (with quantifiers
    ranging over subsets of $\kappa^{<\omega}$), $\phi_\in,\phi_=$ such that
for all $R,S\in \WFE_\kappa$, $\phi_=(R,S)$ if and only if $\Cod_\kappa(R)=\Cod_\kappa(S)$ and 
$\phi_\in(R,S)$  if and only if $\Cod_\kappa(R)\in\Cod_\kappa(S)$. In particular letting
\[
E_\kappa=\bp{(R,S)\in \WFE_\kappa: \phi_\in(R,S)},
\]
\[
\cong_\kappa=\bp{(R,S)\in \WFE_\kappa: \phi_=(R,S)},
\]
$\cong_\kappa$ is a $\ZFC^-_\kappa$-provably definable equivalence relation, $E_\kappa$ respects it, and
\[
(\WFE_\kappa/_{\cong_\kappa}, E_\kappa/_{\cong_\kappa})
\]
is 
isomorphic to $(H_{\kappa^+},\in)$ via the map $[R]\mapsto \Cod_\kappa(R)$.
\end{enumerate}
\end{theorem}

\begin{proof}
A detailed proof requires a careful examination of the syntactic properties of $\Delta_0$-formulae, in line with the one carried in Kunen's \cite[Chapter IV]{KUNEN}.
We outline the main ideas, 
following Kunen's book terminology for certain 
set theoretic operations on sets, functions and relations (such as $\dom(f),\ran(f)$, $\text{Ext}(R)$, etc). 
To simplify the notation, we prove the results
for a transitive model $(N,\in)$ which is then extended
to a structure $(N,\tau_\ST^N,\kappa^N)$ which 
models $\ZFC^-_\kappa$, and whose domain contains $H_{\kappa^+}$. 
The reader can verify by itself that the argument 
is modular and works for any other model of 
$\ZFC^-_\kappa$ 
(transitive or ill-founded, containing the ``true'' $H_{\kappa^+}$ or not). 
\begin{enumerate} 
\item This is proved in details in \cite[Chapter IV]{KUNEN}.
To define $\WFE_\kappa$ by a universal $\tau_\ST\cup\bp{\kappa}$-property over subsets of $\kappa$
and $\Cod_\kappa$ by a $\Delta_1$-property for $\tau_\ST\cup\bp{\kappa}$ over $H_{\kappa^+}$, we proceed
as follows:
\begin{itemize}
\item
\emph{$R$ is an extensional relation with domain contained in 
$\kappa$ and top element $0$}
is defined by the $\tau_\ST\cup\bp{\kappa}$-atomic formula 
$\psi_{\mathrm{EXT}}(R)$ $\ZFC^{-}_\kappa$-provably equivalent to the $\Delta_0(\kappa)$-formula:
\begin{align*}
(R\subseteq\kappa^2)\wedge \\
\wedge (\text{Ext}(R)\in\kappa\vee \text{Ext}(R)=\kappa)
\wedge\\
\wedge\forall \alpha,\beta\in\text{Ext}(R)\,
[\forall u\in\text{Ext}(R)\,(u\mathrel{R}\alpha\leftrightarrow u\mathrel{R}\beta)\rightarrow (\alpha=\beta)]\wedge\\
\wedge \forall \alpha\in\text{Ext}(R)\,\neg (0\mathrel{R}\alpha).
\end{align*}
\item 
$\WFE_\kappa$ is defined by the universal $\tau_\ST\cup\bp{\kappa}$-formula 
$\phi_{\WFE_\kappa}(R)$ (quantifying only over subsets of $\kappa^{<\omega}$)
\begin{align*}
\psi_{\mathrm{EXT}}(R)\wedge \\
\wedge [\forall f\subseteq \kappa^2\,(f\text{ is a function }\rightarrow\exists n\in\omega \,
\neg(\ap{f(n+1),f(n)}\in R))].
\end{align*}
Its interpretation is the subset of $\pow{\kappa^{<\omega}}$ of the $\sigma_\kappa$-symbol 
$S_{\phi_{\WFE_\kappa}}$. 

\item To define $\Cod_\kappa$,
consider the $\tau_\ST\cup\bp{\kappa}$-atomic formula 
$\psi_{\Cod}(G,R)$ provably equivalent to the $\tau_\ST\cup\bp{\kappa}$-formula:
\begin{align*}
\psi_{\mathrm{EXT}}(R)\wedge\\
\wedge (G\text{ is a function})\wedge\\
\wedge (\dom(G)=\text{Ext}(R))\wedge (\ran(G)\text{ is transitive})\wedge\\ \wedge\forall\alpha,\beta\in\text{Ext}(R)\,[\alpha\mathrel{R}\beta\leftrightarrow G(\alpha)\in G(\beta)].
\end{align*}

Then $\Cod_\kappa(R)=a$ can be defined either by the existential $\tau_\ST\cup\bp{\kappa}$-formula\footnote{Given an $R$ such that $\psi_{\mathrm{EXT}}(R)$ holds,
\emph{$R$ is a well founded relation} holds in a model of 
$\ZFC^-_\kappa$
if and only if 
$\Cod_\kappa$ is defined on $R$. In the theory $\ZFC^-_\kappa$, $\WFE_\kappa$ can be defined using a universal property 
by a $\tau_\ST\cup\bp{\kappa}$-formula quantifying only over subsets of $\kappa$. On the other hand if we allow arbitrary quantification over elements 
of $H_{\kappa^+}$, we can express the well-foundedness of $R$ also using the existential formula 
$\exists G\,\psi_{\Cod_\kappa}(G,R)$. This is why $\WFE_\kappa$ is defined by a universal 
$\tau_\ST\cup\bp{\kappa}$-property in the structure $(\pow{\kappa},\tau_\ST^V,\kappa)$,
while the graph of $\Cod_\kappa$ can be defined by a $\Delta_1$-property for $\tau_\ST\cup\bp{\kappa}$
in the structure $(H_{\kappa^+},\tau_\ST^V,\kappa^V)$.} 
\[
\exists G\,(\psi_{\Cod}(G,R)\wedge G(0)=a)
\]
or by the universal $\tau_\ST\cup\bp{\kappa}$-formula 
\[
\forall G\,(\psi_{\Cod}(G,R)\rightarrow G(0)=a).
\]
\end{itemize}

\item The equality relation in $H_{\kappa^+}$ is transferred to the isomorphism relation
between elements of $\WFE_\kappa$: if $R,S$ are well-founded extensional on $\kappa$ with a top-element,
the Mostowski collapsing theorem entails that $\Cod_\kappa(R)=\Cod_\kappa(S)$ if and only if 
$(\mathrm{Ext}(R),R)\cong(\mathrm{Ext}(S),S)$. 
Isomorphism of the two structures $(\mathrm{Ext}(R),R)\cong(\mathrm{Ext}(S),S)$ is expressed by the $\Sigma_1$-formula
for $\tau_{\kappa}$:
\[
\phi_=(R,S)\equiv 
\exists f\,(f \text{ is a bijection of $\kappa$ onto $\kappa$ and $\alpha R\beta$ if and only if 
$f(\alpha) S f(\beta)$}).
\]
In particular we get that $S_{\phi_=}(R,S)$ holds in $H_{\kappa^+}$ for $R,S\in \WFE_\kappa$
if and only if $\Cod_\kappa(R)=\Cod_\kappa(S)$.

Similarly one can express $\Cod_\kappa(R)\in\Cod_\kappa(S)$ by the $\Sigma_1$-property $\phi_\in$
in $\tau_{\kappa}$
stating that
$(\mathrm{Ext}(R),R)$ is isomorphic to $(\mathrm{pred}_S(\alpha),S)$ for some $\alpha\in\kappa$ with $\alpha \mathrel{S}0$, 
where $\mathrm{pred}_S(\alpha)$ is given by the elements
of $\mathrm{Ext}(S)$ which are connected by a finite path to $\alpha$. 

Moreover letting $\cong_\kappa\subseteq \WFE_\kappa^2$ denote the isomorphism relation between elements of $\WFE_\kappa$
and $E_\kappa\subseteq \WFE_\kappa^2$ denote the relation which translates into the $\in$-relation via $\Cod_\kappa$, 
it is clear that $\cong_\kappa$ is a congruence relation
over $E_\kappa$, i.e.: if $R_0 \cong_\kappa R_1$ and $S_0\cong_\kappa S_1$,
$R_0 \mathrel{E_\kappa} S_0$ if and only if $R_1 \mathrel{E_\kappa}  S_1$.

This gives that the structure $(\WFE_\kappa/_{\cong_\kappa}, E_\kappa/_{\cong_\kappa})$ is 
isomorphic to $(H_{\kappa^+},\in)$ via the map $[R]\mapsto \Cod_\kappa(R)$ 
(where $\WFE_\kappa/_{\cong_\kappa}$ is the set
of equivalence classes of $\cong_\kappa$ and the quotient relation $[R] \mathrel{E_\kappa/_{\cong_\kappa}} [S]$ holds 
if and only if $R \mathrel{E_\kappa}  S$).

This isomorphism is defined via the map $\Cod_\kappa$, which is by itself defined by 
a $\ZFC^-_\kappa$-provably $\Delta_1$-property for $\tau_\ST\cup\bp{\kappa}$.

The very definition of $\WFE_\kappa,\cong_\kappa,E_\kappa$ show that
\[
\WFE_\kappa=S_{\phi_{\WFE_\kappa}}^{N},
\]
\[
\cong_\kappa=S_{\phi_{\WFE_\kappa}(x)\wedge \phi_{\WFE_\kappa}(y)\wedge \phi_{=}(x,y)}^{N},
\]
\[
E_\kappa=S_{\phi_{\WFE_\kappa}(x)\wedge \phi_{\WFE_\kappa}(y)\wedge \phi_{\in}(x,y)}^{N}.
\]
\end{enumerate}
\end{proof}

\subsection{Model completeness for the theory of $H_{\kappa^+}$}

\begin{theorem}\label{thm:modcompHkappa+}
Any $\sigma_\kappa$-theory $T$ extending 
\[
\ZFC^{*-}_\kappa\cup\bp{\text{all sets have size $\kappa$}}
\]
is model complete.
\end{theorem}

\begin{proof}
To simplify notation, 
we conform to the assumption of the previous theorem, 
i.e. we assume that the model $(N,\in)$ which is uniquely extended to a model of 
$\ZFC^{*-}_\kappa+$\emph{ every set has size $\kappa$}
on which we work is a transitive superstructure of 
$H_{\kappa^+}$.

The statement \emph{every set has size $\kappa$} is satisified by a
$\ZFC^-_\kappa$-model $(N,\tau_\ST^V,\kappa)$ with 
$N\supseteq H_\kappa^+$ if and only if $N=H_{\kappa^+}$.
From now on we proceed assuming this equality.

By Robinson's test \ref{lem:robtest} it suffices to show that
for all $\in$-formulae $\phi(\vec{x})$ 
\[
\ZFC^-_\kappa+\text{ every set has size $\kappa$}\vdash
\forall\vec{x} \,(\phi(\vec{x})\leftrightarrow\psi_\phi(\vec{x})),
\]
for some universal $\sigma_\kappa$-formula $\psi_\phi$.

We will first define a recursive map 
$\phi\to\theta_\phi$ which maps $\Sigma_n$-formulae $\phi$ for $\bp{\in,\kappa}$ 
quantifying over all elements of
$H_{\kappa^+}$ to $\Sigma_{n+1}$-formulae $\theta_\phi$ for $\tau_\ST\cup\bp{\kappa}$
whose quantifier range just over subsets of $\kappa^{<\omega}$.

The proof of the previous theorem gave $\tau_\ST\cup\bp{\kappa}$-formulae
$\theta_{x=y}$, $\theta_{x\in y}$ such that
\[
S_{\theta_{x=y}}^{H_{\kappa^+}}=\cong_\kappa=\bp{(R,S)\in (\WFE_\kappa)^2:\, \Cod_\kappa(R)=\Cod_\kappa(S)},
\]
\[
S_{\theta_{x\in y}}^{H_{\kappa^+}}=E_\kappa=\bp{(R,S)\in (\WFE_\kappa)^2:\, \Cod_\kappa(R)\in\Cod_\kappa(S)}.
\]
Specifically (following the notation of that proof)
\[
\theta_{x=y}=\phi_{\WFE_\kappa}(x)\wedge \phi_{\WFE_\kappa}(y)\wedge \phi_{=}(x,y),
\]
\[
\theta_{x\in y}=\phi_{\WFE_\kappa}(x)\wedge \phi_{\WFE_\kappa}(y)\wedge \phi_{\in}(x,y).
\]

Now for any $\bp{\in,\kappa}$-formula $\psi(\vec{x})$, 
we proceed to define the $\tau_\ST\cup\bp{\kappa}$-formula $\theta_\psi(\vec{x})$ letting:
\begin{itemize}
\item $\theta_{\psi\wedge\psi}(\vec{x})$ be $\theta_{\psi}(\vec{x})\wedge\theta_{\psi}(\vec{x})$,
\item $\theta_{\neg\psi}(\vec{x})$ be $\neg\theta_{\psi}(\vec{x})$,
\item $\theta_{\exists y\psi(y,\vec{x})}(\vec{x})$ be $\exists y\theta_{\psi}(y,\vec{x})\wedge \phi_{\WFE_\kappa}(y)$.
\end{itemize}
An easy induction on the complexity of the $\tau_\ST\cup\bp{\kappa}$-formulae $\theta_\phi(\vec{x})$
gives that for any $\bp{\in,\kappa}$-definable subset
$A$ of $(H_{\kappa^+})^n$ which is the extension of some $\bp{\in,\kappa}$-formula $\phi(x_1,\dots,x_n)$ 
\[
\bp{(R_1,\dots,R_n)\in (\WFE_\kappa)^n:\, (\Cod_\kappa(R_1),\dots,\Cod_\kappa(R_n))\in A}=
S_{\theta_\phi}^{H_{\kappa^+}},
\] 
with the further property that $S_{\theta_\phi}^{H_{\kappa^+}}\subseteq (\WFE_{\kappa})^n$ respects the
$\cong_\kappa$-relation\footnote{It is also clear from our argument that the map $\phi\mapsto\theta_\phi$ is recursive (and a careful inspection
reveals that it maps a $\Sigma_n$-formula
to a $\Sigma_{n+1}$-formula).}.

Now every $\sigma_\kappa$-formula is $\ZFC^{*-}_\kappa$-equivalent to a 
$\bp{\in,\kappa}$-formula\footnote{The map assigning to any $\sigma_\kappa$-formula a 
$\ZFC^{*-}_\kappa$-equivalent
$\bp{\in,\kappa}$-formula can also be chosen to be recursive.}.

Therefore we can extend $\phi\mapsto\theta_\phi$ 
assigning to any $\sigma_\kappa$-formula $\phi(\vec{x})$ the formula $\theta_\psi(\vec{x})$ for some 
$\bp{\in,\kappa}$-formula $\psi(\vec{x})$ which is $\ZFC^{*-}_\kappa$-equivalent to $\phi(\vec{x})$.

Then for any $\bp{\in,\kappa}$-formula $\phi(x_1,\dots,x_n)$
$H_{\kappa^+}\models \phi(a_1,\dots,a_n)$ if and only if 
\[
(\WFE_\kappa/_{\cong_\kappa}, E_\kappa/_{\cong_\kappa})\models \phi([R_1],\dots,[R_n])
\]
with $\Cod_\kappa(R_i)=a_i$ for $i=1,\dots,n$
if and only if 
\[
H_{\kappa^+}\models \forall R_1,\dots,R_n\, [(\bigwedge_{i=1}^n \Cod_\kappa(R_i)=a_i)\rightarrow\theta_\phi(R_1,\dots,R_n)]
\]
if and only if 
\[
H_{\kappa^+}\models \forall R_1,\dots,R_n\, [(\bigwedge_{i=1}^n \Cod_\kappa(R_i)=a_i)\rightarrow S_{\theta_\phi}(R_1,\dots,R_n)].
\]

Since this argument can be repeated verbatim for any model of 
$\ZFC^{*-}_\kappa$+\emph{every set has size $\kappa$}, and any $\sigma_\kappa$-formula
 is $\ZFC^{*-}_\kappa$-equivalent to a $\bp{\in,\kappa}$-formula,
we have proved the following:
\begin{claim}
For any $\sigma_\kappa$-formula $\phi(x_1,\dots,x_n)$,
$\ZFC^{*-}_\kappa+$\emph{every set has size $\kappa$} proves that
\[
\forall x_1,\dots,x_n\,[\phi(x_1,\dots,x_n)\leftrightarrow \forall y_1,\dots,y_n\,[(\bigwedge_{i=1}^n \Cod_\kappa(y_i)=x_i)\rightarrow S_{\theta_\phi}(y_1,\dots,y_n)]].
\]
\end{claim}
But $\Cod_\kappa(y)=x$ is expressible by an existential
$\tau_\ST\cup\bp{\kappa}$-formula 
provably in $\ZFC^-_\kappa\subseteq \ZFC^{*-}_\kappa$, therefore
\[
\forall y_1,\dots,y_n\,[(\bigwedge_{i=1}^n \Cod_\kappa(y_i)=x_i)\rightarrow S_{\theta_\phi}(y_1,\dots,y_n)]
\]
is a universal $\sigma_\kappa$-formula, and we are done. 
\end{proof}

\subsection{Proof of Thm.~\ref{thm:mainthm0}}

Conforming to the notation of Thm.~\ref{thm:mainthm0},
it is clear that $\sigma_\kappa$ is a signature of the form
$\bp{\in}_{\bar{A}_\kappa}$ whenever $\kappa$ is a $T$-definable cardinal for some $T$ extending $\ZFC$.
Therefore the following result completes the proof of Thm.~\ref{thm:mainthm0}.

\begin{theorem} \label{Thm:mainthm-1}
Assume $T\supseteq \ZFC^*_\kappa$
is a $\sigma_\kappa$-theory.
Then $T$ has a model companion $T^*$.
Moreover for any $\Pi_2$-sentence $\psi$ for $\sigma_\kappa$, TFAE:
\begin{enumerate}
\item $\psi\in T^*$;
\item
$T\vdash\psi^{H_{\kappa^+}}$;
\item
For all universal $\sigma_\kappa$-sentences $\theta$,
$T+\theta$ is consistent if and only if so is $T_\forall+\theta+\psi$.
\end{enumerate}
\end{theorem}

\begin{proof}
By Thm. \ref{thm:modcompHkappa+}, any $\sigma_\kappa$-theory extending
\[
\ZFC^{*-}_\kappa+\emph{every set has size $\kappa$}
\] 
is model complete.
Therefore so is 
\[
T^*=\bp{\phi: H_{\kappa^+}^\mathcal{M}\models \phi,\,\mathcal{M}\models T},
\] 
since $H_{\kappa^+}^\mathcal{M}$ models $\ZFC^{*-}_\kappa$+\emph{every set has size $\kappa$}
for any  $\mathcal{M}$ which models $T$.

We must now show that $T^*_\forall=T_\forall$.
Assume $T^*\models\theta$ for some universal sentece $\theta$.
Then $H_{\kappa^+}^\mathcal{M}\models \theta$ for any model $\mathcal{M}$ of $T$.
Since $H_{\kappa^+}^\mathcal{M}\prec_1 \mathcal{M}$ for any such $\mathcal{M}$,
we get that any such $\mathcal{M}$ models $\theta$ as well. Therefore
$T^*_\forall\subseteq T_\forall$. Appealing again to Levy absoluteness, by a similar argument,
we get that $T_\forall\subseteq T^*_\forall$.

We now show that $T^*$ is the set of $\Pi_2$-sentences $\phi$ such that:
\begin{quote}
For all $\Pi_1$-sentences $\phi$ for $\tau$,
$T+\theta$ is consistent if and only if so is 
$T_\forall+\phi+\theta$.
 \end{quote}

We prove it establishing that $T$ and $T^*$ satisfy the assumption of 
Lemma \ref{fac:proofthm1-2} i.e. for any $\Pi_1$-sentence $\theta$ for $\sigma_\kappa$ $T+\theta$ 
is consistent if and only if so is
$T^*+\theta$.

So assume $T+\theta$ is consistent for some $\Pi_1$-sentence $\theta$, we must show that 
$T^*+\theta$ is also consistent, but this is immediate: by Levy absoluteness if
$\mathcal{M}$ models $\theta$, so does $H_{\kappa^+}^{\mathcal{M}}$.

Conversely assume $T+\theta$ is inconsistent for some $\Pi_1$-sentence $\theta$.
Then $T\models\neg\theta$. Again by Levy absoluteness if
$\mathcal{M}$ models $T$, $H_{\kappa^+}^{\mathcal{M}}\models\neg\theta$.
Hence $\neg\theta\in T^*$ by definition, and $\theta$ is inconsistent with $T^*$. 
\end{proof}

\begin{remark}
Note that the family of models 
$\bp{H_{\kappa^+}^\mathcal{M}:\,\mathcal{M}\models T}$ we used to define 
$T^*$ may not be an elementary class for $\sigma_\kappa$.

Thm. \ref{Thm:mainthm-1} can be proved for 
many other signatures other than $\sigma_\kappa$.
It suffices that the signature in question adds new predicates just for definable subsets of $\pow{\kappa}^n$,
and also that it adds family of predicates which are closed under definability (i.e. projections, complementation, finite unions, permutations) and under the map $\Cod_\kappa$.
Under these assumptions we can still use 
Lemma \ref{lem:levabsgen} and 
Fact \ref{fac:charkaihullnonpi1comp}  to argue for the evident variations of the proof of 
Thm. \ref{Thm:mainthm-1} to this set up.
However linking  these model companionship results to generic absoluteness as we do in Theorem \ref{thm:mainthm1}
requires much more care in the definition of the signature.
We will pursue this matter in more details in the next sections.  
\end{remark}

\subsection{A weak version of Theorem \ref{thm:mainthm1} for third order arithmetic}

We can prove a weak version of Thm. \ref{thm:mainthm1} for the theory of $H_{\aleph_2}$ 
appealing to the generic absoluteness results of \cite{VIAMM+++,VIAASP,VIAAUD14} 
which establish the invariance of the theory of $H_{\aleph_2}$ in models of strong forcing 
axioms with respect to stationary set preserving
forcings preserving these axioms.

Let $\ZFC^*_{\omega_1}\supseteq\ZFC_\ST$ be the 
$\sigma_{\omega_1}=\sigma_\omega\cup\bp{\kappa}$-theory
obtained adding axioms which force in each of its 
$\sigma_{\omega_1}$-models
$\kappa$ to be interpreted by the first uncountable cardinal, and
each predicate symbol $S_\phi$ to be interpreted as the subset of 
$\pow{\omega_1^{<\omega}}^n$ defined by $\phi^{\pow{\omega_1^{<\omega}}}(x_1,\dots,x_n)$.


\begin{theorem}\label{Thm:mainthm-5}
Let
$T$ be a $\sigma_{\omega_1}$-theory extending   
\[
\ZFC^*_{\omega_1}+\MM^{+++}+\emph{there are class many superhuge cardinals}.
\]

TFAE for any
$\Pi_2$-sentence $\psi$ for $\sigma_{\omega_1}$:
\begin{enumerate}
\item \label{Thm:mainthm-5-2}
$S_\forall+\psi$ is consistent for all complete $S$ extending $T$;
\item \label{Thm:mainthm-5-3}
$T$ proves that some stationary set preserving forcing notion $P$ forces 
$\psi^{\dot{H}_{\omega_2}}+\MM^{+++}$;
\item \label{Thm:mainthm-5-1}
$T\vdash \psi^{H_{\omega_2}}$.
\end{enumerate}
\end{theorem}


See Remarks \ref{Rem:keyrem1}(\ref{Rem:keyrem1-4}) for some information on $\MM^{+++}$,
and \ref{Rem:keyrem1}(\ref{Rem:keyrem1-3}) for informations on superhugeness.

The proof of Theorem \ref{Thm:mainthm-5} is a trivial variation of the proof of Theorem \ref{Thm:mainthm-1}:
\begin{proof}
\cite[Thm. 5.18]{VIAMM+++} gives that
\ref{Thm:mainthm-5}(\ref{Thm:mainthm-5-1}) and \ref{Thm:mainthm-5}(\ref{Thm:mainthm-5-3}) 
are equivalent.
Theorem \ref{Thm:mainthm-1}  establishes the equivalence of \ref{Thm:mainthm-5}(\ref{Thm:mainthm-5-1}) and
\ref{Thm:mainthm-5}(\ref{Thm:mainthm-5-2}). 
\end{proof}

\begin{remark}\label{Rem:keyrem1}
\emph{}

\begin{enumerate}
\item\label{Rem:keyrem1-3}
$\delta$ is superhuge if it supercompact and this can be witnessed by huge embeddings.
A superhuge cardinal is consistent
relative to the existence of a $2$-huge cardinal.

\item\label{Rem:keyrem1-4}
For a definition of $\MM^{+++}$ see \cite[Def. 5.19]{VIAMM+++}. We just note that
$\MM^{+++}$ is a natural strengthening of Woodin's axiom 
$(*)$ (by the recent breakthrough of 
Asper\`o and Schindler \cite{ASPSCH(*)}) and of Martin's maximum 
(for example any
of the standard iterations to produce a model of Martin's maximum produce a model of 
$\MM^{+++}$
if the iteration has length a superhuge cardinal \cite[Thm 5.29]{VIAMM+++}).

\item\label{Rem:keyrem1-5}
We can prove exactly the same results of Thm. \ref{Thm:mainthm-5} replacing 
(verbatim in its statement)
$\MM^{+++}$ by any of the axioms $\mathsf{RA}_\omega(\Gamma)$ introduced in \cite{VIAAUD14} or the axioms 
$\mathsf{CFA}(\Gamma)$ and $\mathsf{BCFA}(\Gamma)$ introduced in \cite{VIAASP}, provided
in item \ref{Thm:mainthm-5}(\ref{Thm:mainthm-5-3})
\emph{stationary set preserving forcing notion $P$}
is replaced by $P\in\Gamma$.
\item We consider Thm. \ref{Thm:mainthm-5} weaker than Thm. \ref{thm:mainthm1} or Corollary
\ref{cor:maincor}, because
in Corollary \ref{cor:maincor} one can choose the theory $T$ to be inconsistent with 
$\MM^{++}$ without hampering its conclusion (for example $T$ could satisfy $\mathsf{CH}$,
a statement denied by $\MM^{++}$),
and because Corollary \ref{cor:maincor}\ref{cor:maincorC} holds for all forcing notions $P$
unlike Thm. \ref{Thm:mainthm-5}(\ref{Thm:mainthm-5-3}).
The key point separating these two results is that the signature $\sigma_{\omega_1}$ is too 
expressive and renders many statements incompatible with forcing axioms formalizable by existential 
(or even atomic)
$\sigma_{\omega_1}$-sentences
(for example such is the case for $\mathsf{CH}$).

%
%
\item\label{Rem:keyrem1-2}
A key distinction between the signature
$\sigma_{\omega_1}$ and the signature $\bp{\in}_{\bar{A}_2}$ considered in Thm. \ref{thm:mainthm1}
is that for any $T\supseteq\ZFC+$\emph{appropriate large cardinals} 
$\mathsf{CH}$ cannot be 
$T$-equivalent to a
$\Sigma_1$-sentence for $\bp{\in}_{\bar{A}_2}$ because $\CH$ is a statement 
which can change its truth value across
forcing extensions, while the universal $\bp{\in}_{\bar{A}_2}$-sentences 
maintain the same truth value across all forcing extensions of a model of $T$, by 
Thm. \ref{thm:mainthm1}(\ref{thm:mainthm1.4}).
On the other hand $\CH$ is $\ZFC_{\omega_1}$-equivalent to
an atomic $\sigma_{\omega_1}$-sentence.
$\neg\mathsf{CH}$  is the simplest example of the 
type of $\Pi_2$-sentences which exemplifies why
 Thm. \ref{Thm:mainthm-5}(\ref{Thm:mainthm-5-3}) is much weaker than
 Thm. \ref{thm:mainthm1}, and why Thm. \ref{thm:mainthm1} for the signature $\bp{\in}_{\bar{A}_2}$ needs a different 
 (and as we will see much more sophisticated) proof strategy than the one we use here to 
establish Theorems \ref{Thm:mainthm-1} and \ref{Thm:mainthm-5}.  
\end{enumerate}
\end{remark}

\section{Generic invariance results for signatures of second and third order arithmetic}\label{sec:geninv}

We collect here generic absoluteness results results needed to prove
Thm. \ref{thm:mainthm1}.
We prove all these results working in ``standard''
models of $\ZFC$, i.e. we assume the models are well-founded. This is a practice we already adopted
in Section \ref{sec:Hkappa+}. We leave to the reader to remove this unnecessary assumption.

\subsection{Universally Baire sets and generic absoluteness for second order number theory}\label{sec:genabssecordnumth}

We recall here the properties of universally Baire sets and the generic absoluteness results for second order 
number theory we need to prove Thm. \ref{thm:mainthm1}.

\begin{notation}
$\mathcal{A}\subseteq \bigcup_{n\in\omega}\pow{\kappa}^n$ is projectively closed
if it is closed under projections, finite unions, complementation, and permutations
(if $\sigma:n\to n$ is a permutation and $A\subseteq\pow{\kappa}^n$, 
$\hat{\sigma}[A]=\bp{(a_{\sigma(0)},\dots,a_{\sigma(n-1}):\, (a_0,\dots,a_{n-1})\in A}$).

Otherwise said, $\mathcal{A}$ 
is the class of lightface definable subsets of some signature on 
$\pow{\kappa}$.
\end{notation}

\subsection{Universally Baire sets}\label{subsec:univbaire}
Assuming large cardinals 
there is a very large sample of projectively closed families of subsets of $\pow{\omega}$ which are are ``simple'', 
hence it is natural to consider elements of these families
as atomic predicates. 

The exact definition of what is meant by a ``simple'' subset of $2^\omega$ 
is captured by the notion of universally Baire set.

Given a topological space $(X,\tau)$, $A\subseteq X$ is nowhere dense if its closure 
has a dense complement,
meager if it is the countable union of nowhere dense sets, with the Baire property if it 
has meager symmetric difference with
an open set.
Recall that $(X,\tau)$ is Polish if $\tau$ is a completely metrizable, separable topology on $X$.

\begin{definition}
(Feng, Magidor, Woodin) 
Given a Polish space $(X,\tau)$, $A\subseteq X$ is \emph{universally Baire} 
if for every compact Hausdorff space $(Y,\sigma)$ and
every continuous $f:Y\to X$ we have that $f^{-1}[A]$ has the Baire property in $Y$.

$\bool{UB}$ denotes the family of universally Baire subsets of $X$ for some Polish space $X$.
\end{definition}

We adopt the convention that $\mathsf{UB}$ denotes the class of universally Baire sets and of all elements of 
$\bigcup_{n\in\omega+1}(2^{\omega})^n$ (since the singleton of such elements are universally Baire sets).


The theorem below outlines three simple examples of projectively closed families of universally Baire sets
containing $2^\omega$.
\begin{theorem}\label{thm:UBsetsgenabs}
Let $T_0$ be the $\tau_\ST$-theory $\mathsf{ZFC_\ST}+$\emph{there are infinitely many 
Woodin cardinals and a measurable above}
and $T_1$ be the $\tau_\ST$-theory $\mathsf{ZFC_\ST}+$\emph{there are class many Woodin cardinals}.
\begin{enumerate}
\item \cite[Thm. 3.1.12, Thm. 3.1.19]{STATLARSON}
Assume $V$ models $T_0$. Then every projective subset of $2^\omega$ is universally Baire.
\item \cite[Thm. 3.3.3, Thm. 3.3.5, Thm. 3.3.6, Thm. 3.3.8, Thm. 3.3.13, Thm. 3.3.14]{STATLARSON}
Assume $V\models T_1$.
Then $\mathsf{UB}$ is projectively closed.
\end{enumerate}
\end{theorem}
%

To proceed further we now list the standard facts about universally Baire sets we will need:
 
 \begin{enumerate}
 \item\label{itm1:charUBsets} \cite[Thm. 32.22]{JECHST}
 $A\subseteq 2^{\omega}$ is universally Baire if and only if for each forcing notion $P$ there are 
 trees $T_A,S_A$ on $\omega\times\delta$ for some $\delta> |P|$
 such that $A=p[[T_A]]$ (where $p:(2\times\kappa)^\omega\to 2^{\omega}$ denotes the projection on the first component and $[T]$ denotes the body of the tree $T$), and
 \[
P\Vdash T_A\text{ and }S_A\text{ project to complements},
\]
by this meaning that for all $G$ $V$-generic for $P$
\[
V[G]\models (p[[T_A]]\cap p[[S_A]]=\emptyset)\wedge (p[[T_A]]\cup p[[S_A]]=(2^\omega)^{V[G]})
\]
\item 
Any two Polish spaces $X,Y$ of the same cardinality are Borel isomorphic \cite[Thm. 15.6]{kechris:descriptive}. 
\item 
Any Polish space is Borel isomorphic to a Borel subset of $[0;1]^\omega$ \cite[Thm. 4.14]{kechris:descriptive}, hence also to a Borel subset of $2^\omega$ (by the previous item).
 \item
Given $\phi:\mathbb{N}\to\mathbb{N}$, $\prod_{n\in\omega}2^{\phi(n)}$ is Polish
(it is actually homemomorphic to the union of $2^\omega$ with a countable Hausdorff space) 
\cite[Thm. 6.4, Thm. 7.4]{kechris:descriptive}.
\end{enumerate}

Hence it is not restrictive to focus just on universally Baire subsets of $2^\omega$ and of its
countable products, which is what we will do in the sequel.

\begin{notation}\label{not:notUBsetsVG}
Given $G$ a $V$-generic filter for some forcing $P\in V$, $A\in \UB^{V[G]}$ and
$H$ $V[G]$-generic filter for some forcing $Q\in V[G]$, 
\[
A^{V[G][H]}=\bp{r\in (2^\omega)^{V[G][H]}: V[G][H]\models r\in p[[T_A]]},
\]
where $(T_A,S_A)\in V[G]$ is any pair of trees as given in item \ref{itm1:charUBsets} above
such that $p[[T_A]]=A$ holds in $V[G]$,
and $(T_A,S_A)$ project to complements in $V[G][H]$.
\end{notation}

\subsection{Generic absoluteness for second order number theory}\label{subsec:genabssecnumth}

The following generic absoluteness result is the key to establish
Thm. \ref{thm:mainthm1}(\ref{thm:mainthm1.4}) for the signature $A_1$.



We decide to include a full proof of Woodin's generic absoluteness results
for second order number theory we use in this paper. The version we need
follows readily from \cite[Thm. 3.1.2]{STATLARSON} and the assumptions that there exists
class many Woodin limits of Woodin; here we reduce these large cardinal assumptions to the existence of class many Woodin cardinals, while providing an alternative approach
to the proof of some of these result.
The theorem below is an improvement of  \cite[Thm. 3.1]{VIAMMREV}.

\begin{theorem}\label{thm:genabshomega1}
Assume in $V$ there are class many Woodin cardinals. Let $\mathcal{A}\in V$ be a 
family of universally Baire sets of $V$ and $\tau_{\mathcal{A}}=\tau_{\ST}\cup\mathcal{A}$. Let $G$ be $V$-generic for some forcing notion $P\in V$.

Then 
\[
(H_{\omega_1},\tau_{\mathcal{A}}^V)\prec(H_{\omega_1}^{V[G]},\tau_{\ST}^{V[G]},A^{V[G]}:A\in\mathcal{A}).
\]
\end{theorem}
\begin{proof}
%
%
We proceed by induction on $n$ to prove the following stronger assertion:

\begin{claim}
Whenever $G$ is $V$-generic for some forcing notion $P$ in $V$ 
and $H$ is $V[G]$-generic for some forcing notion $Q$ in $V[G]$
\[
(H_{\omega_1}^{V[G]},\tau_{\ST}^{V[G]},A^{V[G]}: A\in\mathcal{A})\prec_n 
(H_{\omega_1}^{V[G][H]},\tau_{\ST}^{V[G][H]},A^{V[G][H]}: A\in\mathcal{A}).
\]
\end{claim}
\begin{proof}
It is not hard to 
check that for all  $A\in \mathcal{A}$, $A^{V[G]}=A^{V[G][H]}\cap V[G]$ 
(choose in $V$ a pair of trees $(T,S)$ such that $A=p[[T]]$ and the pair $(T,S)$ projects to complements in 
$V[G][H]$, and therefore also in $V[G]$).
Therefore  
$(H_{\omega_1}^{V[G]},\tau_{\ST}^{V[G]},A^{V[G]}: A\in\mathcal{A})$ is a
$\tau_{\mathcal{A}}$-substructure of  $(H_{\omega_1}^{V[G][H]},\tau_{\ST}^{V[G][H]},A^{V[G][H]}: A\in\mathcal{A})$.

This proves the base case of the induction.

We prove the successor step.

Assume that for any $G$ $V$-generic for some forcing $P\in V$ and $H$ $V[G]$-generic for some forcing $Q\in V[G]$
\[
(H_{\omega_1}^{V[G]},\tau_{\ST}^{V[G]},A^{V[G]}: A\in\mathcal{A})\prec_n 
(H_{\omega_1}^{V[G][H]},\tau_{\ST}^{V[G][H]},A^{V[G][H]}: A\in\mathcal{A}).
\]	
Fix $\bar{G}$ and $\bar{H}$ as in the assumptions of the Claim as witnessed by forcings $\bar{P}\in V$ and 
$\bar{Q}\in V[\bar{G}]$. 

We want to show that 
\[
		 (H_{\omega_1}^{V[\bar{G}]},\tau_{\ST}^{V[\bar{G}]},A^{V[\bar{G}]}: A\in\mathcal{A})\prec_{n+1} 
		 (H_{\omega_1}^{V[\bar{G}][\bar{H}]},\tau_{\ST}^{V[\bar{G}][\bar{H}]},A^{V[\bar{G}][\bar{H}]}: A\in\mathcal{A}).
\]
	
Let $\gamma$ be a Woodin cardinal of $V$ such that $\bar{P}\ast\dot{\bar{Q}}\in V_\gamma$ 
(where $\dot{\bar{Q}}\in V^P$ is chosen so that $\dot{\bar{Q}}_G=\bar{Q}$).

Then $\gamma$ is Woodin also in $V[\bar{G}]$. Let $K$ be $V[\bar{G}]$-generic for\footnote{$\tow{T}^{\omega_1}_\gamma$ denotes here the countable tower of height $\gamma$ denoted as $\mathbb{Q}_{<\gamma}$ in \cite[Section 2.7]{STATLARSON}.}  
$(\tow{T}^{\omega_1}_\gamma)^{V[\bar{G}]}$
with $\bar{H}\in V[K]$, so that $V[\bar{G}][K]=V[\bar{G}][\bar{H}][\bar{K}]$ for some 
$\bar{K}\in V[\bar{G}][K]$.	
	
		 Hence we have the following diagram:
		\[
			\begin{tikzpicture}[xscale=1.3,yscale=-0.6]
				\node (A0_0) at (0, 0) {$(H_{\omega_1}^{V[\bar{G}]},\tau_{\ST}^{V[\bar{G}]},A^{V[\bar{G}]}: A\in\mathcal{A})$};
				\node (A0_2) at (6, 0) {$(H_{\omega_1}^{V[\bar{G}][K]},\tau_{\ST}^{V[\bar{G}][K]},
				A^{V[\bar{G}][K]}: A\in\mathcal{A})$};
				\node (A1_1) at (3, 3) {$(H_{\omega_1}^{V[\bar{G}][\bar{H}]},
				\tau_{\ST}^{V[\bar{G}][\bar{H}]},A^{V[\bar{G}][\bar{H}]}: A\in\mathcal{A})$};
				\path (A0_0) edge [->]node [auto] {$\scriptstyle{\Sigma_\omega}$} (A0_2);
				\path (A1_1) edge [->]node [auto,swap] {$\scriptstyle{\Sigma_{n}}$} (A0_2);
				\path (A0_0) edge [->]node [auto,swap] {$\scriptstyle{\Sigma_{n}}$} (A1_1);
			\end{tikzpicture}
		\]
		obtained by inductive hypothesis applied both on $V[\bar{G}]$, $V[\bar{G}][\bar{H}]$ and on 
		$V[\bar{G}][\bar{H}]$, $V[\bar{G}][\bar{H}][\bar{K}]$, and using the fact that 
		$(H_{\omega_1}^{V[\bar{G}][K]},\tau_{\UB^{V[\bar{G}]}}^{V[\bar{G}][K]})$ 
		is a fully elementary 
		superstructure of $(H_{\omega_1}^{V[\bar{G}]},\tau_{\UB^{V[\bar{G}]}}^{V[\bar{G}]})$ \cite[Thm. 2.7.7, Thm. 2.7.8]{STATLARSON}.

		Let $\phi \equiv \exists x \psi(x)$ be any $\Sigma_{n+1}$ formula for $\tau_{\mathcal{A}}$
		with parameters in $H_{\omega_1}^{V[\bar{G}]}$.
		First suppose that $\phi$ holds in $(H_{\omega_1}^{V[\bar{G}]},\tau_{\ST}^{V[\bar{G}]},A^{V[\bar{G}]}: A\in\mathcal{A})$, 
		and fix $\bar{a} \in V[\bar{G}]$ such that $\psi(\bar{a})$ holds
		in $(H_{\omega_1}^{V[\bar{G}]},\tau_{\ST}^{V[\bar{G}]},A^{V[\bar{G}]}: A\in\mathcal{A})$. 
		Since 
		\[
		(H_{\omega_1}^{V[\bar{G}]},\tau_{\ST}^{V[\bar{G}]},A^{V[\bar{G}]}: A\in\mathcal{A})\prec_n
		(H_{\omega_1}^{V[\bar{G}][\bar{H}]},\tau_{\ST}^{V[\bar{G}][\bar{H}]},A^{V[\bar{G}][\bar{H}]}: A\in\mathcal{A}),
		\]
		we conclude that $\psi(\bar{a})$ holds
		in $(H_{\omega_1}^{V[\bar{G}][\bar{H}]},\tau_{\ST}^{V[\bar{G}][\bar{H}]},A^{V[\bar{G}][\bar{H}]}: A\in\mathcal{A})$, 
		hence so does
		$\phi$.
		
		Now suppose that $\phi$ holds in 
		$(H_{\omega_1}^{V[\bar{G}][\bar{H}]},\tau_{\ST}^{V[\bar{G}][\bar{H}]},A^{V[\bar{G}][\bar{H}]}: A\in\mathcal{A})$
		 as witnessed by $\bar{a} \in H_{\omega_1}^{V[\bar{G}][\bar{H}]}$. 
		 
		 Since 
		 \[
		 (H_{\omega_1}^{V[\bar{G}][\bar{H}]},\tau_{\ST}^{V[\bar{G}][\bar{H}]},A^{V[\bar{G}][\bar{H}]}: A\in\mathcal{A})
		 \prec_n
		 (H_{\omega_1}^{V[\bar{G}][K]},\tau_{\ST}^{V[\bar{G}][K]},A^{V[\bar{G}][K]}: A\in\mathcal{A}),
		 \]
		 it follows that $\psi(\bar{a})$ holds in 
		 $(H_{\omega_1}^{V[\bar{G}][K]},\tau_{\ST}^{V[\bar{G}][K]},A^{V[\bar{G}][K]}: A\in\mathcal{A})$, 
		 hence so does $\phi$. 
		 Since 
		 \[
		 (H_{\omega_1}^{V[\bar{G}]},\tau_{\ST}^{V[\bar{G}]},A^{V[\bar{G}]}: A\in\mathcal{A})\prec
		 (H_{\omega_1}^{V[\bar{G}][K]},\tau_{\ST}^{V[\bar{G}][K]},A^{V[\bar{G}][K]}: A\in\mathcal{A}),
		 \]
		the formula $\phi$ holds also in $(H_{\omega_1}^{V[\bar{G}]},\tau_{\ST}^{V[\bar{G}]},A^{V[\bar{G}]}: A\in\mathcal{A})$.
		
		Since $\phi$ is arbitrary, this shows that 
		\[
		 (H_{\omega_1}^{V[\bar{G}][\bar{H}]},\tau_{\ST}^{V[\bar{G}]},A^{V[\bar{G}]}: A\in\mathcal{A})\prec_{n+1}
		 (H_{\omega_1}^{V[\bar{G}][\bar{H}]},\tau_{\ST}^{V[\bar{G}][\bar{H}]}, A^{V[\bar{G}][\bar{H}]}: A\in\mathcal{A}),
		 \]
		concluding the proof of the inductive step for $\bar{G}$ and $\bar{H}$.
		
		Since we have class many Woodin,
		this argument is modular in $\bar{G},\bar{H}$ as in the assumptions of the inductive step,
		because we can always find some Woodin cardinal $\gamma$
		of $V$ which remains Woodin in $V[\bar{G}]$ and is of size larger than the poset 
		in $V[\bar{G}]$ for which 
		$\bar{H}$ is $V[\bar{G}]$-generic.
		The proof of the inductive step is completed.
		\end{proof}
%

\end{proof}

\subsection{Generic invariance for the universal fragment of the theory of $V$ with predicates for the non-stationary ideal and for universally Baire sets}\label{subsec:geninvtoa}

The results of this section are the key to establish
Thm. \ref{thm:mainthm1}(\ref{thm:mainthm1.4}) for the signature $A_1$.
The proofs require some familiarity with the basics of the $\Pmax$-technology and 
with Woodin's stationary tower forcing.

\begin{notation}\label{not:keynotation-genabsthordar}
\emph{}

\begin{itemize}
\item
$\tau_{\NS_{\omega_1}}$ is the signature $\tau_\ST\cup\bp{\omega_1}\cup\bp{\NS_{\omega_1}}$ with $\omega_1$ 
a constant symbol, $\NS_{\omega_1}$ a unary predicate symbol.

%
\item
$T_{\NS_{\omega_1}}$ is the $\tau_{\NS_{\omega_1}}$-theory
given by $T_\ST$ together with the axioms
\[
\omega_1\text{ is the first uncountable cardinal},
\]
\[
\forall x\;[(x\subseteq\omega_1\text{ is non-stationary})\leftrightarrow\NS_{\omega_1}(x)].
\]

\item
$\ZFC^-_{\NS_{\omega_1}}$ is the $\tau_{\NS_{\omega_1}}$-theory 
\[
\ZFC^-_\ST+T_{\NS_{\omega_1}}.
\]
\item
Accordingly we define $\ZFC_{\NS_{\omega_1}}$.
\end{itemize}
\end{notation}

The following is the key to establish
Thm. \ref{thm:mainthm1}(\ref{thm:mainthm1.4}) for the signature $A_2$.

\begin{Theorem}\label{thm:PI1invomega2}
Assume $(V,\in)$ models $\ZFC+$\emph{ there are class many Woodin cardinals}.
Then the $\Pi_1$-theory of $V$ for the language 
$\tau_{\NS_{\omega_1}}\cup\UB^V$ is invariant under set sized forcings\footnote{Here we consider any $A\subseteq (2^\omega)^k$ in $\UB^V$ as a predicate symbol of arity $k$.}.
\end{Theorem}

Asper\'o and Veli\v{c}kovi\`c provided the following basic counterexample to the conclusion of the theorem
if large cardinal assumptions are dropped.
\begin{remark}
Let $\phi(y)$ be the $\Delta_1$-property in $\tau_{\NS_{\omega_1}}$
\[
\exists y (y=\omega_1 \wedge L_{y+1}\models y=\omega_1).
\]
Then $L$ models this property, while the property fails in any forcing extension of $L$ which collapses 
$\omega_1^L$ 
to become countable.
\end{remark}

In order to prove the Theorem we need to recall some basic terminology and facts about iterations of countable structures.
\subsubsection{Generic iterations of countable structures}
\begin{definition}\cite[Def. 1.2]{HSTLARSON}
Let $M$ be a transitive countable 
model of $\ZFC$. 
Let $\gamma$ be an ordinal less than or equal to $\omega_1$. 
An iteration $\mathcal{J}$ of $M$ of length $\gamma$ 
consists of models $\ap{M_\alpha:\,\alpha \leq\gamma}$, sets $\ap{G_\alpha:\,\alpha< \gamma}$ 
and a commuting family of elementary embeddings 
\[
\ap{j_{\alpha\beta}: M_\alpha\to M_\beta:\, \alpha\leq\beta\leq\gamma}
\]
such that:
\begin{itemize}
\item
$M_0 = M$,
\item
each $G_\alpha$ is an $M_\alpha$-generic filter for 
$(\pow{\omega_1}/\NS_{\omega_1})^{M_\alpha}$,
\item
each $j_{\alpha\alpha}$ is the identity mapping,
\item
each $j_{\alpha\alpha+1}$ is the ultrapower embedding induced by $G_\alpha$,
\item
for each limit ordinal $\beta\leq\gamma$,
$M_\beta$ is the direct limit of the system
$\bp{M_\alpha, j_{\alpha\delta} :\, \alpha\leq\delta<\beta}$, and for each $\alpha<\beta$, $j_{\alpha\beta}$ is the induced embedding.
\end{itemize}
\end{definition}

We adopt the convention to denote an iteration $\mathcal{J}$ just
by $\ap{j_{\alpha\beta}:\, \alpha\leq\beta\leq\gamma}$, we also stipulate that
if $X$ denotes the domain of $j_{0\alpha}$, $X_\alpha$ or $j_{0\alpha}(X)$ will denote
the domain of $j_{\alpha\beta}$ for any $\alpha\leq\beta\leq\gamma$.

\begin{definition}
Let $A$ be  a universally Baire sets of reals.
$M$ is $A$-iterable if:
\begin{enumerate}
\item $M$ is transitive and such that $H_{\omega_1}^M$ is countable.
\item 
$M\models\ZFC+\NS_{\omega_1}$\emph{ is precipitous}.
\item
Any iteration 
\[
\bp{j_{\alpha\beta}:\alpha\leq\beta\leq\gamma}
\] 
of $M$ is well founded and such that 
$A\cap M_\beta=j_{\alpha\beta}(A\cap M_0)$ for all $\beta\leq\gamma$.
\end{enumerate}
\end{definition}

\subsubsection{Proof of Theorem \ref{thm:PI1invomega2} }

\begin{proof}
Let $\phi$ be a $\Pi_1$-sentence for $\tau_{\NS_{\omega_1}}\cup\UB^V$ which holds in $V$.
Assume that for some forcing notion $P$, $\phi$ fails in $V[h]$ with $h$ $V$-generic for $P$.
By forcing over $V[h]$ with the appropriate stationary set preserving (in $V[h]$) 
forcing notion (using a Woodin cardinal $\gamma$ of $V[h]$), we may assume that $V[h]$ is extended to a
generic extension $V[g]$ such that $V[g]$ models $\NS_{\omega_1}$ is 
saturated\footnote{A result of Shelah whose outline can be found in \cite[Chapter XVI]{SHEPRO}, or \cite{woodinBOOK}, or in an \href{https://ivv5hpp.uni-muenster.de/u/rds/sat_ideal_better_version.pdf}{handout} of Schindler available on his webpage.}.
Since $V[g]$ is an extension of $V[h]$ by a stationary set preserving forcing and there are in $V[h]$ class many Woodin cardinals, we get that
$V[h]\sqsubseteq V[g]$ with respect to the signature $\tau_{\NS_{\omega_1}}\cup\UB^V$.
Since $\Sigma_1$-properties are upward absolute and $\neg\phi$ holds in $V[h]$, 
$\phi$ fails in $V[g]$ as well.

Let $\delta$ be inaccessible in $V[g]$ and let $\gamma>\delta$ be a Woodin cardinal.

Let $G$ be $V$-generic for $\tow{T}^{\omega_1}_\gamma$ (the countable tower $\mathbb{Q}_{<\gamma}$ according to \cite[Section 2.7]{STATLARSON})
and such that  $g\in V[G]$.
Let $j_G:V\to\Ult(V,G)$ be the induced ultrapower embedding.

Now remark that $V_\delta[g]\in \Ult(V,G)$ is $B^{V[G]}$-iterable for all 
$B\in \mathsf{UB}^{V}$ (since 
$V_\eta[g]\in \Ult(V,G)$ for all $\eta<\gamma$, and this suffices to check that $V_\delta[g]$ 
is $B^{V[G]}$-iterable for all $B\in\mathsf{UB}^V$, see \cite[Thm. 4.10]{HSTLARSON}).

By \cite[Lemma 2.8]{HSTLARSON} applied in $\Ult(V,G)$, there exists in $\Ult(V,G)$ an iteration 
$\mathcal{J}=\bp{j_{\alpha\beta}:\alpha\leq\beta\leq\gamma=\omega_1^{\Ult(V,G)}}$ of 
$V_\delta[g]$ such that
$\NS_{\omega_1}^{X_{\gamma}}=\NS_{\omega_1}^{\Ult(V,G)}\cap X_{\gamma}$, where 
$X_\alpha=j_{0\alpha}(V_\delta[g])$ for all $\alpha\leq\gamma=\omega_1^{\Ult(V,G)}$.

This gives that $X_{\gamma}\sqsubseteq \Ult(V,G)$ for $\tau_{\NS_{\omega_1}}\cup\UB^V$.
Since $V_\delta[g]\models\neg\phi$, so does $X_{\gamma}$, by elementarity.
But $\neg\phi$ is a $\Sigma_1$-sentence, hence it is upward absolute for superstructures, therefore
$\Ult(V,G)\models\neg\phi$. This is a contradiction, since $\Ult(V,G)$ is elementarily equivalent to $V$ for
$\tau_{\NS_{\omega_1}}\cup\UB^V$, and $V\models\phi$.

\smallskip

A similar argument shows that if $V$ models a $\Sigma_1$-sentence $\phi$ for 
$\tau_{\NS_{\omega_1}}\cup\UB^V$
this will remain true in all of its generic extensions:

Assume $V[h]\models\neg\phi$
for some $h$ $V$-generic for some forcing notion $P\in V$. 
Let $\gamma>|P|$ be a Woodin cardinal, and let $g$ be $V$-generic for\footnote{$\tow{T}_\gamma$ is the full stationary tower of height $\gamma$ whose conditions are 
stationary sets in $V_\gamma$, denoted as $\mathbb{P}_{<\gamma}$ in \cite{STATLARSON},
see in particular \cite[Section 2.5]{STATLARSON}.}
 $\tow{T}_\gamma$ with $h\in V[g]$ and $\crit(j_g)=\omega_1^V$ (hence there is in $g$ some stationary set of $V_\gamma$ concentrating on countable sets). 
Then $V[g]\models\phi$ since:
\begin{itemize} 
\item
$V_\gamma\models\phi$, since $V_\gamma\prec_{1} V$ for $\tau_{\NS_{\omega_1}}\cup\UB^V$
by Lemma \ref{lem:levabsgen};
\item
$V_\gamma^{\Ult(V,g)}=V_\gamma^{V[g]}$, since $V[g]$ models that 
$\Ult(V,g)^{<\gamma}\subseteq \Ult(V,g)$;
\item
$V_\gamma^{\Ult(V,g)}\models\phi$, by elementarity of $j_g$, since 
$j_g(V_\gamma)=V_\gamma^{\Ult(V,g)}$;
\item
$V_\gamma^{V[g]}\prec_{\Sigma_1}V[g]$ with respect to $\tau_{\NS_{\omega_1}}\cup\UB^V$,
again by Lemma \ref{lem:levabsgen} applied in $V[g]$.
\end{itemize}

Now repeat the same argument as before to the $\Pi_1$-property $\neg\phi$,
with $V[h]$ in the place of $V$ and $V[g]$ in the place of $V[h]$. 
\end{proof}


\section{Model companionship versus generic absoluteness for the theory of $H_{\aleph_1}$}\label{sec:Homega1}

\subsection{Model companionship for the theory of $H_{\aleph_1}$}

\begin{notation}
Let $\tau\supseteq\tau_\ST$ be a signature.
$\ZFC_\tau$ is the theory extending $\ZFC$ with the replacement schema for all $\tau$-formulae.
Accordingly we define $\ZFC^-_\tau$.
\end{notation}

\begin{definition}

Let $S$ be a $\tau$-theory extending $\ZFC_\tau$. 

$\tau\supseteq\tau_\ST$ is a projective signature for $S$ if any $\tau$-model $\mathcal{M}$ of
$S$ interprets:
\begin{itemize}
\item 
all predicate symbols of arity $k$ of 
$\tau\setminus \tau_\ST$ as subsets of $(2^\omega)^k$ (as defined in $\mathcal{M}$), 
\item
all function symbols of arity $k$ of 
$\tau\setminus \tau_\ST$ as functions from $(2^\omega)^k$ to $2^\omega$ (as defined in $\mathcal{M}$), 
\item 
all constant symbols of 
$\tau\setminus \tau_\ST$ as elements of $2^\omega$ (as defined in $\mathcal{M}$).
\end{itemize}

Assume $\tau$ is a projective signature for $S\supseteq \ZFC_\tau$.

$A\subseteq F_{\tau}$ is $S$-projectively closed if:
\begin{enumerate}[(A)] 
\item
$A$ is closed under logical equivalence;
\item
for any $(V,\tau)$ model of $S$, any formula in $A$ defines a subset
of $((2^\omega)^V)^k$ for some $k\in\omega$;
\item 
in any model $(V,\tau)$ of $S$, 
if $B$ is a definable subset of $((2^\omega)^V)^k$ in the structure
\[
(H_{\omega_1}^V,\tau^V,R_\phi^V,f_\phi^V:\phi\in A),
\]
then $B=R_\psi^V$ for some $\psi\in A$.
\end{enumerate}
\end{definition}

\begin{example}
Given a $\tau_\ST$-theory $T$ extending $\ZFC_\ST$,
simple examples of $T$-projectively closed families for $\tau_\ST$ (which we will use) are:
\begin{enumerate}
\item The family of lightface definable projective sets of reals.
\item
$\lUB^T$, i.e. the $\in$-formulae  defining subsets of $(2^\omega)^k$ 
(as $k$ varies in the natural numbers) which $T$ proves to be the extension
of some $\in$-formula relativized to $L(\UB)$ (the smallest transitive model of $\ZF$
containing all the ordinals and the universally Baire sets).

\item
If  $(V,\tau_\ST^V)$ models the existence of class many Woodin cardinals,
$X\prec (V_\theta,\in)$ for a large enough $\theta$, and
$T_X$ is the $\tau_\ST\cup(\UB^V\cap X)$-theory of 
$V$, one also get that $\tau_\ST\cup(\UB^V\cap X)$ is a projective signature for $T_X$ and
$\UB^V\cap X$ is  $T_X$-projectively closed
(where a universally Baire subset of $(2^\omega)^k$ is considered a predicate symbol of arity $k$; note that $X=V_\theta$ --- i.e. $\UB^V\cap X=\UB^V$ --- is possible).
\end{enumerate}
\end{example}

\begin{theorem}\label{thm:modcompanHomega1}
Let $\tau\supseteq\tau_\ST$ and $S$ be a $\tau$-theory
extending $\ZFC_\tau$ such that $\tau$ is a projective signature for $S$.

Let $A\subseteq F_\tau$ be an $S$-projectively closed family for  $\tau$ and
\[
\bar{A}=A\times\bp{0,1}.
\] 
Then $S_{\bar{A}}$ has as its model companion
 in signature 
$\tau_{\bar{A}}$
\[
S^*_{\bar{A}}=\bp{\phi: (H_{\omega_1}^V,\tau_{\bar{A}}^V)\models \phi,\, (V,\in)\models S}.
\] 
\end{theorem}

It is clear that the above theorem combined with the results of Section
\ref{sec:geninv} proves Thm. \ref{thm:mainthm1} and 
Corollary \ref{cor:maincor} for $A_1$.
More precisely:
\begin{corollary}\label{cor:modcompanHomega1}
Let $S\supseteq \ZFC+$\emph{there are class many Woodin cardinals}
be a $\in$-theory.
Then for any $A\subseteq F_\in$ projectively closed for $S$ and such that 
$\phi$ defines a universally Baire set of reals for any 
$\phi$ in $A$ not a $\Delta_0$-formula, letting $\bar{A}=A\times\bp{0,1}$,
$S+T_{\bar{A}}$ has as model companion
the $\Pi_2$-sentences $\psi$ for $\bp{\in}_{\bar{A}}$ such that
\[
S\vdash\psi^{H_{\omega_1}}.
\]
\end{corollary}

\begin{proof}

Let $(V,\tau^V)$ be a model of $S$. 


By Levy's absoluteness Lemma \ref{lem:levabsgen}, since $A$ includes just formulae definining
subsets of $(2^{\omega})^k$ and the same occurs for the symbols of $\tau\setminus\tau_\ST$ in models of $S$,
\[
(H_{\omega_1},\tau^V,R_\psi^V, f_\psi^V:\psi\in A )\prec_1(V,\tau^V,R_\psi^V, f_\psi^V:\psi\in A);
\]
hence 
the structures $(V,\tau^V,R_\psi^V, f_\psi^V:\psi\in A )$ and
$(H_{\omega_1},\tau^V,R_\psi^V, f_\psi^V:\psi\in A )$ share the same $\Pi_1$-theory for the signature
$\tau_{\bar{A}}$.

Therefore (by the useful characterization of model companionship given in Lemma \ref{fac:proofthm1-2}) 
it suffices to prove that $S^*$ is model complete, where $S^*$ is the $\tau_{\bar{A}}$-theory common to 
$(H_{\omega_1},\tau^V,R_\psi^V, f_\psi^V:\psi\in A )$ as $(V,\tau^V)$ range over
models of $S$.

By Robinson's test (Lemma \ref{lem:robtest}\ref{lem:robtest-4}), 
it suffices to show that any existential $\tau_{\bar{A}}$-formula is 
$S^*$-equivalent to a universal $\tau_{\bar{A}}$-formula.

Let $\psi_1,\dots,\psi_k$ be the formulae in $A$ such that some $R_{\psi_i}$ or some $f_{\psi_i}$
appears in $\phi$.

Let
$\psi(x_1,\dots,x_n)$ be the formula
$\phi(\Cod_\omega(x_1),\dots,\Cod_\omega(x_n))$. 
Since $\Cod_\omega(x)=y$ is  a $\Delta_1$-definable predicate in the structure $(H_{\omega_1},\tau_\ST)$,
we get that $\psi(x_1,\dots,x_n)$ in $A$ since its extension is a subset of $(2^\omega)^k$ in the structure
\[
(H_{\omega_1},\tau^V,R_\psi^V, f_\psi^V:\psi\in A ).
\]
Now for any $a_1,\dots,a_n\in H_{\omega_1}$:
\[
(H_{\omega_1},\tau_{\bar{A}}^V)\models \phi(a_1,\dots,a_n)
\]
\center{ if and only if }
\[
(H_{\omega_1},\tau_{\bar{A}}^V)\models\forall r_1\dots r_n \bigwedge_{i=1}^n\Cod_\omega(r_i)=a_i\rightarrow R_\psi(r_1,\dots,r_n).
\]



This yields that
\[
S^*\vdash 
\forall x_1,\dots,x_n\,(\phi(x_1,\dots,x_n)\leftrightarrow\theta_\psi(x_1,\dots,x_n)).
\]
where $\theta_\phi(x_1,\dots,x_n)$ is the $\Pi_1$-formula in the predicate $R_\psi\in\tau_{\bar{A}}$
\[
\forall y_1,\dots,y_n\,[(\bigwedge_{i=1}^n x_i=\Cod_\omega(y_i))\rightarrow R_\psi(y_1,\dots,y_n)].
\]
\end{proof}

%

It is also convenient to reformulate these notion is a more semantic way which is handy when dealing with a fixed complete first order axiomatization of set theory.

\begin{definition}\label{def:Homega1closed}
Let $\mathcal{A}\subseteq\bigcup_{n\in\omega}\pow{\omega}^n$.
$\mathcal{A}$ is $H_{\omega_1}$-closed if 
any definable subset of $\pow{\omega}^n$ for some $n\in\omega$ in the structure
\[
(H_{\omega_1},\in,U:U\in\mathcal{A})
\]
is in $\mathcal{A}$.
\end{definition}
It is immediate to check that if $T$ is the theory of $(V,\in)$ and $\mathcal{A}$ is a family of universlly Baire subsets of $V$, $\mathcal{A}$ is projectively closed for $T$  for the signature $\tau_\ST\cup\mathcal{A}$ if and only if it
is $H_{\omega_1}$-closed.

We get the following:
\begin{theorem}\label{fac:keyfacHomega1clos}
Assume $(V,\in)$ models $\ZFC+$\emph{there are class many Woodin cardinals}.
Let $\mathcal{A}\subseteq \UB^V$ be $H_{\omega_1}$-closed and $\tau_\mathcal{A}=\tau_\ST\cup\mathcal{A}$
 be the signature in which each element of $\mathcal{A}$ contained in $\pow{\omega}^k$ is a predicate symbol of arity $k$.
 Then for any $G$ $V$-generic for some forcing $P\in V$
 the $\tau_\mathcal{A}$-theory of $H_{\omega_1}^V$ is the model companion of the  
 $\tau_\mathcal{A}$-theory of $V[G]$ and $\bp{A^{V[G]}:A\in\mathcal{A}}$
 is $H_{\omega_1}^{V[G]}$-closed.
\end{theorem}

\begin{proof}
The assumptions grant that
\[
(H_{\omega_1}^{V},\tau_\ST^{V},A:A\in\mathcal{A})\prec
(H_{\omega_1}^{V[G]},\tau_\ST^{V[G]},A:A^{V[G]}\in\mathcal{A})
\prec_1 (V[G],\tau_\ST^{V[G]},A:A^{V[G]}\in\mathcal{A})
\]
(by Thm. \ref{thm:genabshomega1} and by Lemma \ref{lem:levabsgen} applied in $V[G]$).
Now the theory of $H_{\omega_1}^{V}$ in signature $\tau_{\mathcal{A}}$ is complete and model complete,
and is also the $\tau_{\mathcal{A}}$-theory of $H_{\omega_1}^{V[G]}$.
We conclude that 
it is the model companion of the $\tau_{\mathcal{A}}$-theory of $V[G]$.
It is also easy to check that
$\bp{A^{V[G]}:A\in\mathcal{A}}$
 is $H_{\omega_1}^{V[G]}$-closed.
\end{proof}

\section{Model companionship versus generic absoluteness for the theory of $H_{\aleph_2}$}\label{sec:Homega2}

Let $\UB$ denote the family of universally Baire sets, and $L(\UB)$ denote 
the smallest transitive model of $\ZF$ which contains $\UB$ (see for details 
Section \ref{subsec:univbaire}).

Our first result shows that in models of large cardinal axioms admitting a 
strong form of sharp for $\UB$ (what is here called $\maxUB$),
a strong form of Woodin's axiom $(*)$ (what is here called $\stUB$)
can be equivalently formulated as 
the assertion that the theory of
$H_{\aleph_2}$ is the model companion of the theory of $V$ in a signature 
admitting a predicate symbol for the non-stationary ideal on $\omega_1$ and 
predicates for each universally Baire set.

\begin{Theorem}\label{Thm:mainthm-1bis}
Let $\mathcal{V}=(V,\in)$ be a model of 
\[
\ZFC+\maxUB+\emph{there is a supercompact cardinal and class many Woodin cardinals},
\]
and $\UB$ denote the family of universally Baire sets in 
$V$.

TFAE
\begin{enumerate}
\item\label{thm:char(*)-modcomp-1}
$(V,\in)$ models $\stUB$;
\item\label{thm:char(*)-modcomp-2}
$\NS_{\omega_1}$ is precipitous\footnote{See \cite[Section 1.6, pag. 41]{STATLARSON}  for a definition of precipitousness and a discussion of its properties. A key observation is that $\NS_{\omega_1}$ being precipitous is independent of $\mathsf{CH}$ (see for example \cite[Thm. 1.6.24]{STATLARSON}), while $\stUB$ entails $2^{\aleph_0}=\aleph_2$  (for example by the results of \cite[Section 6]{HSTLARSON}).

Another key point is that we stick to the formulation of $\Pmax$ as in \cite{HSTLARSON} so to 
be able in its proof to quote verbatim from \cite{HSTLARSON} all the relevant results on $\Pmax$-preconditions we will use.
It is however possible to  develop $\Pmax$ focusing on Woodin's countable tower rather than 
on the precipitousness of $\NS_{\omega_1}$ to develop the notion of $\Pmax$-precondition. Following this approach in
all its scopes, one should be able to reformulate Thm. \ref{Thm:mainthm-1bis}(\ref{thm:char(*)-modcomp-2}) 
omitting the request that
$\NS_{\omega_1}$ is precipitous. We do not explore this venue any further.} and
the $\tau_{\NS_{\omega_1}}\cup\UB$-theory of $V$ has as model companion the
$\tau_{\NS_{\omega_1}}\cup\UB$-theory of $H_{\omega_2}$.
\end{enumerate}
\end{Theorem}

(\ref{thm:char(*)-modcomp-1}) implies (\ref{thm:char(*)-modcomp-2}) does not 
need the supercompact cardinal.

We give rightaway the definitions of $\maxUB$ and $\stUB$.

\begin{Definition}\label{Keyprop:maxUB}
$\maxUB$: 
There are class many Woodin cardinals in $V$, and  for all 
$G$ $V$-generic for some forcing notion $P\in V$:
\begin{enumerate}
\item\label{Keyprop:maxUB-1}
Any subset of $(2^\omega)^{V[G]}$ definable in $(H_{\omega_1}^{V[G]}\cup\mathsf{UB}^{V[G]},\in)$ 
is universally Baire in $V[G]$.
\item\label{Keyprop:maxUB-2}
Let $H$ be $V[G]$-generic for some forcing notion $Q\in V[G]$. 
Then\footnote{Elementarity is witnessed via the map defined by $A\mapsto A^{V[G][H]}$ for
$A\in \UB^{V[G]}$ and the identity on
$H_{\omega_1}^{V[G]}$ (See Notation \ref{not:notUBsetsVG} 
for the definition of $A^{V[G][H]}$).}:
\[
(H_{\omega_1}^{V[G]}\cup\UB^{V[G]},\in) \prec 
(H_{\omega_1}^{V[G][H]}\cup\UB^{V[G][H]},\in).
\] 
\end{enumerate}
\end{Definition}

We observe that $\maxUB$ is a form of sharp for the family of universally Baire sets 
which holds if $V$ has class many Woodin cardinals and 
is a generic extension obtained by collapsing a supercompact cardinal to become countable 
($\maxUB$ is a weakening of the conclusion of
\cite[Thm 3.4.17]{STATLARSON}). Moreover if $\maxUB$ holds in $V$, it
remains true in all further set forcing extensions of $V$. It is open whether $\maxUB$ is a direct consequence of 
suitable large cardinal axioms.

We now turn to the definition of $\stUB$, a natural maximal strengthening of Woodin's axiom $(*)$.
Key to all results of this section is an analysis of the properties of
generic extensions by $\Pmax$ of $L(\UB)$.
In this analysis $\maxUB$ is used to argue (among other things)
that all sets of reals definable in $L(\UB)$ are universally Baire, so that most of the results established 
in \cite{HSTLARSON}
on the properties of $\Pmax$ for $L(\mathbb{R})$ can be
also asserted for $L(\UB)$.
We will use various forms of Woodin's axiom $(*)$ each stating that $\NS_{\omega_1}$ is saturated together with the existence of $\Pmax$-filters meeting certain families of dense subsets 
of $\Pmax$ definable in $L(\UB)$. 
However in this paper we do not define the $\Pmax$-forcing. 
The reason is that in the proof of all our results, we will use equivalent characterizations of the proper forms of $(*)$ which do not mention at all $\Pmax$. 
We will give at the proper stage the relevant definitions. Meanwhile we assume the reader is familiar with $\Pmax$ or can accept as a blackbox its existence as a certain forcing notion; our reference on this topic 
is \cite{HSTLARSON}.

\begin{Definition}
Let $\mathcal{A}$ be a family of dense subsets of $\Pmax$.
\begin{itemize}
\item
$\stA$ holds if $\NS_{\omega_1}$ is saturated\footnote{See \cite[Section 1.6, pag. 39]{STATLARSON} for a discussion of saturated ideals on $\omega_1$.} and 
there exists a filter $G$ on $\Pmax$ meeting all the dense sets in 
$\mathcal{A}$.
\item
$\stUB$ holds
if $\NS_{\omega_1}$ is saturated and there exists an $L(\UB)$-generic filter $G$ on $\Pmax$. 
\end{itemize}
\end{Definition}

Woodin's definition of $(*)$ \cite[Def. 7.5]{HSTLARSON}
is equivalent to $\stA+$\emph{there are class many Woodin cardinals} 
for $\mathcal{A}$ the family of 
dense subsets of $\Pmax$ existing in $L(\mathbb{R})$.

An objection to Thm. \ref{Thm:mainthm-1bis} is that it subsumes the Platonist standpoint that there exists a definite universe of sets.
At the prize of introducing another bit of notation, we can prove a version of Thm. \ref{Thm:mainthm-1bis} which makes perfect sense also
to a formalist and from which we immediately derive Thm. \ref{thm:mainthm1} and Corollary \ref{cor:maincor} for a certain recursive set of $\in$-formulae $A_2$.

\begin{Notation}
\emph{}

\begin{itemize}
\item
$\sigma_\ST$ is the signature containing a 
predicate symbol 
$S_\phi$ of arity $n$ for any $\in$-formula $\phi$ with $n$-many free variables.

\item
$T_\lUB$ is the 
$\sigma_{\ST}$-theory 
given by the axioms
\[
\forall x_1\dots x_n\,[S_\psi(x_1,\dots,x_n)\leftrightarrow 
(\bigwedge_{i=1}^n x_i\subseteq \omega^{<\omega}\wedge \psi^{L(\UB)}(x_1,\dots,x_n))]
\]
as $\psi$ ranges over the $\in$-formulae.
\item
$\ZFC^{*-}_{\lUB}$ is the $\sigma_{\omega}=\sigma_\ST\cup\tau_\ST$-theory 
\[
\ZFC^-_{\ST}\cup T_\lUB.
\]
\item
$\sigma_{\omega,\NS_{\omega_1}}$ is
the signature $\tau_{\NS_{\omega_1}}\cup\sigma_\ST$ (recall Notation \ref{not:keynotation-genabsthordar}).
\item
$\ZFC^{*-}_{\lUB,\NS_{\omega_1}}$ is the $\sigma_{\omega,\NS_{\omega_1}}$-theory 
\[
\ZFC^-_{\NS_{\omega_1}}\cup T_\lUB.
\]
\item
Accordingly we define $\ZFC^{*}_{\lUB}$, $\ZFC^*_{\lUB,\NS_{\omega_1}}$.
\end{itemize}
\end{Notation}

A key observation is that $\ZFC^-_\ST$,
$\ZFC^-_{\NS_{\omega_1}}$, $\ZFC^{*-}_{\lUB}$, $\ZFC^{*-}_{\lUB,\NS_{\omega_1}}$ are all 
\emph{definable} extension of $\ZFC^-$; more precisely: there are sets $X\subseteq F_{\bp\in}\times 2$ such that each of the above theory is of the form $\ZFC^-+T_X$ according to Def. \ref{def:compspectrum}.
The same applies to 
$\ZFC_\ST$,
$\ZFC_{\NS_{\omega_1}}$, $\ZFC^{*}_{\lUB}$, $\ZFC^{*}_{\lUB,\NS_{\omega_1}}$.

\begin{Theorem} \label{Thm:mainthm-1ter}
Let $T$ be any $\sigma_{\omega,\NS_{\omega_1}}$-theory extending 
\[
\ZFC^*_{\lUB,\NS_{\omega_1}}+\maxUB+\text{ there is a supercompact cardinal and  class many Woodin cardinals}
\]
Then $T$ has a model companion $T^*$. 

Moreover TFAE for any for any $\Pi_2$-sentence $\psi$ for 
$\sigma_{\omega,\NS_{\omega_1}}$:
\begin{enumerate}[(A)]
\item \label{Thm:mainthm-1A}
$T^*\vdash \psi$.
\item \label{Thm:mainthm-1E}
\[
(V[G],\sigma_{\omega,\NS_{\omega_1}}^{V[G]})\models\psi^{H_{\omega_2}}
\]
whenever $(V,\sigma_{\omega,\NS_{\omega_1}}^{V})\models T$, $V[G]$ is a forcing extension of $V$, 
and $V[G]\models\stUB$.
\item \label{Thm:mainthm-1Cbis}
$T$ proves\footnote{$\dot{H}_{\omega_2}$ denotes a canonical $P$-name for $H_{\omega_2}$ as computed in generic extension by $P$.} 
\[
\exists P \,(P\text{ is a \emph{stationary set preserving} partial order }\wedge 
\Vdash_P\psi^{\dot{H}_{\omega_2}}).
\]
\item \label{Thm:mainthm-1C}
$T$ proves
\[
\exists P \,(P\text{ is a partial order }\wedge 
\Vdash_P\psi^{\dot{H}_{\omega_2}}).
\]
\item \label{Thm:mainthm-1D}
$T$ proves
\[
L(\UB)\models[\Pmax\Vdash \psi^{\dot{H}_{\omega_2}}].
\]
\item \label{Thm:mainthm-1Dbis}
If $(V,\sigma_{\omega,\NS_{\omega_1}}^{V})\models T$ 
and $\psi$ is $\forall x\exists y\,\phi(x,y)$ with $\phi$ quantifier free 
$\sigma_{\omega,\NS_{\omega_1}}$-formula, then for all\footnote{See Def. \ref{def:aspschhoncons} for the notion of honest consistency.}  $a\in H_{\omega_2}^V$
\[
\exists y\phi(a,y)\text{ is \emph{honestly consistent} according to $V$}.
\]
\item \label{Thm:mainthm-1B}
For any complete theory 
\[
S\supseteq T,
\] 
$S_\forall\cup\bp{\psi}$ is consistent.

\end{enumerate}

\end{Theorem}

Note that even if $T\models\CH$, $\neg\CH$ is in $T^*$ (for example by \ref{Thm:mainthm-1D} above).
In particular the model companion $T^*$ of $T$ may have models whose theory of $H_{\aleph_2}$ is
completely unrelated to that of models of $T$.
Moreover recall again that $\CH$ is not expressible as a $\Pi_1$-property in $\sigma_{\omega,\NS_{\omega_1}}$ for $T$: it is not preserved by forcing, while $T_\forall$ is.

The rest of this section is devoted to proof of Theorems~\ref{Thm:mainthm-1bis} and \ref{Thm:mainthm-1ter}.

Crucial to their proof is the recent breakthrough of 
Asper\'o and Schindler \cite{ASPSCH(*)} establishing that $(*)$-$\UB$
follows from $\MM^{++}$. 

First of all it is convenient to detail more on $\maxUB$ and its use in our proofs.

\subsection{$\maxUB$}

From now on we will need in several occasions that $\maxUB$ holds in $V$ (recall Def. \ref{Keyprop:maxUB}).
We will always explicitly state where this assumption is used, hence if a statement
 does not mention it in the hypothesis, 
the assumption is not needed for its thesis. 

 We will use both properties of 
$\maxUB$ crucially: (\ref{Keyprop:maxUB-1}) is used in the proof of Lemma \ref{lem:UBcorr};
(\ref{Keyprop:maxUB-2}) in the proof of Fact \ref{fac:densityUBcorrect}.
Similarly they are essentially used in Remark \ref{rmk:maxUBec}.
Specifically we will need $\maxUB$ to prove that 
certain subsets of $H_{\omega_1}$
simply definable using an existential formula quantifying over $\UB$ 
are coded by a universally Baire set, 
and that this coding is absolute between
generic extensions, i.e. if 
\[
\bp{x\in H_{\omega_1}^V: (H_{\omega_1}\cup\UB,\tau_\ST^V)\models \phi(x)}
\] 
is coded by $A\in \UB^V$,
\[
\bp{x\in H_{\omega_1}^{V[G]}: (H_{\omega_1}^{V[G]}\cup\UB^{V[G]},\tau_\ST^{V[G]}))\models \phi(x)}
\] 
is coded by $A^{V[G]}\in \UB^{V[G]}$ for $\phi$ some $\tau_\ST$-formula\footnote{Note that
the structures $(H_{\omega_1}\cup\UB,\in)$, $(H_{\omega_1}\cup\UB,\tau_\ST^V)$ have the same algebra of definable sets, hence we will use one or the other as we deem most convenient,
since any set definable by some formula in one of these structures is also defined by a possibly different formula in the other. The formulation of $\maxUB$ is unaffacted if we choose any of the two structures as the one for which we predicate it.}.

It is useful to outline what is the different expressive power of the structures
 $(H_{\omega_1},\tau_{\ST}^V,A: A\in\UB^V)$ and 
 $(H_{\omega_1}\cup\UB^V,\tau_{\ST}^V)$.
 The latter can be seen as a second order extension of $H_{\omega_1}$, where
we also allow formulae to quantify over the family of universally Baire subsets of $2^{\omega}$;
in the former quantifiers only range over elements of $H_{\omega_1}$, but we can 
use the universally Baire subsets of $H_{\omega_1}$ as parameters.
This is in exact analogy between the comprehension scheme for the Morse-Kelley axiomatization of set theory
(where formulae with quantifiers ranging over classes are allowed) and the 
comprehension scheme for G\"odel-Bernays axiomatization of set theory
(where just formulae using classes as parameters and quantifiers ranging only over sets are allowed).
To appreciate the difference between the two set-up, 
note that that the axiom of determinacy for universally Baire sets
is expressible in
\[
(H_{\omega_1}\cup\UB,\tau_{\ST}^V)
\]
by the
$\tau_\ST$-sentence
\begin{quote}
\emph{For all $A\subseteq 2^\omega$ there is a winning strategy for one of the players in the game with payoff 
$A$},
\end{quote}
while in 
\[
(H_{\omega_1},\tau_{\ST}^V,A:A\in\UB^V)
\]
it is expressed by the axiom schema of $\Sigma_1$-sentences for $\tau_\ST\cup\bp{A}$
\begin{quote}
\emph{There is a winning strategy for some player in the game with payoff 
$A$}
\end{quote}
as $A$ ranges over the universally Baire sets.

We will crucially use the stronger expressive power of the structure $(H_{\omega_1}\cup\UB,\tau_\ST)$
to define 
certain universally Baire sets as the extension in $(H_{\omega_1}\cup\UB,\tau_\ST^V)$ of lightface 
$\Sigma_1$-properties (according to the Levy hierarchy); properties which 
require an existential quantifier ranging over all universally Baire sets.

\subsection{A streamline of the proofs of Theorems \ref{Thm:mainthm-1bis}, \ref{Thm:mainthm-1ter}}
Let us give a general outline of these proofs before getting into details. From now on we assume 
the reader 
is familiar with the basic theory of $\Pmax$ as exposed in \cite{HSTLARSON}.

\begin{notation}
For a given family of universally Baire sets $\mathcal{A}$, 
$\tau_\mathcal{A}$ is the signature $\tau_{\ST}\cup\mathcal{A}$,
$\tau_{\mathcal{A},\NS_{\omega_1}}$ is the signature $\tau_{\NS_{\omega_1}}\cup\mathcal{A}$.
\end{notation}
The key point is to prove (just on the basis that $(V,\in)\models\maxUB+\stUB$) 
the model completeness of the
$\tau_{\UB,\NS_{\omega_1}}$-theory of
$H_{\omega_2}$ assuming $\stUB$. To do so we use Robinson's test and we show the following:

\begin{quote}
Assuming $\maxUB$ there is a \emph{special} universally Baire set
$\bar{D}_{\UB,\NS_{\omega_1}}$ defined by an $\in$-formula \emph{(in no parameters)} 
relativized to $L(\UB)$
coding a family of $\Pmax$-preconditions with the following fundamental property:

\emph{
For any 
$\tau_{\UB,\NS_{\omega_1}}$-formula 
$\psi(x_1,\dots,x_n)$ mentioning the universally Baire predicates $B_1,\dots,B_k$,
there is an algorithmic procedure which finds a \emph{universal} $\tau_{\UB,\NS_{\omega_1}}$-formula 
$\theta_\psi(x_1,\dots,x_n)$ mentioning just the universally Baire predicates 
$B_1,\dots,B_k,\bar{D}_{\UB,\NS_{\omega_1}}$ such that}
\[
(H_{\omega_2}^{L(\UB)[G]},\sigma_{\bp{B_1,\dots,B_k,\bar{D}_{\UB,\NS_{\omega_1}}},\NS_{\omega_1}}^{L(\UB)[G]})
\models
\forall\vec{x}\,(\psi(x_1,\dots,x_n)\leftrightarrow\theta_\psi(x_1,\dots,x_n))
\]
\emph{whenever $G$ is $L(\UB)$-generic for $\Pmax$.}
\end{quote}
Moreover the definition of $\bar{D}_{\UB,\NS_{\omega_1}}$ and the computation
of $\theta_\psi(x_1,\dots,x_n)$ from $\psi(x_1,\dots,x_n)$ are just based on the assumption that
$(V,\in)$ is a model of $\maxUB$, hence can be replicated mutatis-mutandis in any model of 
$\ZFC+\maxUB$. 
We will need that $(V,\in)$ is a model of $\maxUB+\stUB$ just to argue that
in $V$ there is an $L(\UB)$-generic filter $G$ for $\Pmax$ such that\footnote{It is this part of our argument
where the result of Asper\`o and Schindler establishing the consistency of $\stUB$ relative to a supercompact is used in an essential way. We will address again the role of Asper\`o and Schindler's result in all our proofs in  some closing remarks.}
$H_{\omega_2}^{L(\UB)[G]}=H_{\omega_2}^V$. 
Since in all our arguments we will only use that $(V,\in)$ is a model of $\maxUB$ and (in some of them 
also of $\stUB$),
we will be in the position to conclude easily for the truth of Theorem~\ref{Thm:mainthm-1bis} and~\ref{Thm:mainthm-1ter}.

We condense the above information in the following:

\begin{theorem}\label{thm:keythmmodcompanHomega2}
There is an $\in$-formula $\phi_{\UB,\NS_{\omega_1}}(x)$ in one free variable
such that:
\begin{enumerate}
\item
$\ZFC^*_{\lUB}+\maxUB$
proves that 
\emph{$S_{\phi_{\UB,\NS_{\omega_1}}}$ is universally Baire}.
\item
Given predicate symbols $B_1,\dots,B_k$, consider the theory $T_{B_1,\dots,B_k}$ in signature 
$\sigma_{\omega}\cup\bp{B_1,\dots,B_k}$
extending $\ZFC^*_{\lUB}+\maxUB$ by the axioms:
\[
B_j\text{ is universally Baire}
\]
for all predicate symbols $B_1,\dots,B_k$.

There is a recursive procedure assigning to any \emph{existential}
formula $\phi(x_1,\dots,x_k)$ for $\sigma_{\bp{B_1,\dots,B_k},\NS_{\omega_1}}$
a \emph{universal} formula $\theta_\phi(x_1,\dots,x_k)$ for 
$\sigma_{\bp{B_1,\dots,B_k, S_{\phi_{\UB,\NS_{\omega_1}}}},\NS_{\omega_1}}$
such that $T_{B_1,\dots,B_k}$ proves that
\[
\Pmax\Vdash 
[(H_{\omega_2}^{L(\UB)[\dot{G}]},\tau_{\UB,\NS_{\omega_1}}^{L(\UB)[\dot{G}]})
\models\forall\vec{x}\;(\phi(x_1,\dots,x_k)\leftrightarrow\theta_\phi(x_1,\dots,x_k))]
\]
where $\dot{G}\in L(\UB)$ is the canonical $\Pmax$-name for the generic filter.
\end{enumerate}

\end{theorem}

\subsection{Proofs of Thm.~\ref{Thm:mainthm-1ter}, and of 
(\ref{thm:char(*)-modcomp-1})$\to$(\ref{thm:char(*)-modcomp-2})
of Thm.~\ref{Thm:mainthm-1bis}}

Theorem~\ref{Thm:mainthm-1ter}, (\ref{thm:char(*)-modcomp-1})$\to$(\ref{thm:char(*)-modcomp-2})
of Theorem~\ref{Thm:mainthm-1bis} are immediate corollaries
of the above theorem combined with Asper\`o and Schindler's proof that 
$\MM^{++}$ implies $\stUB$, and with Theorem \ref{thm:PI1invomega2}.

We start with the proof of
(\ref{thm:char(*)-modcomp-1})$\to$(\ref{thm:char(*)-modcomp-2}) of Thm.~\ref{Thm:mainthm-1bis}
assuming Thm. \ref{thm:keythmmodcompanHomega2} and Thm. \ref{thm:PI1invomega2}:

\begin{proof}
Assume $(V,\in)$ models $\stUB$. Then there is a $\Pmax$-filter $G\in V$ such that
$H_{\omega_2}^{L(\UB)[G]}=H_{\omega_2}^V$.
By Thm. \ref{thm:keythmmodcompanHomega2} and Robinson's test, 
we get that the first order $\tau_{\UB,\NS_{\omega_1}}$-theory of 
$H_{\omega_2}^{L(\UB)[G]}$ is model complete.
By Levy's absoluteness (Lemma \ref{lem:levabsgen}), $H_{\omega_2}^{L(\UB)[G]}$ is a $\Sigma_1$-elementary substructure of $V$ also according to the signature $\tau_{\UB,\NS_{\omega_1}}$.
We conclude (by Thm. \ref{thm:uniqmodcompan}), since the two theories share the same $\Pi_1$-fragment.
\end{proof}

The proof of the converse implication requires more information on 
$\bar{D}_{\UB,\NS_{\omega_1}}$ then what is conveyed in Thm. \ref{thm:keythmmodcompanHomega2}.
We defer it to a later stage.

\smallskip

We now prove Thm.~\ref{Thm:mainthm-1ter}:
\begin{proof}
Let $T^*_{\lUB,\NS_{\omega_1}}$ be the theory given by the
$\Pi_2$-sentences $\psi$
for $\sigma_{\omega,\NS_{\omega_1}}$ which 
hold in $H_{\omega_2}^{V[G]}$ whenever $(V,\in)$ models
\[
\ZFC^*_{\lUB,\NS_{\omega_1}}+\maxUB+\text{ there is a supercompact cardinal and  class many Woodin cardinals}
\]
and
$V[G]$
is a generic extension of $(V,\in)$ by 
some forcing such that $V[G]\models \stUB$.

This theory is consistent: by Schindler and Asper\`o's result \cite{ASPSCH(*)}
\[
\ZFC+\maxUB+\MM^{++}+\emph{there are class many Woodin cardinals} 
\]
implies $\stUB$;
$\MM^{++}$ is forcible over a model of $\ZFC+$\emph{there is a supercompact}.

By Thm. \ref{thm:keythmmodcompanHomega2} and Robinson's test,
$T^*_{\lUB,\NS_{\omega_1}}$ is a model complete theory.

Given a $\sigma_{\omega,\NS_{\omega_1}}$-theory $T$ extending 
\[
\ZFC^*_{\lUB,\NS_{\omega_1}}+\maxUB+\emph{there is a supercompact cardinal},
\]
let 
\[
T^*=\bp{\phi: (V[G],\sigma_{\omega,\NS_{\omega_1}}^{V[G]})\models \stUB+\phi^{H_{\omega_2}^{V[G]}},\, 
(V,\sigma_{\omega,\NS_{\omega_1}})\models T}.
\]

We start showing that
$T$ and $T^*$ satisfy the assumptions of Lemma~\ref{fac:proofthm1-2}.
This immediately gives \ref{Thm:mainthm-1A}$\Longleftrightarrow$\ref{Thm:mainthm-1B}
for $T$ and $T^*$.

We must show:
\begin{itemize}
\item
$T^*$ is model complete.
\item
$T^*$ is the model companion of $T$.
\item
For any universal sentence $\theta$,
$T+\theta$ is consistent if and only if so is and $T^*+\theta$.
\end{itemize}

First of all $T^*$ is model complete, since it extends $T^*_{\lUB,\NS_{\omega_1}}$:
if $(V,\in)\models T$ and $G$ is such that $(V[G],\in)\models\MM^{++}$, then
\[
\ZFC^*_{\lUB,\NS_{\omega_1}}+\maxUB+\stUB+\text{there are class many Woodin cardinals}.
\]
holds in $V[G]$ by \cite{ASPSCH(*)}, hence 
$H_{\omega_2}^{V[G]}\models T^*_{\lUB,\NS_{\omega_1}}$.

We now show that $T^*_\forall=T_\forall$, i.e. that  $T^*$ is the model companion of $T$.
 
Fix a universal $\sigma_{\omega,\NS_{\omega_1}}$-sentence $\theta$.

Assume $T\vdash\theta$. Fix $V$ a model of $T$.
Let $G$ be $V$-generic for some forcing such that $V[G]\models\stUB$.
By Thm. \ref{thm:PI1invomega2} $V[G]\models\theta$, and by Levy absoluteness 
$H_{\omega_2}^{V[G]}\models\theta$.
Since this argument can be repeated for all models $V$ of $T$, 
we get that $\theta\in T^*$ (by definition of $T^*$).

The converse implication holds by a similar argument which appeals with the obvious variations to Levy absoluteness and to Thm. \ref{thm:PI1invomega2} (i.e. we go backward from $H_{\omega_2}^{V[G]}$ to $V$ for any model
$V$ of $T$ and any forcing extension $V[G]$ of $V$ which models $\stUB$).

Again with the same recipe described above we can prove that 
for any universal sentence $\theta$,
$T+\theta$ is consistent if and only if so is and $T^*+\theta$. We leave the details to the reader.

We are left with the proof of the remaining equivalence between \ref{Thm:mainthm-1A},
\ref{Thm:mainthm-1E}, \ref{Thm:mainthm-1Cbis}, \ref{Thm:mainthm-1C}, 
\ref{Thm:mainthm-1D}, \ref{Thm:mainthm-1Dbis}, \ref{Thm:mainthm-1B}.

\begin{description}

\item[\ref{Thm:mainthm-1A}$\Longrightarrow$\ref{Thm:mainthm-1E}]
By definition of $T^*$.

\item[\ref{Thm:mainthm-1E}$\Longrightarrow$\ref{Thm:mainthm-1Cbis}]
Given a $\sigma_{\omega,\NS_{\omega_1}}$-model 
$(V,\sigma_{\omega,\NS_{\omega_1}}^V)$ of $T$,
by the results of \cite{FORMAGSHE}, we can find a stationary set preserving 
forcing extension $V[G]$ of $V$ which models $\MM^{++}$.
By the key result of Asper\'o and Schindler \cite{ASPSCH(*)} 
$V[G]\models\stUB$. 
By 
\ref{Thm:mainthm-1E} $(V[G],\sigma_{\omega,\NS_{\omega_1}}^{V[G]})$ models $\psi^{H_{\omega_2}^{V[G]}}$,
and we are done.

\item[\ref{Thm:mainthm-1Cbis}$\Longrightarrow$\ref{Thm:mainthm-1C}]
Trivial.

\item[\ref{Thm:mainthm-1C}$\Longrightarrow$\ref{Thm:mainthm-1D}]
By\footnote{$\maxUB$ implies that the same assumption used in the cited theorem for $L(\mathbb{R})$ holds 
for $L(\UB)$.}
 \cite[Thm. 7.3]{HSTLARSON}, if some $P$ forces $\psi^{\dot{H}_{\omega_2}}$,
 we get that $L(\UB)\models \Pmax\Vdash \psi^{\dot{H}_{\omega_2}}$.

\item[\ref{Thm:mainthm-1D}$\Longleftrightarrow$\ref{Thm:mainthm-1Dbis}]
By \cite[Thm. 2.7, Thm. 2.8]{SCHASPBMM*++}.

\item[\ref{Thm:mainthm-1D}$\Longrightarrow$\ref{Thm:mainthm-1B}]
Given some complete $S\supseteq T$, and a model $\mathcal{M}$ of $S$, 
find $\mathcal{N}$ forcing extension of $\mathcal{M}$ which models $\psi^{H_{\omega_2}^{\mathcal{N}}}$.
By Thm. \ref{thm:PI1invomega2} and Levy's absoluteness Lemma \ref{lem:levabsgen}, 
$H_{\omega_2}^{\mathcal{N}}\models \psi+S_\forall$, and we are done.
\end{description}
\end{proof}

\subsection{Proof of Thm.~\ref{thm:keythmmodcompanHomega2}}
The rest of this section is devoted to the proof of Thm.~\ref{thm:keythmmodcompanHomega2}.

What we will do first is to sketch a different proof of Thm.~\ref{thm:modcompanHomega1}.
This will give us the key intuition on how to define $\bar{D}_{\UB,\NS_{\omega_1}}$.

\begin{notation}
From now on given a family of universally Baire sets $\mathcal{A}$,
we let $\tau_{\mathcal{A}}=\tau_\ST\cup\mathcal{A}$ in which allsymbols in $\mathcal{A}$
are interpreted as predicate symbols of the appropriate arity.
\end{notation}

\subsubsection{A different proof of Thm.~\ref{thm:modcompanHomega1}.}

Let $M$ be a countable transitive model of 
$\ZFC+$\emph{there are class many Woodin cardinals}.
Then it will have its own version of Thm.~\ref{thm:modcompanHomega1}.
In particular it will model that the theory of $(H_{\omega_1}^M,\tau_\ST^M,\UB^M)$ 
is model complete,
and also that $\UB^M$ is an $H_{\omega_1}$-closed\footnote{Recall Def. \ref{def:Homega1closed}.} 
family of universally Baire sets in $M$.

Now assume that there is a countable family $\UB_M$ of universally Baire sets in $V$
which is $H_{\omega_1}$-closed in $V$ and is such that
$\UB^M=\bp{B\cap M: B\in\UB_M}$.
Then 
\[
(H_{\omega_1}^M,\tau_\ST^M,\UB^M)=(H_{\omega_1}^M,\tau_\ST^M,\bp{B\cap M: B\in\UB_M})\sqsubseteq
(H_{\omega_1}^V,\tau_{\UB_M}^V)
\]
But $\UB_M$ being $H_{\omega_1}$-closed in $V$ entails that  the first order theory of
$(H_{\omega_1}^V,\tau_{\UB_M}^V)$ is model complete.
In particular if $(H_{\omega_1}^M,\tau_{\UB_M}^M)$ and $(H_{\omega_1}^V,\tau_{\UB_M}^V)$
are elementarily equivalent, then 
\[
(H_{\omega_1}^M,\tau_{\UB_M}^M)\prec
(H_{\omega_1}^V,\tau_{\UB_M}^V).
\]
The setup described above is quite easy to realize (for example $M$ could the transitive collapse of
some countable $X\prec V_\theta$ for some large enough $\theta$); in particular for any
$a\in H_{\omega_1}$ and $B_1,\dots,B_k\in \UB$, we can find $M$ countable transitive model of a suitable fragment of $\ZFC$ with $a\in H_{\omega_1}^M$
and $\UB_M\supseteq \bp{B_1,\dots,B_k}$ countable and $H_{\omega_1}$-closed family of $\UB$-sets
in $V$, such that:
\begin{itemize}
\item $\UB^M=\bp{B\cap M: B\in\UB_M}$;
\item the first order theory $T_{\UB_M}$ of $(H_{\omega_1}^V,\tau_{\UB_M}^V)$ is model complete;
\item $(H_{\omega_1}^M,\tau_\ST^M,\bp{B\cap M: B\in\UB_M})$ models $T_{\UB_M}$.
\end{itemize}
Letting $B_M=\prod\UB_M$, $(H_{\omega_1}\cup\UB,\in)$ is able to compute correctly whether 
$B_M$ encodes a set $\UB_M$ such that the pair $(\UB_M,M)$ satisfies the above list of requirements;
here we use crucially the fact that being a model complete theory is a $\Delta_0$-property, and also
that it is possible to encode
the structure
$(H_{\omega_1}^V,\tau_{\UB_M}^V)$ in a single universally Baire set\footnote{See Def. \ref{def:codkappa} for the definition of $\WFE_\omega$ and $\Cod_\omega$.} 
(for example $\WFE_\omega\times B_M$).

In particular $(H_{\omega_1}\cup\UB,\in)$ correctly computes 
the set $D_{\UB}$ of $M\in H_{\omega_1}$ such that there exists a universally Baire set
$B_M=\prod\UB_M$ with the property that the pair $(M,\UB_M)$ realizes the above set of requirements.
By $\maxUB$, $\bar{D}_{\UB}=\Cod_\omega^{-1}[D_{\UB}]$ is a universally Baire set $\bar{D}_{\UB}$.

Note moreover that $\bar{D}_{\UB}$ is defined by a $\in$-formula $\phi_\UB(x)$ 
in no extra parameters;
in particular for any model $\mathcal{W}=(W,E)$ of $\ZFC+\maxUB$, we can define 
$\bar{D}_{\UB}$ in
$\mathcal{W}$ and all its properties outlined above will hold relativized to $\mathcal{W}$.

For fixed universally Baire sets $B_1,\dots,B_k$ the set $D_{\UB,B_1,\dots,B_k}$ of 
$M\in D_\UB$ such that there is a witness $\UB_M$ of $M\in D_{\UB}$ with $B_1,\dots,B_k\in \UB_M$ is 
also definable in 
\[
(H_{\omega_1}\cup\UB,\in)
\] 
in parameters $B_1,\dots,B_k$.
Hence by $\maxUB$ $\Cod_\omega^{-1}[D_{\UB,B_1,\dots,B_k}]=\bar{D}_{\UB,B_1,\dots,B_k}$ is universally Baire (note as well that $\bar{D}_{\UB,B_1,\dots,B_k}$ belongs to any 
$L(\UB)$-closed family $\mathcal{A}$ containing $B_1,\dots,B_k$).

Now take any $\Sigma_1$-formula $\phi(\vec{x})$ for $\tau_{\UB}$ mentioning just the universally Baire
predicates $B_1,\dots,B_k$.
It doesn't take long to realize that for all $\vec{a}$ in $H_{\omega_1}$
\[
(H_{\omega_1}^V,\tau_{\UB}^V)\models\phi(\vec{a})
\]
if and only if
\[
(H_{\omega_1}^M,\tau_{\UB_M}^M)\models\phi(\vec{a})
\text{\emph{ for all $M\in D_{\UB,B_1,\dots,B_k}$ with 
$\vec{a}\in H_{\omega_1}^M$}. }
\]

But $\bar{D}_{\UB,B_1,\dots,B_k}$ is universally Baire, so the above can be formulated also as:
\[
\forall r\in\bar{D}_{\UB,B_1,\dots,B_k}[\vec{a}\in H_{\omega_1}^{\Cod(r)}\rightarrow 
(H_{\omega_1}^{\Cod(r)},\tau_{\UB_{\Cod(r)}}^{\Cod(r)})\models\phi(\vec{a})].
\]
The latter is a $\Pi_1$-sentence in the universally Baire parameter $\bar{D}_{\UB,B_1,\dots,B_k}$.

This is exactly a proof that Robinson's test applies to the $\tau_{\UB^V}$-first order theory of $H_{\omega_1}^V$ assuming
$\maxUB$; i.e. we have briefly sketched a different (and much more convoluted) 
proof of the conclusion of 
Thm.~\ref{thm:modcompanHomega1} (using as hypothesis Thm.~\ref{thm:modcompanHomega1} itself).
What we gained however is an insight on how to prove Theorem~\ref{thm:keythmmodcompanHomega2}.

We will consider the set $D_{\NS_{\omega_1},\UB}$ 
of $M\in D_{\UB}$ such that:
\begin{itemize} 
\item
$(M,\NS_{\omega_1}^M)$ is a $\Pmax$-precondition 
which is $B$-iterable for all $B\in \UB_M$ (according to \cite[Def. 4.1]{HSTLARSON});
\item
$j_{0\omega_1}$ is a $\Sigma_1$-elementary embedding of $H_{\omega_2}^M$ into $H_{\omega_2}^V$
for $\tau_{\UB_M,\NS_{\omega_1}}$ whenever
$\mathcal{J}=\bp{j_{\alpha\beta}:\alpha\leq\beta\leq\omega_1}$ is an iteration of $M$ with 
$j_{0\omega_1}(\NS_{\omega_1}^M)=\NS_{\omega_1}^V\cap j_{0\omega_1}(H_{\omega_2}^M)$.
\end{itemize}

It will take a certain effort to prove that  assuming $(*)$-$\UB$:
\begin{itemize}
\item for any $A\in H_{\omega_2}$ and
$B\in\UB$, we can find $M\in D_{\NS_{\omega_1},\UB}$ with $B\in \UB_M$, $a\in H_{\omega_2}^M$,
and an iteration $\mathcal{J}=\bp{j_{\alpha\beta}:\alpha\leq\beta\leq\omega_1}$ of $M$ with 
$j_{0\omega_1}(\NS_{\omega_1})=\NS_{\omega_1}^V\cap j_{0\omega_1}(H_{\omega_2}^M)$  such that
$j_{0\omega_1}(a)=A$.
\item
$D_{\NS_{\omega_1},\UB}$ is correctly computable in $(H_{\omega_1}\cup\UB,\in)$.
\end{itemize}
But this effort will pay off since we will then be able to prove the model completeness of the theory
\[
(H_{\omega_2},\tau_{\NS_{\omega_1}}^V\cup\UB^V)
\]
using Robinson's test with $\Cod_\omega^{-1}[D_{\NS_{\omega_1},\UB}]$ in the place of $\bar{D}_{\UB}$ 
and replicating  in the new setting
what was sketched before for $(H_{\omega_1},\tau_{\UB^V}^V)$.

We now get into the details.

\subsubsection{$\UB$-correct models} 
 
\begin{notation}
Given a countable family 
$\mathcal{A}=\bp{B_n:n\in\omega}$ of universally Baire sets with each $B_n\subseteq (2^{\omega})^{k_n}$,
we say that $B_\mathcal{A}=\prod_{n\in\omega}B_n\subseteq \prod_n(2^{\omega})^{k_n}$ is a code for 
$\bp{B_n:n\in\omega}$.

Clearly $B_\mathcal{A}$ is a universally Baire subset of the Polish space $\prod_n(2^{\omega})^{k_n}$.
\end{notation}

\begin{definition}
$T_\mathsf{UB}$ is the $\in$-theory of 
\[
(H_{\omega_1},\tau_{\UB}).
\]

A transitive model of $\ZFC$ $(M,\in)$ is $\mathsf{UB}$-correct if 
there is an  $H_{\omega_1}$-closed (in $V$) family
$\mathsf{UB}_M$ of universally Baire sets in $V$
such that:
\begin{itemize}
\item The map 
\begin{align*}
\Theta_M:&\UB_M\to M\\
&A\mapsto A\cap M
\end{align*}
is injective.
\item
$(M,\in)$ models that $\bp{A\cap M: A\in \mathsf{UB}_M}$ is the family of universally Baire subsets of $M$.

\item Letting $T_{\UB_M}$ be the theory of $(H_{\omega_1}, \tau_{\mathsf{ST}}^V,\mathsf{UB}_M)$
\[
(H_{\omega_1}^M, \tau_{\mathsf{ST}}^M,A\cap M: A\in \mathsf{UB}_M)\models T_{\UB_M}.
\]
\item If $M$ is countable,
$M$ is $A$-iterable for all $A\in \mathsf{UB}_M$.
\end{itemize}
\end{definition}

Remark (by Thm.~\ref{fac:keyfacHomega1clos}) that if $M$ is $\UB$-correct, $T_{\UB_M}$ is model complete, 
since $\UB_M$ is (in $V$) a $H_{\omega_1}$-closed family of 
universally Baire sets.

\begin{notation}
$D_{\UB}$ denotes the set of countable $\mathsf{UB}$-correct $M$; $\bar{D}_\UB=\Cod_\omega^{-1}[D_\UB]$.

For each $M$ $\UB_M$ is a witness that $M\in D_\UB$ and $B_{\UB_M}=\prod\UB_M$ is a universally Baire coding this 
witness\footnote{The Fact below shows that the map $M\mapsto (\UB_M,B_{\UB_M})$, can be chosen in $L(\UB)$.}.

For universally Baire sets $B_1,\dots,B_k$, $E_{\UB,B_1,\dots,B_k}$
denotes the set of $M\in D_\UB$ with $B_1,\dots,B_k\in \UB_M$ for some witness $\UB_M$ that $M\in D_\UB$;
$\bar{E}_{\UB,B_1,\dots,B_k}=\Cod_\omega^{-1}[E_{\UB,B_1,\dots,B_k}]$.
\end{notation}

\begin{fact}  
$(V,\in)$ models $M$\emph{ is countable and $\UB$-correct as witnessed by $\UB_M$}
if and only if so does 
$(H_{\omega_1}\cup{\UB},\in)$.

Consequently 
the set $D_{\mathsf{UB}}$ of countable $\mathsf{UB}$-correct $M$ is properly computed in
$(H_{\omega_1}\cup\UB,\in)$.

Therefore assuming $\maxUB$ 
\[
\bar{D}_\UB=\mathrm{Cod}^{-1}[D_{\mathsf{UB}}]
\] 
is universally Baire. 

Moreover there is in $L(\UB)$ a definable map $M\mapsto \UB_M$ assigning to each $M\in D_\UB$ 
a countable family $\UB_M$ witnessing it.

The same holds for $\bar{E}_{\UB,B_1,\dots,B_k}$ for given universally Baire sets $B_1,\dots,B_k$.
\end{fact}
\begin{proof}
The first part follows almost immediately by the definitions, since the assertion in parameters $B,M$:
\begin{quote}
\emph{
$B=\prod_{n\in\omega}B_n$ codes a $H_{\omega_1}$-closed family 
$\UB_M=\bp{B_n:n\in\omega}$ of sets such that 
\begin{itemize}
\item $M$ is $A$-iterable for all $A\in \UB_M$,
\item $M$ models that $\bp{A\cap M: A\in \UB_M}$ is its family of universally 
Baire sets and is $H_{\omega_1}$-closed,
\item
$(H_{\omega_1}^M,\tau_{\ST}^M,\bp{A\cap M: A\in \mathsf{UB}_M})$ models
$T_{\UB_M}$.
\end{itemize}
}
\end{quote}
gets the same truth value in $(V,\in)$ and in $(H_{\omega_1}\cup \UB,\in)$.

We conclude that $D_{\mathsf{UB}}$ has the same extension in 
$(V,\in)$ and in $(H_{\omega_1}\cup\UB,\in)$.
By $\maxUB$ $\bar{D}_\UB$ is universally Baire.

The existence of class many Woodin cardinals grants
that we can always find\footnote{For example by \cite[Thm. 36.9]{kechris:descriptive} and \cite[Thm. 3.3.14, Thm. 3.3.19]{STATLARSON}.} 
a universally Baire uniformization of the 
universally Baire relation 
on $\bar{D}_\UB\times 2^\omega$ given by the pairs $\ap{r,B}$ such that $B=\prod\bp{B_n:n\in\omega}$ witnesses
$\Cod_\omega(r)\in D_\UB$ .

The same argument can be replicated for  $\bar{E}_{\UB,B_1,\dots,B_k}$.
\end{proof}

\begin{lemma}\label{lem:UBcorr}
Assume $\NS_{\omega_1}$ is precipitous and there are class many 
Woodin cardinals in $V$.
Let $\delta$ be an inaccessible cardinal in $V$ and $G$ be 
$V$-generic for $\Coll(\omega,\delta)$.
Then $V_\delta$ is $\UB^{V[G]}$-correct in $V[G]$ as witnessed by $\bp{B^{V[G]}:B\in \UB^V}$.
\end{lemma}
\begin{proof}
Let in $V$ 
$\bp{(T_A,S_A): A\in\UB^V}$ be an enumeration of pairs of trees $S_A,U_A$ on $\omega\times\gamma$ for 
a large enough inaccessible $\gamma>\delta$
such that $T_A,S_A$ projects to complements in $V[G]$ and $A$ is the projection of $T$.
Then $A^{V[G]}$ is correctly computed as the projection of $T_A$ in $V[G]$ for any $A\in\UB^V$.

By Thm.~\ref{fac:keyfacHomega1clos}
\[
(H_{\omega_1}^V, \tau_{\mathsf{ST}}^V,\mathsf{UB}^V)\prec (H_{\omega_1}^{V[G]}, \tau_{\mathsf{ST}}^{V[G]},A^{V[G]}: 
A\in\mathsf{UB}^{V}),
\]
$\bp{A^{V[G]}: A\in\mathsf{UB}^{V}}$ is a $H_{\omega_1}$-closed family of 
universally Baire sets in $V[G]$, and 
$T_{\UB^V}$ is also the theory of $(H_{\omega_1}^{V[G]}, \tau_{\mathsf{ST}}^{V[G]},A^{V[G]}: 
A\in\mathsf{UB}^{V})$.

To conclude that $\bp{A^{V[G]}: A\in\mathsf{UB}^{V}}$ witnesses in $V[G]$ that $V_\delta$ is $\UB^{V[G]}$-correct in 
$V[G]$ it remains to argue that 
$V_\delta$ is $B^{V[G]}$-iterable for any $B\in\UB^V$.

Let $\mathcal{J}$ be any iteration of $V_\delta$ in $V[G]$.
Then  by standard results on iterations
(see \cite[Lemma 1.5, Lemma 1.6]{HSTLARSON})
$\mathcal{J}$ extends uniquely to an iteration $\bar{\mathcal{J}}$ of $V$ in $V[G]$ such that
\begin{itemize}
\item
$\bar{j}_{\alpha\beta}$ is a proper extension of $j_{\alpha\beta}$ for all $\alpha\leq\beta\leq\gamma$
(i.e. letting $\bar{V}_\alpha=\bar{j}_{0\alpha}(V)$, we have that
$j_{0\alpha}(V_\delta)$ is the rank initial segments of elements of $\bar{V}_\alpha$ of rank less than
$\bar{j}_{0\alpha}(\delta)$).
\item
$\bar{\mathcal{J}}$ is a well defined iteration of transitive structures. 
\end{itemize}
In particular this shows that $V_\delta$ is iterable in $V[G]$.

Now fix $B\in \UB^V$. We must argue that $j_{0\alpha}(B)=B^{V[G]}\cap \bar{j}_{0\alpha}(V)$.
To simplfy notation we assume $B\subseteq 2^\omega$.
Let $(T_B,S_B)$ be the pair of trees selected in $V$ to define $B^{V[G]}$.

Then  
\[
\bar{j}_{0\alpha}(V)\models (\bar{j}_{0\alpha}(T_B), \bar{j}_{0\alpha}(S_B))
\]
projects to complements; clearly $\bar{j}_{0\alpha}[T_B]\subseteq \bar{j}_{0\alpha}(T_B)$,
$\bar{j}_{0\alpha}[S_B]\subseteq \bar{j}_{0\alpha}(S_B)$.
Let $p:(\gamma\times 2)^{\omega}\to 2^\omega$ be the projection map. 

This gives that
\[
B^{V[G]}\cap \bar{j}_{0\alpha}(V)=p[[T_B]]\cap \bar{j}_{0\alpha}(V)=p[[\bar{j}_{0\alpha}[T_B]]]\cap \bar{j}_{0\alpha}(V)\subseteq
p[[\bar{j}_{0\alpha}(T_B)]]\cap \bar{j}_{0\alpha}(V)=\bar{j}_{0\alpha}(B).
\]
Similarly
\[
((2^{\omega})^{V[G]}\setminus B^{V[G]})\cap \bar{j}_{0\alpha}(V)=p[[S_B]]\cap \bar{j}_{0\alpha}(V)\subseteq
p[[\bar{j}_{0\alpha}(S_B)]]\cap \bar{j}_{0\alpha}(V)=\bar{j}_{0\alpha}((2^\omega)^V\setminus B).
\]
By elementarity
\[
\bar{j}_{0\alpha}((2^\omega)^V\setminus B)\cup \bar{j}_{0\alpha}(B)=(2^\omega)\cap \bar{j}_{0\alpha}(V).
\]
These three conditions can be met only if
\[
B^{V[G]}\cap \bar{j}_{0\alpha}(V)=\bar{j}_{0\alpha}(B).
\]

Since $\mathcal{J}$ and $B$ were chosen arbitrarily,
we conclude that $V_\delta$ is $B^{V[G]}$-iterable in $V[G]$ for all $B\in\UB^V$.

Hence $V_\delta$ is $\UB^{V[G]}$-correct in $V[G]$ as witnessed by $\bp{A^{V[G]}: A\in \UB^V}$.
\end{proof}

\begin{definition}
Given $M,N$ iterable structures, 
$M\geq N$ if $M\in (H_{\omega_1})^N$ and there is  an iteration
\[
\mathcal{J}=\bp{j_{\alpha\beta}:\,\alpha\leq\beta\leq\gamma=(\omega_1)^N}
\]
of $M$ with $\mathcal{J}\in N$
such that
\[
\NS_{\gamma}^{M_\gamma}=
\NS_{\gamma}^N\cap M_{\gamma}.
\]
\end{definition}

\begin{fact}\label{fac:densityUBcorrect}
$(\maxUB)$
Assume $\NS_{\omega_1}$ is precipitous and $\maxUB$ holds.
Then for any iterable $M$ and $B_1,\dots,B_k\in \UB$, there is an $\UB$-correct $N\geq M$ with $B_1,\dots,B_k\in \UB_N$.
\end{fact}
\begin{proof}
The assumptions grant that whenever $G$ is $\Coll(\omega,\delta)$-generic for $V$,
in $V[G]$ $V_\delta$ is $\UB^{V[G]}$-correct in $V[G]$ (i.e. Lemma \ref{lem:UBcorr}).

By \cite[Lemma 2.8]{HSTLARSON}, for any iterable $M\in H_{\omega_1}^V$ there is in $V$ an iteration
 $\mathcal{J}=\bp{j_{\alpha\beta}:\alpha\leq\beta\leq \omega_1^V}$ of $M$ such that
 $\NS_{\omega_1}^V\cap M_{\omega_1}=\NS_{\omega_1}^{M_{\omega_1}}$.
 
By $\maxUB$
\[
(H_{\omega_1}^{V}\cup\UB^{V},\in) \prec 
(H_{\omega_1}^{V[G]}\cup\UB^{V[G]},\in).
\]
 Therefore  we have that in $V[G]$ $\bar{E}_{\UB,B_1,\dots,B_k}^{V[G]}$ is exactly 
 $\bar{E}_{\UB,B_1^{V[G]},\dots,B_k^{V[G]}}$.
 
Hence for each iterable $M\in H_{\omega_1}^V$ and $B\in \UB^V$
 \[
 (H_{\omega_1}^{V[G]},\tau_{\UB^V}^{V[G]})\models\exists \,N\geq M\text{ $\UB^{V[G]}$-correct with $B^{V[G]}$ in $\UB_N$},
 \]
 as witnessed by $N=V_\delta$, i.e. 
 \[
 (H_{\omega_1}^{V[G]},\tau_{\UB^V}^{V[G]})\models\exists \,N\geq M \,(\bar{E}_{\UB,B_1,\dots,B_k}^{V[G]}(N)).
 \]

Since
  \[
 (H_{\omega_1}^{V},\tau_{\UB^V}^{V})\prec (H_{\omega_1}^{V[G]},\tau_{\UB^V}^{V[G]}),
 \]
we get that for every iterable $M\in H_{\omega_1}$ and $B\in \UB^V$
 \[
 (H_{\omega_1}^{V},\tau_{\UB^V}^{V})\models\exists \,N\geq M\,(\bar{E}_{\UB,B_1,\dots,B_k}(N)).
 \]
 The conclusion follows.
\end{proof}

\begin{lemma}
$(\maxUB)$

Let $M\geq N$ be both $\UB$-correct structures, with
$\UB_N$ a witness of $N$ being $\UB$-correct  such that $\bar{D}_{\mathsf{UB}}\in\UB_N$.
Then
\[
(H_{\omega_1}^{M},\tau_{\mathsf{ST}}^M,A\cap M: A\in\UB_M)\prec
(H_{\omega_1}^{N},\tau_{\mathsf{ST}}^N,A\cap N: A\in\UB_M).
\]
\end{lemma}
\begin{proof}
Since $N\leq M$, and $N$ is $\UB$-correct with $\bar{D}_{\mathsf{UB}}\in\UB_N$
we get that 
\[
(H_{\omega_1}^N,\tau_{\UB_N}^N)\models M\in D_{\mathsf{UB}}\cap N=\Cod[\bar{D}_{\mathsf{UB}}\cap N],
\]
since
\[
(H_{\omega_1}^N,\tau_{\UB_N}^N)\prec (H_{\omega_1}^V,\tau_{\UB_N}^V)
\]
and
\[
(H_{\omega_1}^V,\tau_{\UB_N}^V)\models M\in D_{\mathsf{UB}}=\Cod[\bar{D}_{\mathsf{UB}}].
\]

Therefore $N$ models that there is a countable set $\UB_M^N=\bp{B_n^N:n\in\omega}\in N$ 
coded by the universally Baire set in $N$ $B_{\UB_M}^N=\prod_{n\in\omega}B_n^N$ such that 
$\bp{A\cap M:A\in \UB_M^N}\in M$ defines the family of universally Baire sets according to $M$, and such that $N$ models that
$M$ is $B^N$ iterable for all $B^N\in\UB_M^N$.
Now $N$ models that 
\[
\prod_{n\in\omega}B_n^N
\]
is a universally Baire set on the appropriate product space. 
Therefore there is $B\in \UB_N$ such that
$B\cap N=\prod_{n\in\omega}B_n^N$.
Clearly $\UB_M^N$ is computable from $B\cap N$.
Since
\[
(H_{\omega_1}^N,\tau_{\UB_N}^N)\prec (H_{\omega_1}^V,\tau_{\UB_N}^V).
\]
we conclude that in $V$ $B=\prod_{n\in\omega}B_n$ codes a set $\UB_M=\bp{B_n:n\in\omega}$ 
witnessing that $M$ is $\UB$-correct.

This gives that $\UB_M\subseteq\UB_N$.

Therefore 
$(H_{\omega_1}^N,\tau_{\UB_M}^N)$
is also a model of
$T_{\UB_M}$.
By model completeness of $T_{\UB_M}$ we conclude that
\[
(H_{\omega_1}^M,\tau_{\UB_M}^M)\prec (H_{\omega_1}^N,\tau_{\UB_M}^N),
\]
as was to be shown.
\end{proof}

\subsection{Three characterizations of $(*)$-$\UB$}

Recall that
for a family $\mathcal{A}$ of universally Baire sets
$\tau_{\mathcal{A},\NS_{\omega_1}}=\tau_{\omega_1}\cup\mathcal{A}$.

\begin{definition}
For a $\UB$-correct $M$ with witness $\UB_M$, 
$T_{\NS_{\omega_1},\UB_M}$ is the 
$\tau_{\UB_M,\NS_{\omega_1}}$-theory of $H_{\omega_2}^M$.

\smallskip

A $\UB$-correct $M$ is \emph{$(\NS_{\omega_1},\UB)$-ec}
if $(M,\in)$ models that $\NS_{\omega_1}$ is precipitous and 
there is a witness $\UB_M$ that $M$ is $\UB$-correct with the following property:   

\begin{quote}
Assume an iterable $N\geq M$ is $\UB$-correct with witness $\UB_N$ such that 
$B_{\UB_M}\in \UB_N$ (so that $\UB_M\subseteq \UB_N$).

Then for all iterations
\[
\mathcal{J}=\bp{j_{\alpha\beta}:\alpha\leq\beta\leq\gamma=\omega_1^N}
\] 
in $N$ witnessing
$M\geq N$, we have that $j_{0\gamma}$ defines a $\Sigma_1$-elementary embedding of
\[
(H_{\omega_2}^{M},\tau_{\mathsf{ST}}^{M},B\cap M: B\in\UB_M,\NS_{\omega_1}^{M})
\]
into 
\[
(H_{\omega_2}^N,\tau_{\mathsf{ST}}^N,B\cap N: B\in\UB_M,\NS_{\omega_1}^N).
\]
\end{quote}

\end{definition}

\begin{remark}\label{rmk:maxUBec}
A crucial observation is that
\emph{``$x$ is $(\NS_{\omega_1},\UB)$-ec''}
is a property correctly definable in $(H_{\omega_1}\cup\UB,\in)$.
Therefore (assuming $\maxUB$)
\[
D_{\NS_{\omega_1},\UB}=\bp{M\in H_{\omega_1}: \, M\text{ is $(\NS_{\omega_1},\UB)$-ec}}
\]
is such that $\bar{D}_{\NS_{\omega_1},\UB}=\Cod_\omega^{-1}[D_{\NS_{\omega_1},\UB}]$ 
is a universally Baire set in $V$.
Moreover letting for $V[G]$ a generic extension of $V$
\[
D_{\NS_{\omega_1},\UB^{V[G]}}=\bp{M\in H_{\omega_1}^{V[G]}: \, M\text{ is $(\NS_{\omega_1},\UB^{V[G]})$-ec}},
\]
we have that 
\[
\bar{D}_{\NS_{\omega_1},\UB}^{V[G]}=\Cod_\omega^{-1}[D_{\NS_{\omega_1},\UB^{V[G]}}].
\]
\end{remark}

\begin{theorem}\label{thm:char(*)}
Assume $V$ models $\maxUB$.
The following are equivalent:
\begin{enumerate}
\item\label{thm:char(*)-1}
Woodin's axiom $(*)$-$\mathsf{UB}$ holds
(i.e. 
$\NS_{\omega_1}$ is saturated,
and there is an $L(\mathsf{UB})$-generic filter $G$ for $\mathbb{P}_{\mathrm{max}}$
such that $L(\mathsf{UB})[G]\supseteq\pow{\omega_1}^V$).
\item\label{thm:char(*)-2}
Let $\delta$ be inaccessible.
Whenever $G$ is $V$-generic for $\Coll(\omega,\delta)$,
$V_\delta$ is $(\NS_{\omega_1},\UB^{V[G]})$-ec in $V[G]$.

\item \label{thm:char(*)-3}
$\NS_{\omega_1}$ is precipitous and
for all $\vec{A}\in H_{\omega_2}$, $B\in\UB$, there is an $(\NS_{\omega_1},\UB)$-ec $M$
with witness $\UB_M$,
and an iteration $\mathcal{J}=\bp{j_{\alpha\beta}:\,\alpha\leq\beta\leq\omega_1}$ of $M$ 
such that:
\begin{itemize}
\item $A\in M_{\omega_1}$,
\item $B\in\UB_M$,
\item $\NS_{\omega_1}^{M_{\omega_1}}=\NS_{\omega_1}\cap M_{\omega_1}$.
\end{itemize}
\end{enumerate}
\end{theorem}

Theorem \ref{thm:char(*)} is the key to the proofs of Theorem~\ref{thm:keythmmodcompanHomega2}
and to the missing implication in the proof of Theorem~\ref{Thm:mainthm-1bis}.

\subsubsection{Proof of Theorem~\ref{thm:keythmmodcompanHomega2}}

The theorem is an immediate corollary of the following:

\begin{lemma}\label{lem:keylemmodcomp(*)}

Let $B_1,\dots,B_k$ be new predicate symbols and 
$T_{B_1,\dots,B_k,\NS_{\omega_1}}$ 
be the $\tau_{\NS_{\omega_1}}\cup\bp{B_1,\dots,B_k}$-theory 
$\ZFC^*_{\NS_{\omega_1}}+\maxUB$ enriched with
the sentences asserting that $B_1,\dots,B_k$ are universally Baire sets.

Let $E_{B_1,\dots,B_k}$ consists of the set of
$M\in D_{\NS_{\omega_1},\UB}$ such that:
\begin{itemize}
\item
$M$ is $B_j$-iterable for all $j=1,\dots,k$;
\item 
there is $\UB_M$ witnessing $M\in D_{\NS_{\omega_1},\UB}$ with 
$B_j\in\UB_M$ for all $j$.
\end{itemize}
Let also $\bar{E}_{B_1,\dots,B_k}=\Cod_\omega^{-1}[E_{B_1,\dots,B_k}]$. 

Then $T_{B_1,\dots,B_k,\NS_{\omega_1}}$ proves that
$\bar{E}_{B_1,\dots,B_k}$ is universally Baire. 

Moreover let $T_{B_1,\dots,B_k,\bar{E}_{B_1,\dots,B_k},\NS_{\omega_1}}$ be the  natural extension of  
$T_{B_1,\dots,B_k,\NS_{\omega_1}}$
adding a predicate symbol for $\bar{E}_{B_1,\dots,B_k}$ 
and the axiom forcing its intepretation to be its definition.

Then $T_{B_1,\dots,B_k,\bar{E}_{B_1,\dots,B_k},\NS_{\omega_1}}$
models that every $\Sigma_1$-formula $\phi(\vec{x})$ for the signature
$\tau_{\NS_{\omega_1}}\cup\bp{B_1,\dots,B_k}$ is equivalent
to a $\Pi_1$-formula $\psi(\vec{x})$ in the signature
$\tau_{\NS_{\omega_1}}\cup\bp{B_1,\dots,B_k, \bar{E}_{B_1,\dots,B_k}}$.
\end{lemma}

\begin{proof}
$\bar{E}_{B_1,\dots,B_k}$ is universally Baire by $\maxUB$,
since $E_{B_1,\dots,B_k}$
is definable in $(H_{\omega_1}\cup\mathsf{UB},\in)$ with parameters the universally Baire sets
$B_1,\dots,B_k,\bar{D}_{\NS_{\omega_1},\UB}$.

Given any $\Sigma_1$-formula $\phi(\vec{x})$ for $\tau_{\NS_{\omega_1}}\cup\bp{B_1,\dots,B_k}$
mentioning the universally Baire predicates $B_1,\dots,B_k$, we want
to find a universal formula $\psi(\vec{x})$ such that
\[
T_{\bp{B_1,\dots,B_k, \bar{E}_{B_1,\dots,B_k}},\NS_{\omega_1}}\models \forall\vec{x}(\phi(\vec{x})\leftrightarrow \psi(\vec{x})).
\]

Let $\psi(\vec{x})$ be the formula
asserting:
 
\begin{quote}
\emph{For all $M\in E_{B_1,\dots,B_k}$, for all iterations 
$\mathcal{J}=\bp{j_\alpha\beta:\alpha\leq\beta\leq\omega_1}$ of $M$ such that:}
\begin{itemize}
\item
$\vec{x}=j_{0\omega_1}(\vec{a})$\emph{ for some }$\vec{a}\in M$,
\item
$\NS_{\omega_1}^{j_{0\omega_1}(M)}=\NS_{\omega_1}\cap j_{0\omega_1}(M)$,
\end{itemize}
\[
(H_{\omega_2}^{M},\tau_{\UB_{M},\NS_{\omega_1}}^{M})\models\phi(\vec{a}).
\]
\end{quote}

More formally:
\begin{align*}
\forall r\, \forall \mathcal{J}&\{&\\
&[&\\
&(r\in \bar{E}_{B_1,\dots,B_k})\wedge&\\
&\wedge\mathcal{J}=\bp{j_\alpha\beta:\alpha\leq\beta\leq\omega_1} \text{ is an iteration of }\Cod(r)\wedge&\\
&\wedge\NS_{\omega_1}^{j_{0\omega_1}(\Cod(r))}=\NS_{\omega_1}\cap j_{0\omega_1}(\Cod(r))\wedge&\\
&\wedge \exists\vec{a}\in\Cod(r)\,(\vec{x}=j_{0\omega_1}(\vec{a}))&\\
&]&\\
&\rightarrow&\\
&(H_{\omega_2}^{\Cod(r)},\tau_{\UB_{\Cod(r)},\NS_{\omega_1}}^{\Cod(r)})\models\phi(\vec{a})&\\
&\}.&
\end{align*}
The above is a $\Pi_1$-formula  for 
$\tau_{\NS_{\omega_1}}\cup\bp{B_1,\dots,B_k, \bar{E}_{B_1,\dots,B_k}}$.

(We leave to the reader to check that the property 
\begin{quote}
\emph{$\mathcal{J}=\bp{j_\alpha\beta:\alpha\leq\beta\leq\omega_1}$ is an iteration of $M$ such that 
$\NS_{\omega_1}^{j_{0\omega_1}(M)}=\NS_{\omega_1}\cap j_{0\omega_1}(M)$}
\end{quote}
is definable by a $\Delta_1$-property in parameters $M,\mathcal{J}$ in the signature
$\tau_{\NS_{\omega_1}}$).

Now it is not hard to check that:
\begin{claim}
For all $\vec{A}\in H_{\omega_2}$
\[
(H_{\omega_2}^V,\tau_{\NS_{\omega_1}}^V,B_1,\dots,B_k)\models\phi(\vec{A})
\]
if and only if 
\[
(H_{\omega_2},\tau_{\NS_{\omega_1}}^V,B_1,\dots,B_k, \bar{E}_{B_1,\dots,B_k})
\models\psi(\vec{A}).
\]
\end{claim}
\begin{proof}
\emph{}

\begin{description}
\item[$\psi(\vec{A})\rightarrow \phi(\vec{A})$]
Take any
$M$ and $\mathcal{J}$ satisfying the premises of the implication in $\psi(\vec{A})$, 
Then $(H_{\omega_2}^M,\tau_{\NS_{\omega_1},\UB^M}^M)\models\phi(\vec{a})$
for some $\vec{a}$ such that $j_{0,\omega_1}(\vec{a})=\vec{A}$ and
$B_j\cap M_{\omega_1}=j_{0\omega_1}(B_j\cap M)$ for all $j=1,\dots,k$.

Since $\Sigma_1$-properties are upward absolute and
$(M_{\omega_1},\tau_{\NS_{\omega_1}}^{M_{\omega_1}},B_j\cap M_{\omega_1}:j=1,\dots,k)$
is a $\tau_{\NS_{\omega_1}}\cup\bp{B_1,\dots,B_k}$-substructure 
of  $(H_{\omega_2},\tau_{\NS_{\omega_1}}^V,B_j:j=1,\dots,k)$ which models $\phi(\vec{A})$, we get that
$\phi(\vec{A})$ holds for $(H_{\omega_2},\tau_{\NS_{\omega_1}}^V,B_1,\dots,B_k)$.

\item[$\phi(\vec{A})\rightarrow \psi(\vec{A})$]
Assume
\[
(H_{\omega_2},\tau_{\NS_{\omega_1}}^V,B_1,\dots,B_k)\models\phi(\vec{A}).
\]
Take any $(\NS_{\omega_1},\UB)$-ec $M\in V$ and any iteration 
$\mathcal{J}=\bp{j_\alpha\beta:\alpha\leq\beta\leq\omega_1}$  of  $M$ witnessing the premises
of the implication in 
$\psi(\vec{A})$, in particular such that:
\begin{itemize}
\item
$\vec{A}=j_{0\omega_1}(\vec{a})\in M_{\omega_1}$ for some $\vec{a}\in M$, 
\item
$\NS_{\omega_1}^{M_{\omega_1}}=\NS_{\omega_1}\cap M_{\omega_1}$,
\item
$M$ is $B_j$-iterable for $j=1,\dots,k$.
\end{itemize}

Such $M$ and $\mathcal{J}$ exists by Thm.~\ref{thm:char(*)}(\ref{thm:char(*)-3}) applied to 
$\bar{E}_{B_1,\dots,B_k}$ and $\vec{A}$.

Let $G$ be $V$-generic for $\Coll(\omega,\delta)$ with $\delta$ inaccessible.
Then in $V[G]$, $V_\delta$ is $\UB^{V[G]}$-correct, by Lemma \ref{lem:UBcorr}.
 
Therefore (since $M$ is $(\NS_{\omega_1},\UB^{V[G]})$-ec also in $V[G]$ by $\maxUB$),
$V[G]$ models that
$j_{0\omega_1^V}$ is a $\Sigma_1$-elementary embedding of
\[
(H_{\omega_2}^{M},\tau_{\NS_{\omega_1}}^{M},B\cap M:B\in\mathsf{UB}_M)
\]
into 
\[
(H_{\omega_2}^V,\tau_{\NS_{\omega_1}}^V,B:B\in\UB_M).
\]
This grants that
\[
(H_{\omega_2}^M,\tau_{\NS_{\omega_1}}^M,B\cap M:B\in\mathsf{UB}_M)\models\phi(\vec{a}),
\]
as was to be shown.
\end{description}
\end{proof}

The Lemma is proved.

\end{proof}

\subsubsection{Proof of (\ref{thm:char(*)-modcomp-2})$\to$(\ref{thm:char(*)-modcomp-1})
of Theorem~\ref{Thm:mainthm-1bis}}

\begin{proof}
Assume $\delta$ is supercompact, $P$ is a standard forcing notion to force $\MM^{++}$ of size $\delta$ (such as the one introduced in 
\cite{FORMAGSHE} to prove the consistency of Martin's maximum), and $G$ is $V$-generic for $P$; then
$(*)$-$\UB$ holds in $V[G]$ by Asper\'o and Schindler's recent breakthrough \cite{ASPSCH(*)}.
By Thm. \ref{thm:PI1invomega2}
$V$ and $V[G]$ agree on the $\Pi_1$-fragment of their $\tau_{\UB^V,\NS_{\omega_1}}$-theory, therefore so do 
$H_{\omega_2}^V$ and $H_{\omega_2}^{V[G]}$ (by Lemma \ref{lem:levabsgen}
applied in $V$ and $V[G]$ respectively).

Since $P\in\SSP$
\[
(H_{\omega_2}^V,\tau_{\NS_{\omega_1}}^V,A:A\in \UB^V)\sqsubseteq
(H_{\omega_2}^{V[G]},\tau_{\NS_{\omega_1}}^{V[G]},A^{V[G]}: A\in\UB^V).
\]

Now the model completeness of $T_{\NS_{\omega_1},\UB}$-grants that any of its models (among which $H_{\omega_2}^V$)
is $(T_{\NS_{\omega_1},\UB})_\forall$-ec. This gives that:
\[
(H_{\omega_2}^V,\tau_{\NS_{\omega_1}}^V,\UB^V)\prec_{\Sigma_1}
(H_{\omega_2}^{V[G]},\tau_{\NS_{\omega_1}}^{V[G]},A^{V[G]}: A\in\UB^V).
\]

Therefore any $\Pi_2$-property for $\tau_{\UB,\NS_{\omega_1}}$ with parameters in 
$H_{\omega_2}^V$ which holds in
\[
(H_{\omega_2}^{V[G]},\tau_{\NS_{\omega_1}}^{V[G]},A^{V[G]}: A\in\UB)
\]
also holds in $(H_{\omega_2}^V,\tau_{\NS_{\omega_1}}^V,\UB^V)$.

Hence in $H_{\omega_2}^V$ it holds characterization (\ref{thm:char(*)-3}) of $(*)$-$\UB$  given by Thm.~\ref{thm:char(*)} and we are done.
\end{proof}

\subsubsection{Proof of Theorem \ref{thm:char(*)}}

\begin{proof}
Schindler and Asper\'o \cite[Def. 2.1]{SCHASPBMM*++} introduced the following:
\begin{definition} \label{def:aspschhoncons}
Let $\phi(\vec{x})$ be a $\tau_{\UB,\NS_{\omega_1}}$-formula in free variables $\vec{x}$,
and $\vec{A}\in H_{\omega_2}^V$.
$\phi(\vec{A})$ is \emph{honestly consistent} if for all universally Baire sets 
$U\in \mathsf{UB}^V$,
there is some large enough cardinal $\kappa\in V$ such that 
whenever $G$ is $V$-generic for $\Coll(\omega,\kappa)$, in $V[G]$ there is 
a $\tau_{\UB,\NS_{\omega_1}}$-structure $\mathcal{M}=(M,\dots)$ such that 
\begin{itemize}
\item
$M$ is transitive and $U^{V[G]}$-iterable in $V[G]$,
\item
$\mathcal{M}\models \phi(\vec{A})$,
\item
$\NS_{\omega_1}^M\cap V=\NS_{\omega_1}^V$.
\end{itemize}
\end{definition}
They also proved the following Theorem \cite[Thm. 2.7, Thm. 2.8]{SCHASPBMM*++}:
\begin{theorem}
Assume $V$ models $\NS_{\omega_1}$ is precipitous and 
$\maxUB$ holds.

TFAE:
\begin{itemize}
\item $(*)$-$\UB$ holds in $V$.
\item Whenever $\phi(\vec{x})$ is a $\Sigma_1$-formula for $\tau_{\UB,\NS_{\omega_1}}$ in free variables $\vec{x}$, and $\vec{A}\in H_{\omega_2}^V$,
$\phi(\vec{A})$ is honestly consistent if and only if it is true in $H_{\omega_2}^V$.
\end{itemize}
\end{theorem}

We use Schindler and Asper\'o characterization of $(*)$-$\UB$ to prove the equivalences
of the three items of Thm. \ref{thm:char(*)}
(the proofs of these implications import key ideas
from \cite[Lemma 3.2]{ASPSCH(*)}).

\begin{description}
\item[(\ref{thm:char(*)-1})
implies (\ref{thm:char(*)-2})]
Let $G$ be $V$-generic for $\Coll(\omega,\delta)$.
By Lemma \ref{lem:UBcorr},
$V_\delta$ is $\UB^{V[G]}$-correct in $V[G]$ as witnessed by 
$\bp{B^{V[G]}:B\in \UB^V}=\UB_V=\bp{B_n^{V[G]}:n\in\omega}$.

\begin{claim}
$V_\delta$ is $(\NS_{\omega_1},\UB^{V[G]})$-ec as witnessed by $\UB_V$.
\end{claim}

\begin{proof}
Let in $V[G]$ $B_V=B_{\UB_V}=\prod_{n\in\omega}B_n^{V[G]}$ be the universally Baire set coding $\UB_V$.

Let $N\leq V_\delta$ in $V[G]$ be $\UB^{V[G]}$-correct with $B_V\in \UB_N$ for some $\UB_N$ witnessing that 
$N$ is $\UB^{V[G]}$-correct.
Then we already observed that $\bp{B^{V[G]}\cap N: B^{V[G]}\in \UB_V}\subseteq \bp{B\cap N: \, B\in \UB_N}$.
Therefore
\[
(H_{\omega_1}^V,\tau_{\UB_V}^V)=(H_{\omega_1}^V,\tau_{\UB^V}^V)
\prec (H_{\omega_1}^N,\tau_{\ST}^N, B^{V[G]}\cap N: B\in \UB^V).
\]

Let 
\[
\mathcal{J}=\bp{j_{\alpha,\beta}:\alpha\leq\beta\leq\gamma=(\omega_1)^N}\in N
\] 
be an iteration witnessing
$V_\delta\geq N$ in $V[G]$.

We must show that 
\[
j_{0\gamma}:H_{\omega_2}^V\to H_{\omega_2}^N
\]
is $\Sigma_1$-elementary for $\tau_{\NS_{\omega_1},\UB^V}$ between
\[
(H_{\omega_2}^V,\tau_{\ST}^V,\UB^V,\NS_{\omega_1}^V)
\] 
and 
\[
(H_{\omega_2}^N,\tau_{\ST}^N,B^{V[G]}\cap N: B\in \UB^V,\NS_{\omega_1}^N).
\] 

Let $\phi(a)$ be a $\Sigma_1$-formula for $\tau_{\NS_{\omega_1},\UB^V}$ in parameter 
$a\in H_{\omega_2}^V$ with $B_1,\dots,B_k\in\UB^V$ the universally Baire predicates  occurring in $\phi$
such that 
\[
(N,\tau_{\ST}^N, B^{V[G]}\cap N: B\in \UB^V, \NS_{\omega_1}^N)\models\phi(j_{0\gamma}(a)).
\]
We must show that
\[
(H_{\omega_2}^V,\tau_{\ST}^V, \UB^V, \NS_{\omega_1}^V)\models\phi(a).
\]

Remark that the iteration $\mathcal{J}$ extends to an iteration 
$\bar{\mathcal{J}}=\bp{\bar{j}_{\alpha,\beta}:\alpha\leq\beta\leq\gamma=(\omega_1)^N}$ 
of $V$ exactly as already done in the proof of Lemma \ref{lem:UBcorr}.

Using this observation, let
$\bar{M}=\bar{j}_{0\gamma}(V)$;
then $\NS_{\omega_1}^{\bar{M}}=\NS_{\omega_1}^N\cap \bar{M}$.

Now let $H$ be $V$-generic for $\Coll(\omega,\eta)$ with $G\in V[H]$
for some $\eta>\delta$ inaccessible in $V[G]$.

By $\maxUB$
$N$ is $\UB^{V[H]}$-correct in $V[H]$:
on the one hand 
\[
D_{\UB^{V[H]}}=\Cod[\bar{D}_{\UB^{V[G]}}^{V[H]}],
\]
on the other hand
\[
N\in \Cod[\bar{D}_{\UB^{V[G]}}]\subseteq \Cod[\bar{D}_{\UB^{V[G]}}^{V[H]}].
\]
In particular for any $B\in \UB_V$, $N$ is $B^{V[H]}$-iterable in $V[H]$.

Therefore in $H_{\omega_1}^{V[H]}$ for any $B\in\mathsf{UB}^V$, the statement 
\begin{quote}
\emph{There exists a $\tau_{\NS_{\omega_1}}\cup\bp{B,B_1,\dots,B_k}$-super-structure 
$\bar{N}$ of
$j_{0\gamma}(V_\delta)$
which is
$\bp{B^{V[H]},B_1^{V[H]},\dots,B_k^{V[H]}}$-iterable and which realizes 
$\phi(j_{0\gamma}(a))$}
\end{quote}
holds true as witnessed by $N$.

The following is a key observation:
\begin{subclaim}
For any $s\in (2^{\omega})^{\bar{M}[H]}$ and $B\in \UB^V$
\[
s\in j_{0\gamma}(B)^{\bar{M}[H]}\text{ if and only if } s\in B^{V[H]}\cap \bar{M}[H].
\]
\end{subclaim}
\begin{proof}
For each $B\in \UB^V$ find in $V$ trees $(T_B,S_B)$ which project to complement in $V[H]$ and such that 
$B=p[T_B]$.
Now since $\bar{j}_{0,\gamma}[T_B]\subseteq \bar{j}_{0,\gamma}(T_B)$ and 
$\bar{j}_{0,\gamma}[S_B]\subseteq \bar{j}_{0,\gamma}(S_B)$, we get that 
\begin{itemize}
\item
$(2^{\omega})^{V[H]}=p[[\bar{j}_{0,\gamma}(T_B)]]\cup p[[\bar{j}_{0,\gamma}(S_B)]]$ (since $(2^{\omega})^{V[H]}$ is already covered by $p[[\bar{j}_{0,\gamma}[T_B]]]\cup p[[\bar{j}_{0,\gamma}[S_B]]]$).
\item
$\emptyset=p[[\bar{j}_{0,\gamma}(T_B)]]\cap p[[\bar{j}_{0,\gamma}(S_B)]]$ by elementarity of $\bar{j}_{0,\gamma}$.
\end{itemize}
Hence $B^{V[H]}$ is also the projection of $\bar{j}_{0,\gamma}(T_B)$ 
and the pair $(\bar{j}_{0,\gamma}(T_B), \bar{j}_{0,\gamma}(S_B))$ projects to complement in $V[H]$.

But this pair belongs to $\bar{M}$, and 
(by elementarity of $\bar{j}_{0\gamma}$)
\[
\bar{M}\models(\bar{j}_{0,\gamma}(T_B), \bar{j}_{0,\gamma}(S_B))\text{ projects to complements for
$\Coll(\omega,\bar{j}_{0,\gamma}(\eta))$.}
\]
Since $\eta\leq \bar{j}_{0,\gamma}(\eta)$ we get that 
\[
\bar{M}\models(\bar{j}_{0,\gamma}(T_B), \bar{j}_{0,\gamma}(S_B))\text{ projects to complements for 
$\Coll(\omega,\eta)$.}
\]

Therefore in 
$V[H]$
$s\in j_{0\gamma}(B)^{\bar{M}[H]}$ if and only if $s\in p[[\bar{j}_{0,\gamma}(T_B)]^{V[H]}]\cap M[H]$
if and only if $s\in p[[T_B]^{V[H]}]\cap \bar{M}[H]$ if and only if $s\in B^{V[H]}\cap \bar{M}[H]$.
\end{proof}

This shows that
\[
(\bar{M}[H],\tau_{\UB^V}^{\bar{M}[H]})\sqsubseteq (V[H],\tau_{\UB^V}^{V[H]}).
\]

Moreover $H_{\omega_1}^{\bar{M}[H]}$ and  $H_{\omega_1}^{V[H]}$
both realize the theory $T_{\UB^V}$ of $H_{\omega_1}^V$ in this language:
on the one hand 
\[
(H_{\omega_1}^V,\tau_{\UB^V}^V)\prec (H_{\omega_1}^{\bar{M}},\tau_{\UB^V}^{\bar{M}})
\prec (H_{\omega_1}^{\bar{M}[H]},\tau_{\UB^V}^{\bar{M}[H]})
\]
(the leftmost $\prec$ holds since $j_{0,\gamma}:V\to \bar{M}$ is elementary, 
the rightmost $\prec$ holds since $\bar{M}$ models $\maxUB$);
on the other hand
\[
(H_{\omega_1}^V,\tau_{\UB^V}^V)\prec (H_{\omega_1}^{V[H]},\tau_{\UB^V}^{V[H]})
\]
(applying  $\maxUB$ in $V$).

Since $T_{\UB^V}$ is model complete, we get that
$H_{\omega_1}^{\bar{M}[H]}$ is an elementary $\tau_{\UB^V}$-substructure of $H_{\omega_1}^{V[H]}$;
therefore $H_{\omega_1}^{\bar{M}[H]}$ models
\begin{quote}
\emph{There exists a $\tau_{\NS_{\omega_1},B,B_1,\dots,B_k}$-super-structure $\bar{N}$ of
$j_{0\gamma}(V_\delta)$
which is \\
$\bp{\bar{j}_{0\gamma}(B)^{\bar{M}[H]},\bar{j}_{0\gamma}(B_1)^{\bar{M}[H]},\dots,\bar{j}_{0\gamma}(B_k)^{\bar{M}[H]}}$-iterable and which realizes 
$\phi(j_{0\gamma}(a))$.}
\end{quote}

By homogeneity of $\Coll(\omega,\eta)$, in $\bar{M}$ we get that
any condition in $\Coll(\omega,\eta)$ forces:
\begin{quote}
\emph{There exists a $\tau_{\NS_{\omega_1},B,B_1,\dots,B_k}$-super-structure $\bar{N}$ of
$j_{0\gamma}(V_\delta)$
which is \\
$\bp{\bar{j}_{0\gamma}(B)^{\bar{M}[\dot{H}]},\bar{j}_{0\gamma}(B_1)^{\bar{M}[\dot{H}]},\dots,\bar{j}_{0\gamma}(B_k)^{\bar{M}[\dot{H}]}}$-iterable and which realizes 
$\phi(j_{0\gamma}(a))$.}
\end{quote}
By elementarity of $\bar{j}_{0\gamma}$ we get that in $V$ it holds that:
\begin{quote}
There exists an $\eta>\delta$ such that
any condition in $\Coll(\omega,\eta)$ forces:
\begin{quote}
\emph{``There exists a countable super structure $\bar{N}$ of
$V_\delta$ with respect to $\tau_{\NS_{\omega_1},\bp{B,B_1,\dots,B_k}}$
which is $\bp{B^{V[\dot{H}]},B_1^{V[\dot{H}]},\dots,B_k^{V[\dot{H}]}}$-iterable 
and which realizes $\phi(a)$''}
\end{quote}
\end{quote}

This procedure can be repeated for any $B\in\UB^V$, 
showing that $\phi(a)$ is honestly consistent in $V$.

By Schindler and Asper\'o characterization of $(*)$ we obtain that $\phi(a)$ holds in $H_{\omega_2}^V$.
\end{proof}

\item[(\ref{thm:char(*)-2})
implies (\ref{thm:char(*)-3})]
Our assumptions grants that the set 
\[
D_{\UB}=
\bp{M\in H_{\omega_1}^V: M\text{ is $\UB^{V}$-correct}}
\]
is coded by a universally Baire set $\bar{D}_\UB$ in $V$. 
Moreover we also get that whenever $G$ is $V$-generic for 
$\Coll(\omega,\delta)$,
the lift $\bar{D}_{\UB}^{V[G]}$ of $\bar{D}_\UB$ to $V[G]$ codes
\[
D_{\UB^{V[G]}}^{V[G]}=\bp{M\in H_{\omega_1}^{V[G]}: M\text{ is $\UB^{V[G]}$-correct}}.
\]

By (\ref{thm:char(*)-2}) we get that $V_\delta\in D_{\NS_{\omega_1},\UB^{V[G]}}^{V[G]}$.

By Fact \ref{fac:densityUBcorrect}
\[
(H_{\omega_1}^V,\tau_{\mathsf{ST}}^V,\UB^V)\models \text{ for all iterable $M$ 
there exists an $\UB$-correct structure 
$\bar{M}\geq M$}.
\]
Again since 
\[
(H_{\omega_1}^V,\tau_{\mathsf{ST}}^V,\UB^V)\prec (H_{\omega_1}^{V[G]},\tau_{\mathsf{ST}}^{V[G]},\UB^V),
\]
and the latter is first order expressible in the predicate $\bar{D}_\UB\in \UB^V$, we get that
\[
(H_{\omega_1}^{V[G]},\tau_{\mathsf{ST}}^{V[G]},\UB^V)\models \text{ for all iterable $M$ 
there exists an $\UB^{V[G]}$-correct structure 
$\bar{M}\geq M$}.
\]
So
let $N\leq V_\delta$ be in $V[G]$ an $\UB^{V[G]}$-correct structure with
$V_\delta\in H_{\omega_1}^N$.

Let $\mathcal{J}=\bp{j_{\alpha\beta}:\,\alpha\leq\beta\leq\gamma=\omega_1^N}\in H_{\omega_2}^N$ be an iteration witnessing $N\leq V_\delta$.

Now for any $A\in \pow{\omega_1}^V$ and $B\in\UB^V$
\[
(H_{\omega_2}^{N},\tau_{\mathsf{ST}}^{N},\NS_{\gamma}^{N},B^{V[G]}\cap N: B\in\UB^V)
\]
models

\begin{quote}
\emph{There exists
an $(\NS_{\omega_1},\UB^{V[G]})$-ec structure $M$ with $B^{V[G]}\cap N\in\UB_M$ and an iteration 
$\bar{\mathcal{J}}=\bp{\bar{j}_{\alpha\beta}:\,\alpha\leq\beta\leq\gamma}$ of $M$ such that
$\bar{j}_{0\gamma}(A)=j_{0\gamma}(A)$}.
\end{quote}
This statement is witnessed exactly by $V_\delta$ in the place of $M$ (since $B=B^{V[G]}\cap V_\delta\in\UB^V$ and 
$\UB^{V[G]}_{V_\delta}=\bp{B^{V[G]}:\, B\in\UB^V}$),
and $\mathcal{J}$ in the place of $\bar{\mathcal{J}}$.

Since $V_\delta$ is $(\NS_{\omega_1},\UB^{V[G]})$-ec in $V[G]$ we get that
$j_{0\gamma}\restriction H_{\omega_2}^V$ is $\Sigma_1$-elementary between $H_{\omega_2}^V$ and $H_{\omega_2}^N$
for $\tau_{\NS_{\omega_1},\UB^V}$.

Hence
\[
(H_{\omega_2}^{V},\tau_{\mathsf{ST}}^{V},\NS_{\gamma}^{V},\UB^V)
\]
models 
\begin{quote}
\emph{There exists 
an $(\NS_{\omega_1}^V,\UB^{V})$-ec structure $M$ with $B\in\UB_M$ and an iteration 
$\bar{\mathcal{J}}=\bp{\bar{j}_{\alpha\beta}:\,\alpha\leq\beta\leq(\omega_1)^V}$ of $M$ such that
$\bar{j}_{0\omega_1}(a)=A$ and 
$\NS_{\omega_1}^{\bar{j}_{0\omega_1}(M)}=\NS_{\omega_1}^V\cap \bar{j}_{0\omega_1}(M)$}.
\end{quote}

\item[(\ref{thm:char(*)-3})
implies (\ref{thm:char(*)-1})]
We use again Schindler and Asper\'o characterization of $(*)$.

Assume $\phi(A)$ is honestly consistent for some $\Sigma_1$-property $\phi(x)$ in the language 
$\tau_{\UB,\NS_{\omega_1}}$
and $A\in\pow{\omega_1}^V$.
Let $B_1,\dots,B_k$ be the universally Baire predicates in $\UB$ mentioned in $\phi(x)$.

By (\ref{thm:char(*)-3}) there is in 
$V$ an $(\NS_{\omega_1},\UB)$-ec $M$ with $B_1,\dots,B_k\in \UB_M$ 
and $a\in \pow{\omega_1}^M$,
and an iteration $\mathcal{J}=\bp{j_{\alpha\beta}:\,\alpha\leq\beta\leq\omega_1}$ of $M$ such that
$j_{0\omega_1}(a)=A$ and $\NS_{\omega_1}^{j_{0\omega_1}(M)}=\NS_{\omega_1}^V\cap j_{0\omega_1}(M)$.

Let $G$ be $V$-generic for $\Coll(\omega,\delta)$. 
Find $N\in V[G]$ such that $N\models \phi(A)$, $N$ is 
$B_1^{V[G]},\dots,B_k^{V[G]}$-iterable in 
$V[G]$ and 
$\NS_{\omega_1}^{N}\cap V=\NS_{\omega_1}^V$ (this $N$ exists by the honest consistency of $\phi(x)$).

Notice that $\mathcal{J}\in V_\delta\subseteq N$ witnesses that $M\geq N$ as well.

Let $\bar{N}\leq N$ in $V[G]$ be a $\UB^{V[G]}$-correct structure with $B_{\UB_V}\in \UB_{\bar{N}}$
($\bar{N}$ exists by Fact \ref{fac:densityUBcorrect} applied in $V[G]$ to $N$ and $B_{\UB_V}$). 
Let $\mathcal{K}=\bp{k_{\alpha\beta}:\alpha\leq\beta\leq\bar{\gamma}=\omega_1^{\bar{N}}}\in\bar{N}$ be an
iteration witnessing that $\bar{N}\leq N$.

Remark that $H_{\omega_2}^{\bar{N}}\models \phi(k_{0\bar{\gamma}}(A))$, since $\Sigma_1$-properties are
upward absolute and $k_{0\bar{\gamma}}(N)$ is a 
$\tau_{\NS_{\omega_1}}\cup\bp{B_1,\dots,B_k}$-substructure of $H_{\omega_2}^{\bar{N}}$.

Also $\bp{B^{V[G]}:B\in\UB_V}\subseteq \UB_{\bar{N}}$ entail that
$B_{\UB_M}^{V[G]}\in \UB_{\bar{N}}$.

Letting 
\[
\bar{\mathcal{J}}=\bp{\bar{j}_{\alpha\beta}:\alpha\leq\beta\leq\bar{\gamma}}=k_{0\bar{\gamma}}(\mathcal{J}),
\]
we get that $\bar{j}_{0\bar{\gamma}}(a)=k_{0\gamma}(j_{0\bar{\gamma}}(a))=k_{0\gamma}(A)$, and
$\bar{\mathcal{J}}$ is such that $B_j^{V[G]}\in \UB_{\bar{N}}$ for all $j=1,\dots,k$ since
$B_{\UB_M}^{V[G]}$ in $\UB_{\bar{N}}$. 

Since $M$ is $(\NS_{\omega_1},\UB^{V[G]})$-ec in $V[G]$ by $\maxUB$, we get that
$\bar{j}_{0\bar{\gamma}}$ defines a $\Sigma_1$-elementary embedding of
\[
(H_{\omega_2}^M,\tau_{\UB_M,\NS_{\omega_1}}^M)
\]
into 
\[
(H_{\omega_2}^{\bar{N}},\tau_{\UB_M,\NS_{\omega_1}}^{\bar{N}}).
\]

Hence 
\[
(H_{\omega_2}^M,\tau_{\UB_M,\NS_{\omega_1}}^M)\models\phi(a).
\] 

This gives that
\[
(H_{\omega_2}^{M_{\omega_1}},\tau_{\UB_M,\NS_{\omega_1}}^{M_{\omega_1}})\models\phi(A)
\] 
(since $j_{0\omega_1}(a)=A$), 
and therefore that
\[
(H_{\omega_2}^V,\tau_{\UB_M,\NS_{\omega_1}}^V)\models\phi(A),
\]
since $M_{\omega_1}$ is a substructure of $H_{\omega_2}^V$ for $\tau_{\UB_M,\NS_{\omega_1}}$.

\end{description}
\end{proof}

\section{Some questions and comments}

\subsection*{Do we really need $\maxUB$ to establish Thm. \ref{thm:mainthm1}?}

It is not at all clear whether the chain 
of equivalences for $(*)$-$\UB$ given in Thm. \ref{Thm:mainthm-1ter} could be proved without appealing to $\maxUB$. 
What we can for sure say is that the equivalence between forcibility and consistency as given by items \ref{Thm:mainthm-1C} and \ref{Thm:mainthm-1B} of Thm.
\ref{Thm:mainthm-1ter}
holds for the signature $\tau_{\omega_1}$ and its $\Pi_2$-sentences
$\psi$.

More precisely: 
\begin{Theorem}
Consider any $\tau_{\omega_1}$-theory $S$
extending
\[
\ZFC_\ST+\omega_1\emph{ is the first uncountable cardinal $+$ there are class many supercompact cardinals}
\]
and which \emph{is preserved by any forcing} (e.g. $S$ itself
or $S+T_\forall$ for
 any $T$ extending $S$). Then the Kaiser hull of
 $S$ is equivalently given by those 
$\Pi_2$-sentences $\psi$ for $\tau_{\omega_1}$ satysfying items \ref{Thm:mainthm-1C} or \ref{Thm:mainthm-1B} of Thm.
\ref{Thm:mainthm-1ter}.
\end{Theorem}
\begin{proof}
First assume that $S$ proves that
$\psi^{H_{\omega_2}}$ is forcible; given a model 
$V$ of $S$, by collapsing a supercompact of $V$ to countable one gets some $V[G]$ which
models $S+\maxUB$ and satisfies the same universal sentence for $\tau_{\omega_1}$ as $V$ (by Thm. \ref{thm:PI1invomega2}).
Hence by forcing over $V[G]$ (which is still a model of $S$), we get to some $V[H]$ which models $\psi^{H_{\omega_2}}+\maxUB+S$
and satisfies the same universal sentence for $\tau_{\omega_1}$ as $V[G]$. Hence we get that
$\psi$ is consistent with the universal fragment of any $\tau_{\omega_1}$-completion of $S$.

Now assume $\psi$ is 
consistent with the universal fragment of any completion of $S$: 
Any $\tau_{\omega_1}$-model $V$ of 
$S$ can be extended  (using forcing) to a 
$\tau_{\omega_1}$-model $V[G]$ of $S+\maxUB+\stUB$ which satisfies the same $\tau_{\omega_1}$-universal sentences of $V$ (again by Thm. \ref{thm:PI1invomega2}).
Since $\tau_{\omega_1}\subseteq\sigma_{\lUB,\NS_{\omega_1}}$ 
and any $\tau_{\omega_1}$-model $W$
of $S$ admits a unique extension to 
$\sigma_{\lUB,\NS_{\omega_1}}$-model which interprets correctly the new predicate symbols,
we get that $\psi$ is in the model companion of the $\sigma_{\lUB,\NS_{\omega_1}}$-theory of $V[G]$,
and also that this model companion is the $\sigma_{\lUB,\NS_{\omega_1}}$-theory of $H_{\omega_2}^{V[G]}$.
By the equivalence of \ref{Thm:mainthm-1E} and \ref{Thm:mainthm-1B} of Thm.
\ref{Thm:mainthm-1ter}
we get that $H_{\omega_2}^{V[G]}\models \psi$.

Using a similar argument (and appealing to Lemma \ref{fac:proofthm1-2} for the unique extension of $S$ to $\sigma_{\lUB,\NS_{\omega_1}}$ which inteprets correctly the new predicate symbols) one can also prove that these $\Pi_2$-sentences $\psi$ for $\tau_{\omega_1}$ axiomatize the Kaiser hull of $S$. We leave the details to the reader.
\end{proof}

The above argument is not restricted to $\tau_{\omega_1}$ and $S$, but holds mutatis mutandis for many 
other signatures contained in $\sigma_{\omega,\NS_{\omega_1}}$ and theories extending $\ZFC$ 
with large cardinals; we leave the details to the reader.

Let us also note that for $S$ as above 
$\mathsf{CH}$ cannot be 
$S$-equivalent to a
$\Sigma_1$-sentence for $\tau_{\omega_1}$, because $\CH$ is a statement which can change its truth value across
forcing extensions, while the universal $\tau_{\omega_1}$-sentences maintain the same truth value across all forcing extensions of a model of $T$ by Thm. \ref{thm:PI1invomega2}.

\subsection*{Can we prove model companionship results coupled with generic absoluteness
for the theory of $H_{\aleph_3}$?}
We can also argue that we cannot hope to find a signature 
$\sigma\supseteq \tau_\ST\cup\bp{\omega_1,\omega_2}$ 
such that the universal theory of $V$ in signature $\sigma$
is invariant across forcing extension of $V$. In particular we cannot hope to get a
signature $\sigma$ which makes the theory of $H_{\aleph_3}$ the model companion of the theory of $V$ in this signature and such that
it suffices to use forcing to compute which $\Pi_2$-sentences
fall into this model companion theory of $V$ (as we argued to be the case for the  theory of $H_{\aleph_2}$ in signature $\bp{\in}_{\bar{A}_2}\supseteq \tau_\ST\cup\bp{\omega_1}$).
This observation is due to Boban Veli\v{c}kovi\`c.

\begin{Remark}\label{rem:limbobangeninv}
$\Box_{\omega_2}$ is a $\Sigma_1$-statement for $\tau_{\omega_2}=\tau_{\ST}\cup\bp{\omega_1,\omega_2}$:
\begin{align*}
\exists\bp{C_\alpha:\alpha<\omega_2}&[\\
&\forall \alpha\in\omega_2\, (C_\alpha\text{ is a club subset of }\alpha)\wedge\\
&\wedge\forall\alpha\in\beta\in\omega_2 \,(\alpha\in \lim(C_\beta)\rightarrow C_\alpha=C_\beta\cap\alpha)\wedge\\
&\wedge\forall\alpha\in\omega_2\,( \otp(C_\alpha)\leq\omega_1)\\
&].
\end{align*}

$\Box_{\omega_2}$ is forcible by very nice forcings 
(countably directed and $<\omega_1$-strategically closed), 
and its negation is forcible by $\Coll(\omega_1,<\delta)$ whenever $\delta$ is Mahlo.

In particular the $\Pi_1$-theory for $\tau_{\omega_2}$ of any forcing extension $V[G]$ of $V$ can be 
destroyed in a further forcing extension $V[G][H]$  assuming mild large cardinals.
\end{Remark}
Suppose now we want to find 
$A_3\subseteq F_\in$ so to be able
to extend Thm. \ref{thm:mainthm1} by:
\begin{itemize}
\item
assuming as base theory 
$\ZFC+$\emph{suitable large cardinal axioms}
\item
replacing $H_{\aleph_2}$ with 
$H_{\aleph_3}$ in all statements of the theorem pertaining to $A_3$,
\item requiring that 
$\tau_{\omega_2}\subseteq\bp{\in}_{\bar{A}_3}$.
\end{itemize}
In this case the best we can hope for is to 
replace clause \ref{thm:mainthm1.4} of Thm. \ref{thm:mainthm1} with a weaker 
clause
asserting that we consider just forcing notions which do not change the
universal $\bp{\in}_{\bar{A}_3}$-theory of $H_{\aleph_3}$ (which means restricting our 
attention to a narrow class of forcings).

\bibliographystyle{plain}
	\bibliography{Biblio}

\end{document}